\documentclass{article}
\usepackage{a4wide}
\usepackage{amsthm}
\usepackage{graphicx}
\usepackage{amssymb}
\usepackage{amsmath}
\usepackage{ascmac}
\usepackage{setspace}
\usepackage{float}
\usepackage[dvips,usenames]{color}
\usepackage{colortbl}
\usepackage{algorithm}
\usepackage{algorithmic}
\usepackage{setspace}
\usepackage{authblk}
\newtheorem{Theorem}[equation]{Theorem}
\newtheorem{Corollary}[equation]{Corollary}
\newtheorem{Lemma}[equation]{Lemma}
\newtheorem{Proposition}[equation]{Proposition}

\theoremstyle{definition}
\newtheorem{Definition}[equation]{Definition}

\theoremstyle{remark}
\newtheorem{Remark}[equation]{Remark}
\numberwithin{equation}{section}
\newtheorem{Claim}[equation]{Claim}

\DeclareMathOperator{\re}{re}

\DeclareMathOperator{\ev}{ev}

\DeclareMathOperator{\ad}{ad}

\newcommand{\ve}{\varepsilon}

\newcommand{\relmiddle}[1]{\mathrel{}\middle#1\mathrel{}}

\newcommand{\plim}[1][]{\mathop{\varprojlim}\limits_{#1}}
\allowdisplaybreaks
\begin{document}
\title{Construction of Affine Super Yangian}
\author{Mamoru Ueda}
\date{}
\maketitle
\begin{abstract}
In this paper, we define the affine super Yangian $Y_{\ve_1,\ve_2}(\widehat{\mathfrak{sl}}(m|n))$ with a coproduct structure.
We also obtain an evaluation homomorphism,
that is, an algebra homomorphism from $Y_{\ve_1,\ve_2}(\widehat{\mathfrak{sl}}(m|n))$ to the completion
of the universal enveloping algebra of $\widehat{\mathfrak{gl}}(m|n)$. \footnote[0]{{\bf 2020 Math. Subject
Classification;} 17B37\\
\qquad{\bf Institute;} Research Institute for Mathematical Sciences\\
\qquad{\bf adress;} Research Institute for Mathematical Sciences, Kyoto University, Kyoto 606-8502 JAPAN\\
\qquad Communicated by T. Arakawa. Received
 October 17, 2020. Revised April 14, 2021; May 2, 2021.}
\end{abstract}
\section{Introduction}

Drinfeld (\cite{D1}, \cite{D2})  defined the Yangian of a finite dimensional simple Lie algebra $\mathfrak{g}$ in order to obtain a solution of the Yang-Baxter equation. The Yangian is a quantum group which is the deformation of the current algebra $\mathfrak{g}[z]$. He defined it by three different presentations. One of those presentations is called the Drinfeld presentation whose generators are $\{h_{i,r},x^\pm_{i,r}\mid r\in\mathbb{Z}_{\geq0}\}$,
where $\{h_i,x^\pm_i\}$  are Chevalley generators of $\mathfrak{g}$.
The definition of Yangian as an associative algebra naturally extends to the case that $\mathfrak{g}$ is a symmetrizable Kac-Moody Lie algebra in the Drinfeld presentation.
Defining its quasi-Hopf algebra structure is more involved, but this problem has been settled for affine Kac-Moody Lie algebras in \cite{GNW}, \cite{BL} and \cite{U}.

It is known that the Yangians are closely related to $W$-algebras. It was shown in \cite{RS} that there exist surjective homomorphisms from Yangians of type $A$ to rectangular finite $W$-algebras of type $A$. More generally, Brundan and Kleshchev (\cite{BK}) constructed a surjective homomorphism from a shifted Yangian, a subalgebra of the Yangian of type $A$, to a finite $W$-algebra of type $A$.
Using a geometric realization of the Yangian, Schiffmann and Vasserot (\cite{SV}) have constructed a surjective homomorphism from the Yangian of $\widehat{\mathfrak{gl}}(1)$ to the universal enveloping algebra of the principal $W$-algebra of type $A$, and proved the celebrated AGT conjecture (\cite{Ga}, \cite{BFFR}).

In the case of the Lie superalgebra $\mathfrak{sl}(m|n)$, the corresponding Yangian in the Drinfeld presentation was first introduced by Stukopin (\cite{S}, see also \cite{G}).
The relationship between Yangians and $W$-algebras were also studied in the case of finite Lie superalgbras;
by Briot and Ragoucy \cite{BR} for $\mathfrak{sl}(m|n)$ and
by Peng \cite{P} for $\mathfrak{gl}(1|n)$. In the recent paper \cite{GLPZ},
Gaberdiel, Li, Peng and H. Zhang defined the Yangian $\widehat{\mathfrak{gl}}(1|1)$
for the affine Lie superalgebra $\widehat{\mathfrak{gl}}(1|1)$
and obtained the similar result as \cite{SV} in the super setting.

In this article we define the affine super Yangian $Y_{\ve_1,\ve_2}(\widehat{\mathfrak{sl}}(m|n))$ as a quantum group (=an associative algebra equipped with a coproduct satisfying compatibility conditions) in the Drinfeld presentation. We upgrade the definition of the Yangian associated with $\mathfrak{sl}(m|n)$ of Gow \cite{G} to define the affine super Yangian $Y_{\ve_1,\ve_2}(\widehat{\mathfrak{sl}}(m|n))$ as an associative algebra, see Definition \ref{Def}.
However, to define the coproduct for $Y_{\ve_1,\ve_2}(\widehat{\mathfrak{sl}}(m|n))$, we need to obtain
yet another presentation, that is, a {\em minimalistic presentation}. 
\begin{Theorem}\label{T1}
The affine super Yangian $Y_{\ve_1,\ve_2}(\widehat{\mathfrak{sl}}(m|n))$ is isomorphic to the associative superalgebra over $\mathbb{C}$ generated by $x_{i,r}^{+}, x_{i,r}^{-}, h_{i,r}$ $(0\leq i\leq m+n-1, r = 0,1)$ subject to the defining relations \eqref{eq2.1}-\eqref{eq2.9}.
\end{Theorem}
By Theorem~\ref{T1}, the following assertion gives a coproduct $\Delta$  for $Y_{\ve_1,\ve_2}(\widehat{\mathfrak{sl}}(m|n))$
that is compatible with the defining relations \eqref{eq2.1}-\eqref{eq2.9}.
\begin{Theorem}
We can define an algebra homomorphism 
\begin{equation*}
{\Delta} \colon Y_{\ve_1,\ve_2}(\widehat{\mathfrak{sl}}(m|n))\rightarrow
 Y_{\ve_1,\ve_2}(\widehat{\mathfrak{sl}}(m|n)) \widehat{\otimes}  Y_{\ve_1,\ve_2}(\widehat{\mathfrak{sl}}(m|n))
\end{equation*}
that satisfies the coassociativity.
 Here,
$Y_{\ve_1,\ve_2}(\widehat{\mathfrak{sl}}(m|n)) \widehat{\otimes}  Y_{\ve_1,\ve_2}(\widehat{\mathfrak{sl}}(m|n))$ is the degreewise completion of $Y_{\ve_1,\ve_2}(\widehat{\mathfrak{sl}}(m|n)) \otimes  Y_{\ve_1,\ve_2}(\widehat{\mathfrak{sl}}(m|n))$ in the sense of \cite{MNT}.
\end{Theorem}

When $\mathfrak{g}$ is $\mathfrak{sl}(n)$, $Y_h(\mathfrak{sl}(n))$ has an evaluation map $\ev\colon Y_h(\mathfrak{sl}(n))\twoheadrightarrow
U(\mathfrak{sl}(n))$,
which enables us to define actions of $Y_h(\mathfrak{sl}(n))$ on any highest weight representation of $\mathfrak{sl}(n)$.
In \cite{Gu}, Guay showed that the affine Yangian $Y_{\ve_1,\ve_2}(\widehat{\mathfrak{sl}}(n))$ has the evaluation map $\ev\colon Y_{\ve_1,\ve_2}(\widehat{\mathfrak{sl}}(n))\to\widetilde{U}(\widehat{\mathfrak{gl}}(n))$, where $\widetilde{U}(\widehat{\mathfrak{gl}}(n))$ is a completion of the universal enveloping algebra of $\widehat{\mathfrak{gl}}(n)$. The surjectivity of the Guay's evaluation map is
not trivial and was recently shown 
 in \cite{K2}. In the second half of this paper, we construct an evaluation map of the affine super Yangian $Y_{\ve_1,\ve_2}(\widehat{\mathfrak{sl}}(m|n))$ (see Theorem~\ref{thm:main}). 
\begin{Theorem}
Assume $c\hbar = (-m+n) \ve_1$.
Then, there exists a non-trivial algebra homomorphism $\ev \colon Y_{\ve_1,\ve_2}(\widehat{\mathfrak{sl}}(m|n)) \to U(\widehat{\mathfrak{gl}}(m|n))_{{\rm comp},+}$ determined by \eqref{evalu1}-\eqref{evalu4}, where $U(\widehat{\mathfrak{gl}}(m|n))_{{\rm comp},+}$ is a completion of the universal enveloping algebra of $\widehat{\mathfrak{gl}}(m|n)$.
\end{Theorem}
This paper is organized as follows. In Section 2, we recall the definition of the Lie superalgebras $\mathfrak{sl}(m|n)$ and $\widehat{\mathfrak{sl}}(m|n)$. In Section 3, we define the affine super Yangian of type A and give the minimalistic presentation. Note that the Yangian for finite dimensional Lie superalgebra is defined for only type $A$ in the literature. In Section 4, we define its coproduct. Finally, we give the evaluation map for the affine super Yangian in Section 5.
\section*{Acknowledgement}
The author wishes to express his gratitude to his supervisor Tomoyuki Arakawa for suggesting lots of advice to improve this paper. The author is also grateful for the support and encouragement of Junichi Matsuzawa. This work was supported by Iwadare Scholarship and JSPS KAKENHI, Grant-in-Aid for JSPS Fellows, Grant Number JP20J12072.
\section{Preliminaries}
In this section, we recall the definition and presentation of the Lie superalgebra $\mathfrak{\widehat{sl}}(m|n)$ (see \cite{Kac1}). First, we recall the definition of $\mathfrak{sl}(m|n)$ and $\mathfrak{gl}(m|n)$.
\begin{Definition}
Let us set $M_{k,l}(\mathbb{C})$ as the set of $k\times l$ matrices over $\mathbb{C}$. We define the Lie superalgebras $\mathfrak{sl}(m|n)$ and $\mathfrak{gl}(m|n)$ as follows;
\begin{gather*}
\mathfrak{gl}(m|n)=\left\{\begin{pmatrix}
A&B\\
C&D
\end{pmatrix} \relmiddle| A \in M_{m,m}(\mathbb{C}), B \in M_{m,n}(\mathbb{C}), C \in M_{n,m}(\mathbb{C}),\text{ and }D \in M_{n,n}(\mathbb{C})\right\},\\
\mathfrak{sl}(m|n)=\left\{\begin{pmatrix}
A&B\\
C&D
\end{pmatrix}\in\mathfrak{gl}(m|n) \relmiddle| tr(A)-tr(D)=0\right\},
\end{gather*}
where we define $\left[\begin{pmatrix}
A&B\\
C&D
\end{pmatrix},\begin{pmatrix}
E&F\\
G&H
\end{pmatrix}\right]$ as
\begin{gather*}
\left[\begin{pmatrix}
A&B\\
C&D
\end{pmatrix},\begin{pmatrix}
E&F\\
G&H
\end{pmatrix}\right]=\begin{pmatrix}
AE-EA+(BG+FC)&AF+BH-(EB+FD)\\
CE+DG-(GA+HC)&DH-HD+(CF+GB)
\end{pmatrix}.
\end{gather*}
\end{Definition}
As with $\mathfrak{sl}(m)$, $\mathfrak{sl}(m|n)$ has a presentation whose generators are Chevalley generators (see \cite{Sc} and \cite{GL}).
\begin{Proposition}\label{Prop1.1}
We set $p\colon\{1,\cdots,m+n\}\to\{0,1\}$ as
\begin{gather*}
p(i)=\begin{cases}
0&(1\leq i\leq m),\\
1&(m+1\leq i\leq m+n).
\end{cases}
\end{gather*}
Suppose that $m,n\ge2, m\neq n$ and $A=(a_{i,j})_{1\leq i,j\leq m+n-1}$ is an $(m+n-1) \times (m+n-1)$ matrix whose components are
\begin{gather*}
a_{i,j} =
	\begin{cases}
		{(-1)}^{p(i)}+{(-1)}^{p(i+1)}  &\text{if } i=j, \\
	         -{(-1)}^{p(i+1)}&\text{if }j=i+1,\\
	         -{(-1)}^{p(i)}&\text{if }j=i-1,\\
		0  &\text{otherwise.}
	\end{cases}
\end{gather*}
Then, $\mathfrak{sl}(m|n)$ is isomorphic to the Lie superalgebra over $\mathbb{C}$ defined by the generators $\{x^\pm_i,h_i\mid1\leq i\leq m+n-1\}$ and by the relations
\begin{gather*}
[h_i,h_j]=0,\qquad[h_i,x^\pm_j]=\pm a_{i,j}x^\pm_j,\qquad[x^+_i,x^-_j]=\delta_{i,j}h_i,\qquad
{\ad(x^\pm_i)}^{1+\mid a_{i,j}\mid}x^\pm_j=0,\\
[x^\pm_m, x^\pm_m]=0,\qquad[[x^\pm_{m-1},x^\pm_m],[x^\pm_{m+1},x^\pm_m]]=0,
\end{gather*}
where the generators $x^\pm_{m}$ are odd and all other generators are even.

The isomorphism $\Psi$ is given by
\begin{gather*}
\Psi(h_i)={(-1)}^{p(i)}E_{ii}-{(-1)}^{p(i+1)}E_{i+1,i+1},\quad
\Psi(x^+_i)=E_{i,i+1},\quad\Psi(x^-_i)={(-1)}^{p(i)}E_{i+1,i}.
\end{gather*}
\end{Proposition}
Next, we recall the definition of the affinization of $\mathfrak{sl}(m|n)$ and $\mathfrak{gl}(m|n)$ (see \cite{M}). Lie superalgebra $\mathfrak{sl}(m|n)$ has a non-degenerate invariant bilinear form $\kappa:\mathfrak{g}\otimes\mathfrak{g}\to\mathbb{C}$. The bilinear form is uniquely determined up to the scalar multiple, so we fix it.
\begin{Definition}
Suppose that $\mathfrak{g}$ is $\mathfrak{sl}(m|n)$ or $\mathfrak{gl}(m|n)$. Then, we set a Lie superalgebra $\widetilde{\mathfrak{g}}$ as $\mathfrak{g}\otimes\mathbb{C}[t^{\pm1}]\oplus\mathbb{C}c\oplus\mathbb{C}d$ whose commutator relations are following;
\begin{gather*}
[a\otimes t^s, b\otimes t^u]=[a,b]\otimes t^{s+u}+s\delta_{s+u, 0}\kappa(a,b)c,\\
\text{$c$ is a central element of }\widetilde{\mathfrak{g}},\\
[d, a\otimes t^s]=sa\otimes t^{s}.
\end{gather*}
We also set a subalgebra $\widehat{\mathfrak{g}}\subset\widetilde{\mathfrak{g}}$ as $\mathfrak{g}\otimes\mathbb{C}[t^{\pm1}]\oplus\mathbb{C}c$.
\end{Definition}
We have another presentation of $\widehat{\mathfrak{sl}}(m|n)$ (see \cite{Y}).
\begin{Proposition}\label{Prop2}
Suppose that $m,n\geq2,m\neq n$ and $A=(a_{i,j})_{0\leq i,j\leq m+n-1}$ is a $(m+n) \times (m+n)$ matrix whose components are
\begin{gather*}
a_{i,j} =
	\begin{cases}
	{(-1)}^{p(i)}+{(-1)}^{p(i+1)}  &\text{if } i=j, \\
	         -{(-1)}^{p(i+1)}&\text{if }j=i+1,\\
	         -{(-1)}^{p(i)}&\text{if }j=i-1,\\
	        1 &\text{if }(i,j)=(0,m+n-1),(m+n-1,0),\\
		0  &\text{otherwise,}
	\end{cases}
\end{gather*}
Then, $\widetilde{\mathfrak{sl}}(m|n)$ is isomorphic to the Lie superalgebra  over $\mathbb{C}$ defined by the generators $\{x^\pm_i, h_i, d\mid0\leq i\leq m+n-1\}$ and by the relations
\begin{gather}
[d, h_i]=0,\quad
[d, x^+_i]=\begin{cases}
x^+_i&(i=0),\\
0&(\text{otherwise}),
\end{cases}\quad
[d, x^-_i]=\begin{cases}
-x^-_i&(i=0),\\
0&(\text{otherwise}),
\end{cases}\label{eat1}\\
[h_i,h_j]=0,\quad
[h_i,x^\pm_j]=\pm a_{i,j}x^\pm_j,\quad
[x^+_i,x^-_j]=\delta_{i,j}h_i,\quad
{\ad(x^\pm_i)}^{1+\mid a_{i,j}\mid}x^\pm_j=0,\label{eat2}\\
[x^\pm_0,x^\pm_0]=0,\quad[x^\pm_m,x^\pm_m]=0,\label{eat2.5}\\
[[x^\pm_{m-1},x^\pm_m],[x^\pm_{m+1},x^\pm_m]]=0,\quad[[x^\pm_{m+n-1},x^\pm_0],[x^\pm_{1},x^\pm_0]]=0,\label{eat3}
\end{gather}
where the generators $x^\pm_{m}$ and $x^\pm_{0}$ are odd and all other generators are even.

The isomorphism $\Xi$ is given by
\begin{gather*}
\Xi(h_i)=\begin{cases}
-E_{1,1}-E_{m+n,m+n}+c&(i=0),\\
{(-1)}^{p(i)}E_{ii}-{(-1)}^{p(i+1)}E_{i+1,i+1}&(1\leq i\leq m+n-1),
\end{cases}\\
\Xi(x^+_i)=\begin{cases}
E_{m+n,1}\otimes t&(i=0),\\
E_{i,i+1}&(\text{otherwise}),
\end{cases}
\quad\Xi(x^-_i)=\begin{cases}
-E_{1,m+n}\otimes t^{-1}&(i=0),\\
{(-1)}^{p(i)}E_{i+1,i}&(\text{otherwise}).
\end{cases}
\end{gather*}
Moreover, $\widehat{\mathfrak{sl}}(m|n)$ is isomorphic to the Lie superalgebra  over $\mathbb{C}$ defined by the generators $\{x^\pm_i, h_i\mid0\leq i\leq m+n-1\}$ and by the relations \eqref{eat2}-\eqref{eat3}.
\end{Proposition}
Finally, we set some notations.
Let us set $\{\alpha_i\}_{0\leq i\leq m+n-1}$ as a set of simple roots of $\widetilde{\mathfrak{sl}}(m|n)$ and $\delta$ as a positive root $\sum_{0\leq i\leq m+n-1}\limits \alpha_i$. 
Moreover, we set $\Delta$ (resp. $\Delta_+$) as a set of roots (resp. positive roots) of $\widetilde{\mathfrak{sl}}(m|n)$. We denote the parity of $E_{i,j}$ as $p(E_{i,j})$. Obviously, $p(E_{i,j})$ is equal to $p(i)+p(j)$. We also set $\Delta_+^{\re}$ and $\Delta^{\re}$ as $\Delta_+\setminus\mathbb{Z}_{>0}\delta$ and $\Delta\setminus\mathbb{Z}\delta$. We also take an inner product on $\bigoplus_{0\leq i\leq m+n-1}\limits\mathbb{C}\alpha_i$ determined by $(\alpha_i,\alpha_j)=a_{i,j}$. Assume that $\mathfrak{g}=\widetilde{\mathfrak{sl}}(m|n)$ and let $\mathfrak{g}_\alpha$ be the root $\alpha$ space of $\mathfrak{g}$. We set $\{x^{k_\alpha}_{\alpha}\}_{1\leq k_\alpha\leq\text{dim}\mathfrak{g}_\alpha}$ as a basis of $\mathfrak{g}_\alpha$ which satisfies $\kappa(x^{p}_\alpha,x^{q}_{-\alpha})=\delta_{p,q}$ for all $\alpha\in\Delta_+$. We also denote the parity of $x^{k_\alpha}_{\alpha}$ by $p(\alpha)$. Moreover, we sometimes identify $\{0,\cdots, m+n-1\}$ with $\mathbb{Z}/(m+n)\mathbb{Z}$ and denote it by $I$.

\section{The minimalistic presentation of the Affine Super Yangian}
First, we define the affine super Yangian $Y_{\ve_1,\ve_2}(\widehat{\mathfrak{sl}}(m|n))$. This definition is a generalization of  Stukopin's super Yangian (\cite{S}). Let us set $\{x,y\}$ as $xy+yx$.
\begin{Definition}\label{Def}
Suppose that $m, n\geq2$ and $m\neq n$. The affine super Yangian $Y_{\ve_1,\ve_2}(\widehat{\mathfrak{sl}}(m|n))$ is the associative super algebra over $\mathbb{C}$ generated by $x_{i,r}^{+}, x_{i,r}^{-}, h_{i,r}$ $(i \in \{0,1,\cdots,m+n-1\}, r \in \mathbb{Z}_{\geq 0})$ with parameters $\ve_1, \ve_2 \in \mathbb{C}$ subject to the defining relations:
\begin{gather}
	[h_{i,r}, h_{j,s}] = 0, \label{eq1.1}\\
	[x_{i,r}^{+}, x_{j,s}^{-}] = \delta_{i,j} h_{i, r+s}, \label{eq1.2}\\
	[h_{i,0}, x_{j,r}^{\pm}] = \pm a_{i,j} x_{j,r}^{\pm},\label{eq1.3}\\
	[h_{i, r+1}, x_{j, s}^{\pm}] - [h_{i, r}, x_{j, s+1}^{\pm}] 
	= \pm a_{i,j} \dfrac{\varepsilon_1 + \varepsilon_2}{2} \{h_{i, r}, x_{j, s}^{\pm}\} 
	- b_{i,j} \dfrac{\varepsilon_1 - \varepsilon_2}{2} [h_{i, r}, x_{j, s}^{\pm}],\label{eq1.4}\\
	[x_{i, r+1}^{\pm}, x_{j, s}^{\pm}] - [x_{i, r}^{\pm}, x_{j, s+1}^{\pm}] 
	= \pm a_{i,j}\dfrac{\varepsilon_1 + \varepsilon_2}{2} \{x_{i, r}^{\pm}, x_{j, s}^{\pm}\} 
	- b_{i,j} \dfrac{\varepsilon_1 - \varepsilon_2}{2} [x_{i, r}^{\pm}, x_{j, s}^{\pm}],\label{eq1.5}\\
	\sum_{w \in \mathfrak{S}_{1 + |a_{i,j}|}}[x_{i,r_{w(1)}}^{\pm}, [x_{i,r_{w(2)}}^{\pm}, \dots, [x_{i,r_{w(1 + |a_{i,j}|)}}^{\pm}, x_{j,s}^{\pm}]\dots]] = 0\ (i \neq j),\label{eq1.6}\\
	[x^\pm_{i,r},x^\pm_{i,s}]=0\ (i=0, m),\label{eq1.7}\\
	[[x^\pm_{i-1,r},x^\pm_{i,0}],[x^\pm_{i,0},x^\pm_{i+1,s}]]=0\ (i=0, m),\label{eq1.8}
\end{gather}
where\begin{gather*}
a_{i,j} =
	\begin{cases}
	{(-1)}^{p(i)}+{(-1)}^{p(i+1)}  &\text{if } i=j, \\
	         -{(-1)}^{p(i+1)}&\text{if }j=i+1,\\
	         -{(-1)}^{p(i)}&\text{if }j=i-1,\\
	        1 &\text{if }(i,j)=(0,m+n-1),(m+n-1,0),\\
		0  &\text{otherwise,}
	\end{cases}\\
	 b_{i,j}=
	\begin{cases}
	-{(-1)}^{p(i+1)} &\text{if } i=j - 1,\\
		{(-1)}^{p(i)} &\text{if } i=j + 1,\\
	        -1 &\text{if }(i,j)=(0,m+n-1),\\
	        1 &\text{if }(i,j)=(m+n-1,0),\\
		0  &\text{otherwise,}
	\end{cases}
\end{gather*}
and the generators $x^\pm_{m, r}$ and $x^\pm_{0, r}$ are odd and all other generators are even.
\end{Definition}
\begin{Remark}
In this paper, we set $[x,y]$ as $xy-{(-1)}^{p(x)p(y)}yx$ for all homogeneous elements $x,y$. Thus, \eqref{eq1.7} is non-trivial.
\end{Remark}
We also define the affine super Yangian associated with $\widetilde{\mathfrak{sl}}(m|n)$.
\begin{Definition}
Suppose that $m, n\geq2$ and $m\neq n$. We define $Y_{\ve_1,\ve_2}(\widetilde{\mathfrak{sl}}(m|n))$ is the associative super algebra over $\mathbb{C}$ generated by $\{x_{i,r}^{\pm}, h_{i,r},d\mid i \in \{0,1,\cdots,m+n-1\}, r \in \mathbb{Z}_{\geq 0}\}$ with parameters $\ve_1, \ve_2 \in \mathbb{C}$ subject to the defining relations \eqref{eq1.1}-\eqref{eq1.8} and
\begin{gather}
[d,h_{i,r}]=0,\quad[d,x^+_{i,r}]=\begin{cases}
1&\text{ if }i=0,\\
0&\text{ if }i\neq0,
\end{cases}\quad[d,x^-_{i,r}]=\begin{cases}
-1&\text{ if }i=0,\\
0&\text{ if }i\neq0,
\end{cases}\label{eq1.9}
\end{gather}
where the generators $x^\pm_{m, r}$ and $x^\pm_{0, r}$ are odd and all other generators are even.
\end{Definition}
One of the difficulty of Definition~\ref{Def} is that the number of generators is infinite. The rest of this section, we construct a new presentation of the affine super Yangian such that the number of generators are finite. 

Let us set $\widetilde{h}_{i,1} = {h}_{i,1} - \dfrac{\ve_1 + \ve_2}{2} h_{i,0}^2$. By the definition of $\widetilde{h}_{i,1}$, we can rewrite \eqref{eq1.4} as
\begin{equation}
[\widetilde{h}_{i,1}, x_{j,r}^{\pm}] = \pm a_{i,j}\left(x_{j,r+1}^{\pm}-b_{i,j}\dfrac{\varepsilon_1 - \varepsilon_2}{2} x_{j, r}^{\pm}\right).\label{11111}
\end{equation}
By \eqref{11111}, we find that $Y_{\ve_1,\ve_2}(\widehat{\mathfrak{sl}}(m|n))$ is generated by $x_{i,r}^{+}, x_{i,r}^{-}, h_{i,r}$ $(i \in\{0,1,\cdots,m+n-1\}, r = 0,1)$. In fact, by \eqref{11111} and \eqref{eq1.2}, we have the following relations;
\begin{gather}
x^\pm_{i,r+1}=\pm\dfrac{1}{a_{i,i}}[\widetilde{h}_{i,1},x^\pm_{i,r}],\quad h_{i,r+1}=[x^+_{i,r+1},x^-_{i,0}]\ \text{if}\  i\neq m,0,\label{eq1297}\\
x^\pm_{i,r+1}=\pm\dfrac{1}{a_{i+1,i}}[\widetilde{h}_{i+1,1},x^\pm_{i,r}]+b_{i+1,i}\dfrac{\varepsilon_1 - \varepsilon_2}{2} x_{i, r}^{\pm},\quad h_{i,r+1}=[x^+_{i,r+1},x^-_{i,0}]\ \text{if}\  i=m,0\label{eq1298}
\end{gather}
for all $r\geq1$. In the following theorem, we construct the minimalistic presentation of the affine super Yangian $Y_{\ve_1,\ve_2}(\widehat{\mathfrak{sl}}(m|n))$ whose generators are $x_{i,r}^{+}, x_{i,r}^{-}, h_{i,r}$ $(i \in \{0,1,\cdots,m+n-1\}, r = 0,1)$. We remark that we have not checked that the presentation is minimalistic yet. However, we call this presentation ``minimalistic presentation'' since, in the non-super case, the corresponding presentation is called ``minimalistic presentation''.
\begin{Theorem}\label{Mini}
Suppose that $m, n\geq2$ and $m\neq n$. The affine super Yangian $Y_{\ve_1,\ve_2}(\widehat{\mathfrak{sl}}(m|n))$ is isomorphic to the associative super algebra generated by $x_{i,r}^{+}, x_{i,r}^{-}, h_{i,r}$ $(i \in \{0,1,\cdots,m+n-1\}, r = 0,1)$ subject to the defining relations:
\begin{gather}
[h_{i,r}, h_{j,s}] = 0,\label{eq2.1}\\
[x_{i,0}^{+}, x_{j,0}^{-}] = \delta_{i,j} h_{i, 0},\label{eq2.2}\\
[x_{i,1}^{+}, x_{j,0}^{-}] = \delta_{i,j} h_{i, 1} = [x_{i,0}^{+}, x_{j,1}^{-}],\label{eq2.3}\\
[h_{i,0}, x_{j,r}^{\pm}] = \pm a_{i,j} x_{j,r}^{\pm},\label{eq2.4}\\
[\widetilde{h}_{i,1}, x_{j,0}^{\pm}] = \pm a_{i,j}\left(x_{j,1}^{\pm}-b_{i,j}\dfrac{\varepsilon_1 - \varepsilon_2}{2} x_{j, 0}^{\pm}\right),\label{eq2.5}\\
[x_{i, 1}^{\pm}, x_{j, 0}^{\pm}] - [x_{i, 0}^{\pm}, x_{j, 1}^{\pm}] = \pm a_{i,j}\dfrac{\varepsilon_1 + \varepsilon_2}{2} \{x_{i, 0}^{\pm}, x_{j, 0}^{\pm}\} - b_{i,j} \dfrac{\varepsilon_1 - \varepsilon_2}{2} [x_{i, 0}^{\pm}, x_{j, 0}^{\pm}],\label{eq2.6}\\
(\ad x_{i,0}^{\pm})^{1+|a_{i,j}|} (x_{j,0}^{\pm})= 0 \ (i \neq j), \label{eq2.7}\\
[x^\pm_{i,0},x^\pm_{i,0}]=0\ (i=0, m),\label{eq2.8}\\
	[[x^\pm_{i-1,0},x^\pm_{i,0}],[x^\pm_{i,0},x^\pm_{i+1,0}]]=0\ (i=0, m),\label{eq2.9}
\end{gather}
where the generators $x^\pm_{m, r}$ and $x^\pm_{0, r}$ are odd and all other generators are even.
\end{Theorem}
The outline of the proof of Theorem~\ref{Mini} is similar to that of Theorem 2.13 of \cite{GNW}. To simplify the notation, we denote the associtive super algebra defined in Theorem~\ref{Mini} as $\widetilde{Y}_{\ve_1,\ve_2}(\widehat{\mathfrak{sl}}(m|n))$. We construct $x^\pm_{i,r}$ and $h_{i,r}$ as the elements of $\widetilde{Y}_{\ve_1,\ve_2}(\widehat{\mathfrak{sl}}(m|n))$ inductively by \eqref{eq1297} and \eqref{eq1298}. 
Since \eqref{eq2.1}-\eqref{eq2.9} are contained in the defining relations of the affine super Yangian, it is enough to check that the defining relations of the affine super Yangians \eqref{eq1.1}-\eqref{eq1.8} are deduced from \eqref{eq2.1}-\eqref{eq2.9} in $\widetilde{Y}_{\ve_1,\ve_2}(\widehat{\mathfrak{sl}}(m|n))$. The proof of Theorem~\ref{Mini} is divided into eight lemmas, that is, Lemma~\ref{Lemma92}, Lemma~\ref{Lemma29}, Lemma~\ref{Lem2}, Lemma~\ref{Lem45}, Lemma~\ref{Lemma30}, Lemma~\ref{Lemma31}, Lemma~\ref{Lemma32}, and Lemma~\ref{Lemma33}. 

Most of the defining relations \eqref{eq1.1}-\eqref{eq1.8} are obtained in the same way as that of \cite{L} or \cite{GNW}. 
For example, we have the following lemma in a similar way as that of Lemma 2.22 of \cite{GNW}.
\begin{Lemma}\label{Lemma92}
\textup{(1)} The defining relation \eqref{eq1.3} holds for all $i, j\in I$ in $\widetilde{Y}_{\ve_1,\ve_2}(\widehat{\mathfrak{sl}}(m|n))$.

\textup{(2)} For all $i, j\in I$, we obtain
\begin{equation}\label{eq2.1.1}
[\widetilde{h}_{i,1}, x_{j,r}^{\pm}] = \pm a_{i,j}\left(x_{j,r+1}^{\pm}-b_{i,j}\dfrac{\varepsilon_1 - \varepsilon_2}{2} x_{j, r}^{\pm}\right)
\end{equation}
in $\widetilde{Y}_{\ve_1,\ve_2}(\widehat{\mathfrak{sl}}(m|n))$. 
\end{Lemma}
\begin{proof}
We only show the case that $j=0,m$. The other case is proven in the same way as that of Lemma 2.22 of \cite{GNW}. We prove (1), (2) by the induction on $r$. When $r=0$, they are nothing but \eqref{eq2.4} and \eqref{eq2.5}. Suppose that \eqref{eq1.3} and \eqref{eq2.1.1} hold when $r=k$. First, let us show that \eqref{eq1.3} holds when $r=k+1$. By \eqref{eq1298}, we obtain
\begin{gather}
[h_{i,0}, x^\pm_{j,k+1}]=\pm\dfrac{1}{a_{j,j+1}}[h_{i,0},[\widetilde{h}_{j+1,1}, x_{j,k}^{\pm}]]+b_{j,j+1}\dfrac{\varepsilon_1 - \varepsilon_2}{2}[h_{i,0}, x_{j,k}^{\pm}].\label{541}
\end{gather}
By $[h_{i,0}, h_{j,1}]=0$, we find that the first term of the right hand side of \eqref{541} is equal to
\begin{gather*}
\pm\dfrac{1}{a_{j,j+1}}[h_{i,0},[\widetilde{h}_{j+1,1}, x_{j,k}^{\pm}]]=\pm\dfrac{1}{a_{j,j+1}}[\widetilde{h}_{j+1,1}, [h_{i,0},x_{j,k}^{\pm}]].
\end{gather*}
By the induction hypothesis on $r$, we can rewrite the right hand side of \eqref{541} as
\begin{align*}
&\phantom{{}={}}\pm\dfrac{1}{a_{j,j+1}}[\widetilde{h}_{j+1,1}, [h_{i,0},x_{j,k}^{\pm}]]+b_{j,j+1}\dfrac{\varepsilon_1 - \varepsilon_2}{2}[h_{i,0}, x_{j,k}^{\pm}]\nonumber\\
&=\dfrac{a_{i,j}}{a_{j,j+1}}[\widetilde{h}_{j+1,1}, x_{j,k}^{\pm}]\pm a_{i,j}b_{j,j+1}\dfrac{\varepsilon_1 - \varepsilon_2}{2}x_{j,k}^{\pm}\nonumber\\
&=\dfrac{a_{i,j}}{a_{j,j+1}}\big(\pm a_{j,j+1}(x_{j,k+1}^{\pm}-b_{j,j+1}\dfrac{\varepsilon_1 - \varepsilon_2}{2}x_{j,k}^{\pm})\big)\pm a_{i,j}b_{j,j+1}\dfrac{\varepsilon_1 - \varepsilon_2}{2}x_{j,k}^{\pm}\nonumber\\
&=\pm a_{i,j}x_{j,k+1}^{\pm}.
\end{align*}
Thus, we have shown that $[h_{i,0}, x^\pm_{j,k+1}]=\pm a_{i,j}x_{j,k+1}^{\pm}.$

Next, we show that \eqref{eq1.3} holds when $r=k+1$. Since we have already proved that \eqref{eq1.3} holds when $r=k+1$, it is enough to check the relation
\begin{equation*}
[\widetilde{h}_{i,1}, x^\pm_{j,k+1}]=\pm a_{i,j}\left(x^\pm_{j,k+2}-b_{i,j}\dfrac{\varepsilon_1 - \varepsilon_2}{2}x^\pm_{j,k+1}\right).
\end{equation*}
By \eqref{eq1298}, we obtain
\begin{gather}
[\widetilde{h}_{i,1}, x^\pm_{j,k+1}]=\pm\dfrac{1}{a_{j,j+1}}[\widetilde{h}_{i,1}, [\widetilde{h}_{j+1,1}, x_{j,k}^{\pm}]]+b_{j,j+1}\dfrac{\varepsilon_1 - \varepsilon_2}{2}[\widetilde{h}_{i,1}, x_{j,k}^{\pm}].\label{542}
\end{gather}
By $[h_{i,1}, h_{j,1}]=0$, we find that the right hand side of \eqref{542} is equal to
\begin{gather*}
\pm\dfrac{1}{a_{j,j+1}}[\widetilde{h}_{j+1,1},[\widetilde{h}_{i,1}, x_{j,k}^{\pm}]]+b_{j,j+1}\dfrac{\varepsilon_1 - \varepsilon_2}{2}[\widetilde{h}_{i,1}, x_{j,k}^{\pm}].
\end{gather*}
By the induction hypothesis on $r$, we can rewrite the right hand side of \eqref{542} as
\begin{align}
&\quad\pm\dfrac{1}{a_{j,j+1}}[\widetilde{h}_{j+1,1},[\widetilde{h}_{i,1}, x_{j,k}^{\pm}]]+b_{j,j+1}\dfrac{\varepsilon_1 - \varepsilon_2}{2}[\widetilde{h}_{i,1}, x_{j,k}^{\pm}]\nonumber\\
&=\dfrac{a_{i,j}}{a_{j,j+1}}\left([\widetilde{h}_{j+1,1},x_{j,k+1}^{\pm}]-b_{i,j}\dfrac{\varepsilon_1 - \varepsilon_2}{2}[\widetilde{h}_{j+1,1},x_{j,k}^{\pm}]\right)\nonumber\\
&\qquad\qquad\qquad\pm a_{i,j}b_{j,j+1}\dfrac{\varepsilon_1 - \varepsilon_2}{2}\left(x_{j,k+1}^{\pm}-b_{i,j}\dfrac{\varepsilon_1 - \varepsilon_2}{2}x_{j,k}^{\pm}\right).\label{956}
\end{align}
Since $x_{j,k+2}^{\pm}$ is defined by \eqref{eq1298}, we find that the right hand side of \eqref{956} is equal to
\begin{align*}
&\pm a_{i,j}\left(x_{j,k+2}^{\pm}-b_{j,j+1}\dfrac{\varepsilon_1 - \varepsilon_2}{2}x_{j,k+1}^{\pm}\right)\mp a_{i,j}b_{i,j}\dfrac{\varepsilon_1 - \varepsilon_2}{2}\left(x^\pm_{j,k+1}-b_{j,j+1}\dfrac{\varepsilon_1 - \varepsilon_2}{2}x^\pm_{j,k}\right)\\
&\qquad\qquad\qquad\pm a_{i,j}b_{j,j+1}\dfrac{\varepsilon_1 - \varepsilon_2}{2}\left(x_{j,k+1}^{\pm}-b_{i,j}\dfrac{\varepsilon_1 - \varepsilon_2}{2}x_{j,k}^{\pm}\right).
\end{align*}
By direct computation, it is equal to
\begin{align*}
\pm a_{i,j}\left(x^\pm_{j,k+2}-b_{i,j}\dfrac{\varepsilon_1 - \varepsilon_2}{2}x^\pm_{j,k+1}\right).
\end{align*}
This completes the proof.
\end{proof}
Similarly, we also obtain the following lemma in a similar way to the one of \cite{GNW}. 
\begin{Lemma}\label{Lemma29}
\textup{(1)} The relation \eqref{eq1.2} holds in $\widetilde{Y}_{\ve_1,\ve_2}(\widehat{\mathfrak{sl}}(m|n))$ when $i=j$ and $r+s\leq2$.

\textup{(2)} Suppose that $i,j\in I$ and $i\neq j$. Then, the relations \eqref{eq1.2} and \eqref{eq1.5} hold for any $r$ and $s$ in $\widetilde{Y}_{\ve_1,\ve_2}(\widehat{\mathfrak{sl}}(m|n))$. 

\textup{(3)} The relation \eqref{eq1.5} holds in $\widetilde{Y}_{\ve_1,\ve_2}(\widehat{\mathfrak{sl}}(m|n))$ when $i=j$, $(r,s)=(1,0)$.

\textup{(4)} The relation \eqref{eq1.4} holds in $\widetilde{Y}_{\ve_1,\ve_2}(\widehat{\mathfrak{sl}}(m|n))$ when $i=j$, $(r,s)=(1,0)$.

\textup{(5)} For all $i,j\in I$, the relation \eqref{eq1.4} holds in $\widetilde{Y}_{\ve_1,\ve_2}(\widehat{\mathfrak{sl}}(m|n))$ when $(r,s)=(1,0)$.

\textup{(6)} Set $\widetilde{h}_{i,2}$ as $h_{i,2}-h_{i,0}h_{i,1}+\dfrac{1}{3}h_{i,0}^3$. Then, the following equation holds for all $i,j\in I$ in $\widetilde{Y}_{\ve_1,\ve_2}(\widehat{\mathfrak{sl}}(m|n))$;
\begin{align*}
[\widetilde{h}_{i,2},x^\pm_{j,0}]&=\pm a_{i,j}x^\pm_{j,2}\pm\dfrac{1}{12}a_{i,j}^3x^\pm_{j,0}\mp a_{i,j}b_{i,j}\dfrac{\ve_1-\ve_2}{2}(x^\pm_{j,1}-\dfrac{1}{2}x^\pm_jh_i- b_{i,j}\dfrac{\ve_1-\ve_2}{2}x^\pm_j).
\end{align*} 

\textup{(7)} For all $i,j\in I$, the relation \eqref{eq1.6} holds in $\widetilde{Y}_{\ve_1,\ve_2}(\widehat{\mathfrak{sl}}(m|n))$ when
 \begin{enumerate}
  \item $r_1 = \cdots = r_b = 0$, $s\in\mathbb{Z}_{\ge 0}$,
  \item $r_1 = 1$, $r_2 = \cdots = r_b = 0$, $s\in\mathbb{Z}_{\ge 0}$,
  \item $r_1 = 2$, $r_2 = \cdots = r_b = 0$, $s\in\mathbb{Z}_{\ge 0}$,
  \item \textup($b\ge 2$ and\textup) $r_1 = r_2 = 1$, $r_3 =\cdots = r_b = 0$, $s \in\mathbb{Z}_{\ge 0}$. 
\end{enumerate}

\textup{(8)} In $\widetilde{Y}_{\ve_1,\ve_2}(\widehat{\mathfrak{sl}}(m|n))$, we have
\begin{equation*}
[h_{j,1}, x^\pm_{i,1}]=\dfrac{a_{i,j}}{a_{i,i}}[h_{i,1},x^\pm_{i,1}]\pm\dfrac{a_{i,j}}{2}(\{h_{j,0},x^\pm_{i,1}\}-\{h_{i,0},x^\pm_{i,1}\})\mp a_{j,i}m_{j,i}\dfrac{\ve_1-\ve_2}{2}x^\pm_{i,1},
\end{equation*}
for all $i, j\in I$ such that $a_{i,i}\neq0$.

\textup{(9)} For all $i, j\in I$, we have
\begin{equation*}
[h_{i,2}, h_{j,0}]=0
\end{equation*}
in $\widetilde{Y}_{\ve_1,\ve_2}(\widehat{\mathfrak{sl}}(m|n))$. 

\textup{(10)}Suppose that $i, j\in I$ such that $a_{i,i}=2$ and $a_{i,j}=-1$. Then,
\begin{align*}
[h_{i,2}, h_{i,1}]=0
\end{align*}
holds in $\widetilde{Y}_{\ve_1,\ve_2}(\widehat{\mathfrak{sl}}(m|n))$. 
\end{Lemma}
\begin{proof}
We only prove (1)-(5) since the proof of (6) (resp.\ (7), (8), (9)) is same as that of Lemma 2.33 (resp.\ Lemma 2.34, Lemma 2.35, Proposition 2.36).

The proofs of (1) and (2) are the same as those of Lemma 2.22 and Lemma 2.26 in \cite{GNW}. In the case where $i,j\neq 0,m$, the proof of (3) (resp.\ (4) and (5)) is also the same as that of Lemma 2.23 (resp.\ Lemma 2.24 and Lemma 2.28) in \cite{GNW}. We omit it. We only show that (3) holds since (4) and (5) are derived from (3) in a similar way to the one of \cite{GNW}.

Suppose that $i=j=0,m$. First, we show that $[x^+_{i,1},x^+_{i,0}]=[x^+_{i,0},x^+_{i,1}]=0$ holds. Applying $\ad(\widetilde{h}_{i+1,1})$ to \eqref{eq2.8}, we have $\pm a_{i,i+1}[x^\pm_{i,1},x^+_{i,0}]\pm a_{i,i+1}[x^\pm_{i,0},x^\pm_{i,1}]$. Since $[x^\pm_{i,1},x^\pm_{i,0}]$ is equal to $[x^\pm_{i,0},x^\pm_{i,1}]$, we obtain $[x^\pm_{i,1},x^\pm_{i,0}]=[x^\pm_{i,0},x^\pm_{i,1}]=0$. Next, we show that $[x^\pm_{i,2},x^\pm_{i,0}]=[x^\pm_{i,1},x^\pm_{i,1}]=[x^\pm_{i,0},x^\pm_{i,2}]$ holds.
Applying $\ad(\widetilde{h}_{i+1,1})$ to $[x^\pm_{i,1},x^\pm_{i,0}]=[x^\pm_{i,0},x^\pm_{i,1}]=0$, we obtain
\begin{gather}
\pm a_{i,i+1}([x^\pm_{i,2},x^\pm_{i,0}]+[x^\pm_{i,1},x^\pm_{i,1}])=0,\label{er1}\\
\pm a_{i,i+1}([x^\pm_{i,1},x^\pm_{i,1}]+[x^\pm_{i,0},x^\pm_{i,2}])=0.\label{er2}
\end{gather}
In the case where $j=0,m$ and $i=j+1$, we can prove (5) in a similar way to the one of Lemma~2.28 in \cite{GNW}. Then, in the similar discussion to that of Lemma~1.4 in \cite{L}, there exists $\widehat{h}_{i+1,2}$ such that
\begin{align*}
[\widehat{h}_{i+1,2},x^\pm_{i,0}]=\pm a_{i,i+1}x^\pm_{i,2}.
\end{align*}
Applying $\ad(\widehat{h}_{i+1,2})$ to \eqref{eq2.8}, we obtain 
\begin{equation}\label{er3}
\pm a_{i,i+1}([x^\pm_{i,2},x^\pm_{i,0}]+[x^\pm_{i,0},x^\pm_{i,2}])=0.
\end{equation}
Since \eqref{er1}, \eqref{er2}, and \eqref{er3} are linearly independent, we obtain $[x^\pm_{i,2},x^\pm_{i,0}]=[x^\pm_{i,1},x^\pm_{i,1}]=[x^\pm_{i,0},x^\pm_{i,2}]$.
We have proved (3).
\end{proof}
In the case where $a_{i,i}=-2$ and $a_{i,j}=1$, we obtain $[h_{i,2}, h_{i,1}]=0$ by changing the proof of Proposition 2.36 of \cite{GNW} a little.
\begin{Lemma}\label{Lem2}
Suppose that $i, j\in I$ such that $a_{i,i}=-2$ and $a_{i,j}=1$. Then, we obtain
\begin{gather*}
[h_{i,2}, h_{i,1}]=0
\end{gather*}
in $\widetilde{Y}_{\ve_1,\ve_2}(\widehat{\mathfrak{sl}}(m|n))$. 
\end{Lemma}
\begin{proof}
We change $h_{i,r}$, $x^+_{i,r}$, and $x^-_{i,r}$, which are written in the proof of Proposition 2.36 of \cite{GNW}, into $-h_{i,r}$, $-x^+_{i,r}$, and $x^-_{i,r}$. Then, we obtain $[-h_{i,2}, -h_{i,1}]=0$.
\end{proof}
By Lemma~\ref{Lemma29} (10) and Lemma~\ref{Lem2}, we obtain the following lemma in the same way as Proposition 2.39 of \cite{GNW} since we only need the condition that $a_{i,i}\neq0$ and $a_{i,j}\neq0$. We omit the proof.
\begin{Lemma}\label{Lem45}
Suppose that $i, j\in I$ such that $a_{i,i}\neq0$ and $a_{i,j}\neq0$. Then, we have
\begin{equation*}
[h_{j,2}, h_{j,1}]=0
\end{equation*}
in $\widetilde{Y}_{\ve_1,\ve_2}(\widehat{\mathfrak{sl}}(m|n))$. 
\end{Lemma}
Therefore, we know that $[h_{i,2}, h_{i,1}]=0$ holds for all $i\in I$. By using the relation $[h_{i,2}, h_{i,1}]=0$, we obtain the following lemma in a similar way as that of Theorem~1.2 in \cite{L} since the proof of these statements needs only the condition that $a_{i,i}\neq0$.
\begin{Lemma}\label{Lemma30}
\textup{(1)} The relation \eqref{eq1.1} holds in $\widetilde{Y}_{\ve_1,\ve_2}(\widehat{\mathfrak{sl}}(m|n))$ when $i=j\neq0,m$.

\textup{(2)} The relation \eqref{eq1.2} holds in $\widetilde{Y}_{\ve_1,\ve_2}(\widehat{\mathfrak{sl}}(m|n))$ when $i=j\neq0,m$.

\textup{(3)} The relation \eqref{eq1.5} holds in $\widetilde{Y}_{\ve_1,\ve_2}(\widehat{\mathfrak{sl}}(m|n))$ when $i=j\neq0,m$.

\textup{(4)} The relation \eqref{eq1.4} holds in $\widetilde{Y}_{\ve_1,\ve_2}(\widehat{\mathfrak{sl}}(m|n))$ when $i=j\neq0,m$.
\end{Lemma}
Next, we prove the same statement as that of Lemma~\ref{Lemma30} in the case that $i=j=0,m$.
\begin{Lemma}\label{Lemma31}
\textup{(1)} The relation \eqref{eq1.5}  holds in $\widetilde{Y}_{\ve_1,\ve_2}(\widehat{\mathfrak{sl}}(m|n))$ when $i=j=0,m$. In particular, \eqref{eq1.7} holds in $\widetilde{Y}_{\ve_1,\ve_2}(\widehat{\mathfrak{sl}}(m|n))$.

\textup{(2)} The relation \eqref{eq1.2} holds in $\widetilde{Y}_{\ve_1,\ve_2}(\widehat{\mathfrak{sl}}(m|n))$ when $i=j=0,m$.

\textup{(3)} We obtain $[h_{i,r},x^\pm_{i,0}]=0$ when $i=0,m$ in $\widetilde{Y}_{\ve_1,\ve_2}(\widehat{\mathfrak{sl}}(m|n))$.

\textup{(4)} The relation \eqref{eq1.4} holds in $\widetilde{Y}_{\ve_1,\ve_2}(\widehat{\mathfrak{sl}}(m|n))$ when $i=j=0,m$.

\textup{(5)} The relation \eqref{eq1.1} holds in $\widetilde{Y}_{\ve_1,\ve_2}(\widehat{\mathfrak{sl}}(m|n))$ when $i=j=0,m$.
\end{Lemma}
\begin{proof}
\textup{(1)}. It is enough to check the equality $[x^\pm_{i,r}, x^\pm_{i,s}]=0$. We only show that $[x^+_{i,r}, x^+_{i,s}]=0$ holds. We can obtain $[x^-_{i,r}, x^-_{i,s}]=0$ in a similar way. We prove \eqref{eq1.5} holds by the induction on $k=r+s$. When $k=0$, it is nothing but \eqref{eq2.8}. Applying $\ad(\widetilde{h}_{i+1,1})$ to $[x^+_{i,0}, x^+_{i,0}]=0$, we obtain
\begin{equation*}
a_{i,i+1}([x^+_{i,1}, x^+_{i,0}]+[x^+_{i,0}, x^+_{i,1}])=0.
\end{equation*}
Since $[x^+_{i,1}, x^+_{i,0}]=[x^+_{i,0}, x^+_{i,1}]$, we have $[x^+_{i,1}, x^+_{i,0}]=[x^+_{i,0}, x^+_{i,1}]=0$.

Suppose that $[x^+_{i,r}, x^+_{i,s}]=0$ holds for all $r, s$ such that $r+s=k, k+1$. Applying $\ad(\widetilde{h}_{i+1,1})$ to $[x^+_{i,u}, x^+_{i,k+1-u}]=0$, we have
\begin{align}
[\widetilde{h}_{i+1,1}, [x^+_{i,u}, x^+_{i,k+1-u}]]=0.\label{eqt}
\end{align}
By Lemma~\ref{Lemma29} (4) and the induction hypothesis, we have
\begin{equation}
[\widetilde{h}_{i+1,1}, [x^+_{i,u}, x^+_{i,k+1-u}]]=a_{i,i+1}([x^+_{i,u+1}, x^+_{i,k+1-u}]+[x^+_{i,u}, x^+_{i,k+2-u}]).\label{eqs}
\end{equation}
Since $a_{i,i+1}\neq0$, we find the relation
\begin{equation}\label{eqautiongat}
[x^+_{i,u+1}, x^+_{i,k+1-u}]=-[x^+_{i,u}, x^+_{i,k+2-u}]
\end{equation}
by \eqref{eqt} and \eqref{eqs}. In particular, we obtain 
\begin{equation}\label{etn12}
[x^+_{i,u+2}, x^+_{i,k-u}]=[x^+_{i,u}, x^+_{i,k+2-u}]. 
\end{equation}
Applying $\ad(\widetilde{h}_{i+1,2})$ to $[x^+_{i,u}, x^+_{i,k-u}]=0$, we have
\begin{align}
[\widetilde{h}_{i+1,2}, [x^+_{i,u}, x^+_{i,k-u}]]=0\label{eqc}
\end{align}
by the induction hypothesis. By Lemma~\ref{Lemma29} (7), Lemma~\ref{Lem45} and the induction hypothesis, we have
\begin{equation}
[\widetilde{h}_{i+1,2}, [x^+_{i,u}, x^+_{i,k-u}]]=a_{i,i+1}([x^+_{i,u+2}, x^+_{i,k-u}]+[x^+_{i,u}, x^+_{i,k+2-u}]). \label{eqd}
\end{equation}
Since $a_{i,i+1}\neq0$, we obtain the relation
\begin{align}
[x^+_{i,u+2}, x^+_{i,k-u}]=-[x^+_{i,u}, x^+_{i,k+2-u}].\label{eqe}
\end{align}
by \eqref{eqc} and \eqref{eqd}. 
Since \eqref{eqe} and \eqref{etn12} are linearly independent, we have shown that $[x^+_{i,u}, x^+_{i,k+2-u}]=0$ holds. 

\textup{(2)} We prove the statement by the induction on $r+s=k$. When $k=0$, it is nothing but \eqref{eq2.8}. Suppose that $[x^+_{i,r}, x^-_{i,s}]=h_{i,r+s}$ for all $r, s$ such that $r+s\leq k$. Then, we have the following claim.
\begin{Claim}
\textup{(a)}\  For all $r,s$, we obtain
\begin{equation}\label{LABE1}
[h_{i, r+1}, x_{i+1, s}^+] - [h_{i, r}, x_{i+1, s+1}^+] 
	= a_{i,i+1} \dfrac{\varepsilon_1 + \varepsilon_2}{2} \{h_{i, r}, x_{i+1, s}^+\} 
	- b_{i,i+1} \dfrac{\varepsilon_1 - \varepsilon_2}{2} [h_{i, r}, x_{i+1, s}^+].
\end{equation}

\textup{(b)}\ For all $r+s=k-1$, we obtain
\begin{equation}\label{LABE2}
[h_{i, r+1}, x_{i+1, s}^-] - [h_{i, r}, x_{i+1, s+1}^-] 
	= - a_{i,i+1} \dfrac{\varepsilon_1 + \varepsilon_2}{2} \{h_{i, r}, x_{i+1, s}^-\} 
	- b_{i,i+1} \dfrac{\varepsilon_1 - \varepsilon_2}{2} [h_{i, r}, x_{i+1, s}^-].
\end{equation}
\end{Claim}
\begin{proof}
\textup{(a)}\ By the definition of $h_{i,r}$, we have
\begin{align*}
[h_{i, r+1}, x_{i+1, s}^+] - [h_{i, r}, x_{i+1, s+1}^+]=[[x^+_{i,r+1},x^-_{i,0}],x_{i+1, s}^+]-[[x^+_{i,r},x^-_{i,0}],x_{i+1, s+1}^+].
\end{align*}
By the Jacobi identity and Lemma~\ref{Lemma29} (4), we obtain
\begin{align*}
[h_{i, r+1}, x_{i+1, s}^+] - [h_{i, r}, x_{i+1, s+1}^+]=[\{[x^+_{i,r+1},x_{i+1, s}^+]-[x^+_{i,r},x_{i+1, s+1}^+]\},x^-_{i,0}].
\end{align*}
By Lemma~\ref{Lemma29} (4), we have
\begin{align*}
&[h_{i, r+1}, x_{i+1, s}^+] - [h_{i, r}, x_{i+1, s+1}^+]\\
&\qquad\qquad\qquad=[\pm a_{i,i+1}\dfrac{\varepsilon_1 + \varepsilon_2}{2} \{x_{i, r}^+, x_{i+1, s}^+\} - b_{i,i+1} \dfrac{\varepsilon_1 - \varepsilon_2}{2} [x_{i, r}^+, x_{i+1, s}^+],x^-_{i,0}].
\end{align*}
By Lemma~\ref{Lemma29} (4), we obtain
\begin{align*}
[h_{i, r+1}, x_{i+1, s}^+] - [h_{m, r}, x_{m+1, s+1}^+]=\pm a_{i,i+1}\dfrac{\varepsilon_1 + \varepsilon_2}{2} \{h_{i, r}, x_{i+1, s}^+\} - b_{i,i+1} \dfrac{\varepsilon_1 - \varepsilon_2}{2} [h_{i, r}, x_{i+1, s}^+].
\end{align*}
\textup{(b)}\ 
By the assumption that $[x^+_{i,p}, x^-_{i,q}]=h_{i,p+q}$ holds for all $p+q\leq k$, we have
\begin{align*}
[h_{i, r+1}, x_{i+1, s}^-] - [h_{i, r}, x_{i+1, s+1}^-]=[[x^+_{i,r},x^-_{i,1}],x_{i+1, s}^-]-[[x^+_{i,r},x^-_{i,0}],x_{i+1, s+1}^-]
\end{align*}
since $r+1\leq k$. Similar discussion to (a), we have
\begin{align*}
[h_{i, r+1}, x_{i+1, s}^-] - [h_{i, r}, x_{i+1, s+1}^-]=[x^+_{i,r},\{[x^-_{i,1},x_{i+1, s}^-]-[x^-_{i,0},x_{i+1, s+1}^-]\}].
\end{align*}
By Lemma~\ref{Lemma29} (4), we obtain
\begin{align*}
&[h_{i, r+1}, x_{i+1, s}^-] - [h_{i, r}, x_{i+1, s+1}^-]\\
&\qquad\qquad\qquad=[x^+_{i,r},-a_{i,i+1}\dfrac{\varepsilon_1 + \varepsilon_2}{2} \{x_{i, 0}^-, x_{i+1, s}^-\} - b_{i,i+1} \dfrac{\varepsilon_1 - \varepsilon_2}{2} [x_{i, 0}^-, x_{i+1, s}^-]].
\end{align*}
Then, by Lemma~\ref{Lemma29} (4), we have
\begin{align*}
[h_{i, r+1}, x_{i+1, s}^-] - [h_{i, r}, x_{i+1, s+1}^-]=-a_{i,i+1}\dfrac{\varepsilon_1 + \varepsilon_2}{2} \{h_{i, r}, x_{i+1, s}^-\} - b_{i,i+1} \dfrac{\varepsilon_1 - \varepsilon_2}{2} [h_{i, r}, x_{i+1, s}^-].
\end{align*}
\end{proof}
By the similar discussion to Lemma~1.4 in \cite{L}, there exists $\widetilde{h}_{i,k}$ such that
\begin{gather*}
\widetilde{h}_{i,k}=h_{i,k}+\text{polynomial of }\{h_{i,t}\mid0\leq t\leq k-1\},\\
[\widetilde{h}_{i,k}, x^+_{i+1,1}]=a_{i,i+1}x^+_{i+1,k+1},\quad[\widetilde{h}_{i,k}, x^-_{i+1,0}]=-a_{i,i+1}x^-_{i+1,k}.
\end{gather*}
\begin{Claim}
The following equation holds;
\begin{equation}
[\widetilde{h}_{i+1,1},h_{i,k}]=0.\label{eqi}
\end{equation}
\end{Claim}
\begin{proof}
By the assumption that $[x^+_{i,p}, x^-_{i,q}]=h_{i,k}$ holds for all $p+q\leq k$
we have
\begin{align*}
[\widetilde{h}_{i+1,1},h_{i,s}]=[[\widetilde{h}_{i+1,1},x^+_{i,s}],x^-_{i,0}]+[x^+_{i,s},[\widetilde{h}_{i+1,1},x^-_{i,0}]]=0
\end{align*}
for all $s<k$.
Thus, it is enough to show that $[\widetilde{h}_{i,k},h_{i+1,1}]=0$ holds. 
By the definition of $h_{i+1,1}$, we obtain
\begin{align}
[\widetilde{h}_{i,k},h_{i+1,1}]&=[\widetilde{h}_{i,k},[x^+_{i+1,1},x^-_{i+1,0}]]\nonumber\\
&=a_{i,i+1}[x^+_{i+1,k+1},x^-_{i+1,0}]-a_{i,i+1}[x^+_{i+1,1},x^-_{i+1,k}].
\end{align}
By Lemma~\ref{Lemma30}, it is equal to zero.
\end{proof}
Applying $\ad(\widetilde{h}_{i+1,1})$ to $[x^+_{i,r}, x^-_{i,k-r}]=h_{i,k}$, we obtain
\begin{equation}
[\widetilde{h}_{i+1,1},[x^+_{i,r}, x^-_{i,k-r}]]=[\widetilde{h}_{i+1,1},h_{i,k}]\label{eqg}
\end{equation}
by the induction hypothesis.  By Lemma~\ref{Lemma29} (4), we can rewrite \eqref{eqg} as
\begin{equation}
a_{i,i+1}([x^+_{i,r+1}, x^-_{i,k-r}]-[x^+_{i,r}, x^-_{i,k-r+1}])=[\widetilde{h}_{i+1,1},h_{i,k}]=0.\label{equationg}
\end{equation}
It is nothing but the statement.

\textup{(3)} We only show the statement for $+$. The other case is proven in a similar way. By (2), $[h_{i,r},x^+_{i,0}]$ is equal to $[[x^+_{i,r},x^-_{i,0}],x^+_{i,0}]$. By (1) and the Jacobi identity, we have
\begin{equation}
[[x^+_{i,r},x^-_{i,0}],x^+_{i,0}]=[x^+_{i,r},[x^-_{i,0},x^+_{i,0}]].\label{eqj}
\end{equation}
The right hand side of \eqref{eqj} is equal to $[x^+_{i,r},h_{i,0}]$.
By Lemma~\ref{Lemma92} (1), the right hand side is equal to zero since $a_{i,i}=0$.

\textup{(4)} It is enough to check the equality $[h_{i,r},x^\pm_{i,s}]=0$. We only show the statement for $+$. The other case is proven in a similar way. We prove by the induction on $s$. When $s=0$, it is nothing but (3). Suppose that $[h_{i,r},x^+_{i,s}]=0$ holds. Applying $\ad(\widetilde{h}_{i+1,1})$ to $[h_{i,r},x^+_{i,s}]=0$, we find the equality
\begin{equation}
[\widetilde{h}_{i+1,1},[h_{i,r},x^+_{i,s}]]=0\label{equationf}
\end{equation}
by the induction hypothesis. By the proof of (2), we obtain $[\widetilde{h}_{i+1,1},h_{i,n}]=0$. Thus, the right hand side of \eqref{equationf} is equal to $[h_{i,r},[\widetilde{h}_{i+1,1},x^+_{i,s}]]$. By Lemma~\ref{Lemma29} (4), we obtain
\begin{equation}
[h_{i,r},[\widetilde{h}_{i+1,1},x^+_{i,s}]]=a_{i,i+1}[h_{i,r},(x^+_{i,s+1}-\dfrac{\ve_1-\ve_2}{2}b_{i+1,i}x^+_{i,s})]. \label{equationc}
\end{equation}
By the induction hypothesis, the right hand side of \eqref{equationc} is equal to $a_{i,i+1}[h_{i,r},x^+_{i,s+1}]$. Since $a_{i,i+1}\neq0$, we obtain $[h_{i,r},x^+_{i,s+1}]=0$.

\textup{(5)} By (2), $[h_{i,r}, h_{i,s}]$ is equal to $[h_{i,r},[x^+_{i,s},x^-_{i,0}]]$. By the Jacobi identity, we have
\begin{equation*}
[h_{i,r},[x^+_{i,s},x^-_{i,0}]]=[[h_{i,r},x^+_{i,s}],x^-_{i,0}]+[x^+_{i,s},[h_{i,r},x^-_{i,0}]].
\end{equation*}
By (4), the right hand side is equal to zero. We have shown the relation $[h_{i,r}, h_{i,s}]=0$.
\end{proof}
We obtain the relation \eqref{eq1.5} by Lemma~\ref{Lemma29} (2), Lemma~\ref{Lemma30} (3), and Lemma~\ref{Lemma31} (1). We also find that the relation \eqref{eq1.2}  holds by Lemma~\ref{Lemma29} (2), Lemma~\ref{Lemma30} (2), and Lemma~\ref{Lemma31} (2).

In the same way as that of Theorem~1.2 in \cite{L}, we obtain the defining relations \eqref{eq1.4}, \eqref{eq1.1}, and \eqref{eq1.6}. Thus, we omit the proof.
\begin{Lemma}\label{Lemma32}
\textup{(1)} The relations \eqref{eq1.4} and \eqref{eq1.1} hold in $\widetilde{Y}_{\ve_1,\ve_2}(\widehat{\mathfrak{sl}}(m|n))$ when $i\neq j$.

\textup{(2)} The relation \eqref{eq1.6} holds for all $i,j\in I$ in $\widetilde{Y}_{\ve_1,\ve_2}(\widehat{\mathfrak{sl}}(m|n))$.
\end{Lemma}
We remark that the relation \eqref{eq1.1} holds by Lemma~\ref{Lemma30} (1), Lemma~\ref{Lemma31} (5), and Lemma~\ref{Lemma32} (1). We also find that the relation \eqref{eq1.4} holds by Lemma~\ref{Lemma30} (4), Lemma~\ref{Lemma31} (4), and Lemma~\ref{Lemma32} (1).

Now, it is enough to show that \eqref{eq1.7} and \eqref{eq1.8} are deduced from \eqref{eq2.1}-\eqref{eq2.9}. However, we have already obtained \eqref{eq1.7},  since \eqref{eq1.7} is equivalent to \eqref{eq1.5} when $i=j=0,m$.
Thus, to accomplish the proof, we only need to show that \eqref{eq1.8} holds. 
\begin{Lemma}\label{Lemma33}
The relation \eqref{eq1.8} holds for $i=0,m$ in $\widetilde{Y}_{\ve_1,\ve_2}(\widehat{\mathfrak{sl}}(m|n))$.
\end{Lemma}
\begin{proof}
We prove by the induction on $k=r+s$. When $k=0$, it is nothing but \eqref{eq2.9}. Suppose that \eqref{eq2.9} holds for all $r, s$ such that $r+s=k$. Applying $\ad(\widetilde{h}_{i+2,1})$ to $[[x^\pm_{i-1,r},x^\pm_{i,0}],[x^\pm_{i,0},x^\pm_{i+1,s}]]=0$, we obtain
\begin{equation*}
a_{i-2,i-1}[[x^\pm_{i-1,r+1},x^\pm_{i,0}],[x^\pm_{i,0},x^\pm_{i+1,s}]]=0.
\end{equation*}
Similarly, Applying $\ad(\widetilde{h}_{i+2,1})$ to $[[x^\pm_{i-1,r},x^\pm_{i,0}],[x^\pm_{i,0},x^\pm_{i+1,s}]]=0$, we have
\begin{equation*}
a_{i+2,i+1}[[x^\pm_{i-1,r},x^\pm_{i,0}],[x^\pm_{i,0},x^\pm_{i+1,s+1}]]=0.
\end{equation*}
Thus, we have shown that \eqref{eq1.8} holds for all $r, s$ such that $r+s=k+1$. 
\end{proof}
This completes the proof of Theorem~\ref{Mini}.

By Theorem~\ref{Mini}, we also obtain the minimalistic presentation of $Y_{\ve_1,\ve_2}(\widetilde{\mathfrak{sl}}(m|n))$.
\begin{Theorem}\label{Mimi}
Suppose that $m, n\geq2$ and $m\neq n$. Then, $Y_{\ve_1,\ve_2}(\widetilde{\mathfrak{sl}}(m|n))$ is isomorphic to the super algebra generated by $x_{i,r}^{+}, x_{i,r}^{-}, h_{i,r}$ $(i \in \{0,1,\cdots,m+n-1\}, r = 0,1)$ subject to the defining relations \eqref{eq2.1}-\eqref{eq2.9} and
\begin{gather}
[d,h_{i,r}]=0,\quad[d,x^+_{i,r}]=\begin{cases}
x^+_{i,r}&\text{ if }i=0,\\
0&\text{ if }i\neq0,
\end{cases}\quad[d,x^-_{i,r}]=\begin{cases}
-x^-_{i,r}&\text{ if }i=0,\\
0&\text{ if }i\neq0,
\end{cases}\label{eq2.10}
\end{gather}
where the generators $x^\pm_{m, r}$ and $x^\pm_{0, r}$ are odd and all other generators are even.
\end{Theorem}
The relation \eqref{eq1.9} is derived from \eqref{eq2.10} in a similar way to the one of Lemma~\ref{Lemma92}. We omit the proof.
\section{Coproduct for the Affine Super Yangian}
In this section, we define the coproduct for the affine super Yangian $Y_{\ve_1,\ve_2}(\widehat{\mathfrak{sl}}(m|n))$. We recall the definition of standard degreewise completion (see \cite{MNT}).
\begin{Definition}\label{Defin}
Let $A=\bigoplus_{i\in\mathbb{Z}}\limits A(i)$ be a graded algebra. For all $i\in\mathbb{Z}$, we set a topology on $A(i)$ such that for $a\in A(i)$ the set
\begin{equation*}
\{a+\sum_{r>N}\limits A(i-r)\cdot A(r)\mid N\in\mathbb{Z}_{\geq0}\}
\end{equation*}
forms a fundamental system of open neighborhoods of $a$. The standard degreewise completion of $A$ is $\bigoplus_{i\in\mathbb{Z}}\limits\widehat{A}(i)$ where $\widehat{A}(i)$ is the completion of the space $A(i)$. By the definition of $\widehat{A}(i)$, we find that
\begin{equation*}
\widehat{A}=\bigoplus_{i\in\mathbb{Z}}\limits\plim[N] A(i)/\sum_{r>N}\limits A(i-r)\cdot A(r).
\end{equation*}
\end{Definition}
Let us set the degree on $Y_{\ve_1,\ve_2}(\widehat{\mathfrak{sl}}(m|n))$ determined by
\begin{align}
\text{deg}(h_{i,r})=0,\quad\text{deg}(x^+_{i,r})=\begin{cases}
1\text{ if }i=0,\\
0\text{ if }i\neq0,
\end{cases}\quad\text{deg}(x^-_{i,r})=\begin{cases}
-1&\text{ if }i=0,\\
0&\text{ if }i\neq0.
\end{cases}\label{degree}
\end{align}
Then,  $Y_{\ve_1,\ve_2}(\widehat{\mathfrak{sl}}(m|n))$ and $Y_{\ve_1,\ve_2}(\widehat{\mathfrak{sl}}(m|n))^{\otimes 2}$ become the graded algebra.
We define $\widehat{Y}_{\ve_1,\ve_2}(\widehat{\mathfrak{sl}}(m|n))$ (resp. $Y_{\ve_1,\ve_2}(\widehat{\mathfrak{sl}}(m|n)) \widehat{\otimes}  Y_{\ve_1,\ve_2}(\widehat{\mathfrak{sl}}(m|n))$) as the standard degreewise completion of $Y_{\ve_1,\ve_2}(\widehat{\mathfrak{sl}}(m|n))$ (resp. $Y_{\ve_1,\ve_2}(\widehat{\mathfrak{sl}}(m|n))^{\otimes 2}$) in the sense of Definition~\ref{Defin}. 

We prepare some notations.
There exists a homomorphism from $\widetilde{\mathfrak{sl}}(m|n)$ to $Y_{\ve_1,\ve_2}(\widetilde{\mathfrak{sl}}(m|n))$ determined by $\Phi(h_i)=h_{i,0},\ \Phi(x^\pm_i)=x^\pm_{i,0}$, and $\Phi(d)=d$. We sometimes denote $\Phi(x)$ by $x$ in order to simplify the notation. In particular, we denote $\Phi(x^p_\alpha)$ by $x^p_\alpha$ for all $\alpha\in\Delta$. By Theorem~\ref{thm:main}, we note that $\text{dim}(\Phi(\mathfrak{g}_\alpha))=1$ for all $\alpha\in\Delta_{\text{re}}$.
\begin{Theorem}\label{Maim}
The linear map
${\Delta} \colon Y_{\ve_1,\ve_2}(\widehat{\mathfrak{sl}}(m|n))\rightarrow Y_{\ve_1,\ve_2}(\widehat{\mathfrak{sl}}(m|n)) \widehat{\otimes}  Y_{\ve_1,\ve_2}(\widehat{\mathfrak{sl}}(m|n))$
uniquely determined by
\begin{gather}
{\Delta}(h_{i,0})={h_{i,0}}{\otimes}1+1{\otimes}{h_{i,0}},\quad{\Delta}(x^{\pm}_{i,0})={x^{\pm}_{i,0}}{\otimes}1+1{\otimes}x^{\pm}_{i,0},\nonumber\\
{\Delta}(h_{i,1})={h_{i,1}}{\otimes}1+1{\otimes}{h_{i,1}}+(\ve_1+\ve_2){h_{i,0}}{\otimes}{h_{i,0}}-(\ve_1+\ve_2)\sum_{{\alpha}\in{\Delta}_{+}}\sum_{1\leq k_\alpha\leq\text{dim}\mathfrak{g}_\alpha}(\alpha, {\alpha}_i)x^{k_\alpha}_{-\alpha}{\otimes}x^{k_\alpha}_{\alpha}\label{Coproduct}
\end{gather}
is an algebra homomorphism. Moreover, $\Delta$ satisfies the coassociativity. 
\end{Theorem}
The rest of this section is devoted to the proof of Theorem~\ref{Maim}.
The outline of the proof is similar to that of Theorem 4.9 of \cite{GNW}. In \cite{GNW}, the analogy of the Drinfeld $J$ presentation is considered in order to prove the existence of the coproduct for the affine Yangian. We construct elements similar to those constructed in (3.7) of \cite{GNW}
\begin{Definition} 
We set 
\begin{gather*}
J(h_i)=h_{i,1}+v_i,\qquad J(x^\pm_i)=x^\pm_{i,1}+w^\pm_{i},
\end{gather*}
where
\begin{gather*}
   v_i =\dfrac{\ve_1+\ve_2}{2} \sum_{\alpha\in\Delta_+}\sum_{\substack{1\leq k_\alpha\leq\text{dim}\mathfrak{g}_\alpha}} (\alpha,\alpha_i) 
  x^{k_\alpha}_{-\alpha} x^{k_\alpha}_{\alpha} - \dfrac{\ve_1+\ve_2}{2} h_i^2,\\
  w^+_i=-\dfrac{\ve_1+\ve_2}{2} \sum_{\alpha\in\Delta_+}\sum_{\substack{1\leq k_\alpha\leq\text{dim}\mathfrak{g}_\alpha}}[x^+_i,x^{k_\alpha}_{-\alpha}]x^{k_\alpha}_{\alpha},\quad w^-_i =\dfrac{\ve_1+\ve_2}{2} \sum_{\alpha\in\Delta_+}\sum_{\substack{1\leq k_\alpha\leq\text{dim}\mathfrak{g}_\alpha}}x^{k_\alpha}_{-\alpha}[x^{k_\alpha}_{\alpha},x^-_i].
\end{gather*}
Then, $J(h_i)$ and $J(x^\pm_i)$ are elements of $\widehat{Y}_{\ve_1,\ve_2}(\widehat{\mathfrak{sl}}(m|n))$.
\end{Definition}
Next, we prove the similar results to Lemma 3.9 and Proposition 3.21 in \cite{GNW}. In fact, they are \eqref{3.3.1}-\eqref{3.3.4} and \eqref{992}. 
We prepare one lemma in order to obtain \eqref{3.3.1}-\eqref{3.3.4} and \eqref{992}. It is an analogy of Proposition 2.4 of \cite{Kac2}.
\begin{Lemma}[\cite{M}, Lemma 18.4.1]\label{Kac-thm}
For all $\alpha,\beta\in\Delta_+$, we obtain
\begin{equation*}
\sum_{1\leq k_\beta\leq\text{dim}\mathfrak{g}_\beta}[x^{k_\beta}_{\beta},z]\otimes x^{k_\beta}_{-\beta}=\sum_{1\leq k_\alpha\leq\text{dim}\mathfrak{g}_\alpha}x^{k_\alpha}_{\alpha}\otimes [z,x^{k_\alpha}_{-\alpha}]
\end{equation*}
if $z\in\mathfrak{g}_{\beta-\alpha}$.
\end{Lemma}
\begin{Lemma}\label{Lemma39}
The following relations hold:
\begin{gather}
[J(h_i), h_j]=0,\label{3.3.1}\\
[J(h_i), x^{\pm}_{j}]=\pm({\alpha}_{i},{\alpha}_{j})J(x^{\pm}_{j})\mp a_{i,j}b_{i,j}\dfrac{\ve_1-\ve_2}{2}x^\pm_{j,0},\label{3.3.2}\\
[J(x^{\pm}_{i}),x^{\pm}_{j}]=[x^{\pm}_{i},J(x^{\pm}_{j})]-\dfrac{\ve_1-\ve_2}{2}b_{i,j}[x^\pm_{i,0},x^\pm_{j,0}],\label{3.3.3}\\
[J(x^{\pm}_{i}),x^{\mp}_{j}]=[x^{\pm}_{i},J(x^{\mp}_{j})]={\delta}_{i,j}J(h_i).\label{3.3.4}
\end{gather}
\end{Lemma}
\begin{proof}
Since $h_{i,1}$ commutes with $h_j$ by \eqref{eq1.1} and $v_i$ commutes with $h_j$ by the definition of $v_i$, we obtain \eqref{3.3.1}. We only show the other relations hold for $+$. In a similar way, we obtain them for $-$. First, we prove \eqref{3.3.2} holds for $+$.
By \eqref{eq2.5}, the left hand side of \eqref{3.3.2} is equal to
\begin{align}
&\phantom{{}={}}[\widetilde{h}_{i,1}+v_i+\dfrac{\ve_1+\ve_2}{2}h_{i,0}^2, x^+_j]\nonumber\\
&=a_{i,j}(x^+_{j,1}-b_{i,j}\dfrac{\ve_1-\ve_2}{2}x^+_{j,0})+\big[\dfrac{\ve_1+\ve_2}{2} \sum_{\alpha,\beta\in\Delta_+}\sum_{\substack{1\leq k_\alpha\leq\text{dim}\mathfrak{g}_\alpha}} (\alpha,\alpha_i)x^{k_\alpha}_{-\alpha}x^{k_\alpha}_{\alpha}, x^+_j\big].\label{eq1432}
\end{align}
By direct computation, the second term of the right hand side of \eqref{eq1432} is equal to
\begin{align}
&\dfrac{\ve_1+\ve_2}{2} \sum_{\alpha\in\Delta_+}\sum_{\substack{1\leq k_\alpha\leq\text{dim}\mathfrak{g}_\alpha}} (\alpha,\alpha_i)x^{k_\alpha}_{-\alpha}[x^{k_\alpha}_{\alpha},x^+_j]\nonumber\\
&\qquad\qquad+\dfrac{\ve_1+\ve_2}{2} \sum_{\alpha\in\Delta_+}\sum_{\substack{1\leq k_\alpha\leq\text{dim}\mathfrak{g}_\alpha}}{(-1)}^{p(\alpha)p(\alpha_j)} (\alpha,\alpha_i)[x^{k_\alpha}_{-\alpha},x^+_j]x^{k_\alpha}_{\alpha}.\label{eq1433}
\end{align}
By Lemma~\ref{Kac-thm}, \eqref{eq1433} is equal to
\begin{align}
&\dfrac{\ve_1+\ve_2}{2} \sum_{\alpha\in\Delta_+}\sum_{\substack{1\leq k_\alpha\leq\text{dim}\mathfrak{g}_\alpha}} (\alpha-\alpha_j,\alpha_i)[x^+_j,x^{k_\alpha}_{-\alpha}]x^{k_\alpha}_{\alpha}\nonumber\\
&\qquad\qquad+\dfrac{\ve_1+\ve_2}{2} \sum_{\alpha\in\Delta_+}\sum_{\substack{1\leq k_\alpha\leq\text{dim}\mathfrak{g}_\alpha}} {(-1)}^{p(\alpha)p(\alpha_j)}(\alpha,\alpha_i)[x^{k_\alpha}_{-\alpha},x^+_j]x^{k_\alpha}_{\alpha}.\label{eq1434}
\end{align}
Since ${(-1)}^{p(\alpha)p(\alpha_j)}[x^{k_\alpha}_{-\alpha},x^+_j]+[x^+_j,x^{k_\alpha}_{-\alpha}]=0$ holds, the sum of the first and second terms of \eqref{eq1434} is equal to $-\dfrac{\ve_1+\ve_2}{2} \displaystyle\sum_{\alpha\in\Delta_+}\sum_{\substack{1\leq k_\alpha\leq\text{dim}\mathfrak{g}_\alpha}}\limits (\alpha_j,\alpha_i)[x^+_j,x^{k_\alpha}_{-\alpha}]x^{k_\alpha}_{\alpha}$. Thus, we obtain
\begin{align*}
[J(h_i), x^+_{j}]&=a_{i,j}(x^+_{j,1}-b_{i,j}\dfrac{\ve_1-\ve_2}{2}x^+_{j,0})-\dfrac{\ve_1+\ve_2}{2} \displaystyle\sum_{\alpha\in\Delta_+}\sum_{\substack{1\leq k_\alpha\leq\text{dim}\mathfrak{g}_\alpha}}\limits (\alpha_j,\alpha_i)[x^+_j,x^{k_\alpha}_{-\alpha}]x^{k_\alpha}_{\alpha}.
\end{align*}
Thus, we have obtained \eqref{3.3.2} for $+$.

Next, we show that \eqref{3.3.3} holds for $+$. By the definition of $J(x^+_i)$, $[J(x^+_i),x^+_j]-[x^+_i,J(x^+_j)]$ is equal to
\begin{align*}
[x^+_{i,1},x^+_{j,0}]-[x^+_{i,0},x^+_{j,1}]+[w^+_i,x^+_j]-[x^+_i,w^+_j].
\end{align*}
By \eqref{eq2.6}, $[x^+_{i,1},x^+_{j,0}]-[x^+_{i,0},x^+_{j,1}]$ is equal to $\dfrac{\ve_1+\ve_2}{2}a_{i,j}\{x^+_{i,0},x^+_{j,0}\}-\dfrac{\ve_1-\ve_2}{2}b_{i,j}[x^+_{i,0},x^+_{j,0}]$. By the definition of $w^+_i$, we obtain
\begin{align}
&\phantom{{}={}}[w^+_i,x^+_j]-[x^+_i,w^+_j]\nonumber\\
&=-\dfrac{\ve_1+\ve_2}{2} \sum_{\alpha\in\Delta_+}\sum_{\substack{1\leq k_\alpha\leq\text{dim}\mathfrak{g}_\alpha}}[x^+_i,x^{k_\alpha}_{-\alpha}][x^{k_\alpha}_{\alpha},x^+_j]\nonumber\\
&\quad-\dfrac{\ve_1+\ve_2}{2} \sum_{\alpha\in\Delta_+}\sum_{\substack{1\leq k_\alpha\leq\text{dim}\mathfrak{g}_\alpha}}{(-1)}^{p(\alpha)p(\alpha_j)}[[x^+_i,x^{k_\alpha}_{-\alpha}],x^+_j]x^{k_\alpha}_{\alpha}\nonumber\\
&\quad+\dfrac{\ve_1+\ve_2}{2} \sum_{\alpha\in\Delta_+}\sum_{\substack{1\leq k_\alpha\leq\text{dim}\mathfrak{g}_\alpha}}[x^+_i,[x^+_j,x^{k_\alpha}_{-\alpha}]]x^{k_\alpha}_{\alpha}\nonumber\\
&\quad+\dfrac{\ve_1+\ve_2}{2} \sum_{\alpha\in\Delta_+}\sum_{\substack{1\leq k_\alpha\leq\text{dim}\mathfrak{g}_\alpha}}{(-1)}^{p(\alpha)p(\alpha_i)+p(\alpha_j)p(\alpha_i)}[x^+_j,x^{k_\alpha}_{-\alpha}][x^+_i,x^{k_\alpha}_{\alpha}].\label{329}
\end{align}
By Lemma~\ref{Kac-thm}, we find the equality
\begin{align}
&\text{the first term of the right hand side of \eqref{329}}\nonumber\\
&\qquad=-\dfrac{\ve_1+\ve_2}{2} \sum_{\alpha\in\Delta_+}\sum_{\substack{1\leq k_\alpha\leq\text{dim}\mathfrak{g}_\alpha}}[x^+_i,[x^+_j,x^{k_\alpha}_{-\alpha}]]x^{k_\alpha}_{\alpha}+\dfrac{\ve_1+\ve_2}{2}[x^+_i, h_j]x^+_j.\label{eqn}
\end{align}
We also find the relation
\begin{align}
&\quad\text{the fourth term of the right hand side of \eqref{329}}\nonumber\\
&=\dfrac{\ve_1+\ve_2}{2} \sum_{\alpha\in\Delta_+}\sum_{\substack{1\leq k_\alpha\leq\text{dim}\mathfrak{g}_\alpha}}{(-1)}^{p(\alpha)p(\alpha_i)+p(\alpha_j)p(\alpha_i)}[x^+_j,[x^{k_\alpha}_{-\alpha},x^+_i]]x^{k_\alpha}_{\alpha}\nonumber\\
&\qquad\qquad\qquad\qquad+\dfrac{\ve_1+\ve_2}{2}{(-1)}^{p(\alpha_i)p(\alpha_j)}[x^+_j,h_i]x^+_i\label{eqo}
\end{align}
by Lemma~\ref{Kac-thm}. Applying \eqref{eqn} and \eqref{eqo} to \eqref{329}, we obtain
\begin{equation*}
[w^+_i,x^+_j]-[x^+_i,w^+_j]=\dfrac{\ve_1+\ve_2}{2}[x^+_i, h_j]x^+_j+\dfrac{\ve_1+\ve_2}{2}{(-1)}^{p(\alpha_i)p(\alpha_j)}[x^+_j,h_i]x^+_i. 
\end{equation*}
Since $m,n\geq2$, there exists no $i,j$ such that $a_{i,j}\neq0$ and $p(\alpha_i)p(\alpha_j)=1$. Thus, we obtain
\begin{gather*}
\dfrac{\ve_1+\ve_2}{2}[x^+_i, h_j]x^+_j+\dfrac{\ve_1+\ve_2}{2}{(-1)}^{p(\alpha_i)p(\alpha_j)}[x^+_j,h_i]x^+_i=-\dfrac{\ve_1+\ve_2}{2}a_{i,j}\{x^+_i,x^+_j\}.
\end{gather*}
Hence, we have obtained
\begin{equation*}
[J(x^+_i), x^+_j]-[x^+_i, J(x^+_j)]=-\dfrac{\ve_1-\ve_2}{2}b_{i,j}[x^+_{i,0},x^+_{j,0}].
\end{equation*}
Finally, we show that $[J(x^+_i),x^-_j]=\delta_{i,j}J(h_i)$ holds. By the definiton of $J(x^+_i)$, $[J(x^+_i),x^-_j]$ is equal to $[x^+_{i,1},x^-_{j,0}]+[w^+_{i},x^-_{j,0}]$. By \eqref{eq1.2}, $[x^+_{i,1},x^-_{j,0}]$ is $\delta_{i,j}h_{i,1}$. By direct computation, we have
\begin{align}
&\phantom{{}={}}[w^+_{i},x^-_{j,0}]\nonumber\\
&=-\dfrac{\ve_1+\ve_2}{2} \sum_{\alpha\in\Delta_+}\sum_{\substack{1\leq k_\alpha\leq\text{dim}\mathfrak{g}_\alpha}}[x^+_i,x^{k_\alpha}_{-\alpha}][x^{k_\alpha}_{\alpha},x^-_j]\nonumber\\
&\quad-\dfrac{\ve_1+\ve_2}{2} \sum_{\alpha\in\Delta_+}\sum_{\substack{1\leq k_\alpha\leq\text{dim}\mathfrak{g}_\alpha}}{(-1)}^{p(\alpha)p(\alpha_j)}[[x^+_i,x^{k_\alpha}_{-\alpha}],x^-_j]x^{k_\alpha}_{\alpha}.\label{eqk}
\end{align}
By Lemma~\ref{Kac-thm}, we have
\begin{align}
&\text{the first term of the right hand side of \eqref{eqk}}\nonumber\\
&\qquad=-\dfrac{\ve_1+\ve_2}{2} \sum_{\alpha\in\Delta_+}\sum_{\substack{1\leq k_\alpha\leq\text{dim}\mathfrak{g}_\alpha}}[x^+_i,[x^-_j,x^{k_\alpha}_{-\alpha}]]x^{k_\alpha}_{\alpha}-\dfrac{\ve_1+\ve_2}{2}\delta_{i,j}h_i^2.\label{eql}
\end{align}
By the Jacobi identity, we find the equality
\begin{align}
[x^+_i,[x^-_j,x^{k_\alpha}_{-\alpha}]]&=-{(-1)}^{p(\alpha)p(\alpha_j)}[x^+_i,[x^{k_\alpha}_{-\alpha},x^-_j]]\nonumber\\
&=-{(-1)}^{p(\alpha)p(\alpha_j)}[[x^+_i,x^{k_\alpha}_{-\alpha}],x^-_j]x^{k_\alpha}_{\alpha}-{(-1)}^{p(\alpha)p(\alpha_j)}{(-1)}^{p(\alpha)p(\alpha_i)}[x^{k_\alpha}_{-\alpha},[x^+_i,x^-_j]]x^{k_\alpha}_{\alpha}.\label{eqm}
\end{align}
Thus, we obtain
\begin{align*}
&\phantom{{}={}}[w^+_{i},x^-_{j,0}]\\
&=\dfrac{\ve_1+\ve_2}{2}\sum_{\alpha\in\Delta_+}\sum_{\substack{1\leq k_\alpha\leq\text{dim}\mathfrak{g}_\alpha}}{(-1)}^{p(\alpha)p(\alpha_j)}[[x^+_i,x^{k_\alpha}_{-\alpha}],x^-_j]x^{k_\alpha}_{\alpha}\\
&\quad+\dfrac{\ve_1+\ve_2}{2}\delta_{i,j}\sum_{\alpha\in\Delta_+}\sum_{\substack{1\leq k_\alpha\leq\text{dim}\mathfrak{g}_\alpha}}{(-1)}^{p(\alpha)p(\alpha_j)}{(-1)}^{p(\alpha)p(\alpha_j)}[x^{k_\alpha}_{-\alpha},[x^+_i,x^-_j]]x^{k_\alpha}_{\alpha}-\dfrac{\ve_1+\ve_2}{2}\delta_{i,j}h_i^2\\
&\quad-\dfrac{\ve_1+\ve_2}{2} \sum_{\alpha\in\Delta_+}\sum_{\substack{1\leq k_\alpha\leq\text{dim}\mathfrak{g}_\alpha}}{(-1)}^{p(\alpha)p(\alpha_j)}[[x^+_i,x^{k_\alpha}_{-\alpha}],x^-_j]x^{k_\alpha}_{\alpha}\\
&=\dfrac{\ve_1+\ve_2}{2}\sum_{\alpha\in\Delta_+}\sum_{\substack{1\leq k_\alpha\leq\text{dim}\mathfrak{g}_\alpha}}\delta_{i,j}(\alpha_i,\alpha)x^{k_\alpha}_{-\alpha}x^{k_\alpha}_{\alpha}-\dfrac{\ve_1+\ve_2}{2}\delta_{i,j}h_i^2,
\end{align*}
where the first equality is due to \eqref{eql} and the second equality is due to \eqref{eqm}. Then, we have shown that $[J(x^+_i),x^-_i]=\delta_{i,j}J(h_i)$. Similarly, we can obtain $[x^+_i, J(x^-_j)]=\delta_{i,j}J(h_i)$.  This completes the proof. 
\end{proof}
By \eqref{3.3.1}-\eqref{3.3.4}, we obtain the following convenient relation.
\begin{Corollary}\label{Cor297}
\textup{(1)}\ When $i\neq j, j\pm1$, $[J(x^\pm_i),x^\pm_j]=0$ holds.

\textup{(2)}\ Suppose that $j<i-1$. We have the following relation;
\begin{align*}
&\ad(J(x^\pm_i))\prod_{i+1\leq k\leq m+n-1}\ad(x^\pm_k)\prod_{0\leq k\leq j-1}\ad(x^\pm_k)(x^\pm_j)\\
&\qquad=\ad(x^\pm_i)\ad(J(x^\pm_{i+1}))\prod_{i+2\leq k\leq m+n-1}\ad(x^\pm_k)\prod_{0\leq k\leq j-1}\ad(x^\pm_k)(x^\pm_j)\\
&\qquad\qquad\qquad-b_{i,i+1}\dfrac{\ve_1-\ve_2}{2}\prod_{i\leq k\leq m+n-1}\ad(x^\pm_k)\prod_{0\leq k\leq j-1}\ad(x^\pm_k)(x^\pm_j).
\end{align*}
\textup{(3)}\ For all $\alpha\in\displaystyle\sum_{1\leq l \leq m+n-1}\limits\mathbb{Z}_{\geq0}\alpha_i$ and $x_{\pm\alpha}\in\mathfrak{g}_{\pm\alpha}$, there exists a number $d_{i,j}^\alpha$ such that
\begin{equation*}
(\alpha_j,\alpha)[J(h_i),x_{\pm\alpha}]-(\alpha_i,\alpha)[J(h_j),x_{\pm\alpha}]=\pm d_{i,j}^\alpha x_{\pm\alpha}.
\end{equation*}
\textup{(4)}\ Suppose that $j<i-1$. We have
\begin{align*}
&\quad[J(h_s),\prod_{i\leq k\leq m+n-1}\ad(x^\pm_k)\prod_{0\leq k\leq j-1}\ad(x^\pm_k)(x^\pm_j)]\\
&=\pm(\alpha_s,\alpha)\prod_{i\leq k\leq m+n-1}\ad(x^\pm_k)\prod_{0\leq k\leq j-1}\ad(x^\pm_k)J(x^\pm_j)\\
&\quad\pm c_2\prod_{i\leq k\leq m+n-1}\ad(x^\pm_k)\prod_{0\leq k\leq j-1}\ad(x^\pm_k)(x^\pm_j),
\end{align*}
where $\alpha=\displaystyle\sum_{i\leq k\leq m+n-1}\limits\alpha_k+\displaystyle\sum_{0\leq k\leq j}\limits \alpha_k$ and $c_2$ is a complex number.
\end{Corollary}
\begin{proof}
We only show the relations for $+$. The other case is proven in a similar way.

\textup{(1)}\ By the definition of the commutator relations of $\widehat{\mathfrak{sl}}(m|n)$, $[x^+_i,x^+_j]=0$ holds when $i\neq j, j\pm1$. There exists an index $p$ such that $a_{i,p}\neq0$ and $a_{j,p}=0$. Appling $\ad(J(h_p))$ to $[x^+_i,x^+_j]=0$, we obtain
\begin{equation*}
a_{i,p}([J(x^+_i),x^+_j]-b_{i,p}\dfrac{\ve_1-\ve_2}{2}[x^+_i,x^+_j])=0
\end{equation*}
by \eqref{3.3.2}. Since $a_{i,p}\neq0$, we have shown that $[J(x^+_i),x^+_j]=0$ holds.

\textup{(2)}\ By (1), the left hand side is equal to
\begin{equation*}
\ad([J(x^+_i),x^+_{i+1}])\prod_{i+2\leq k\leq m+n-1}\ad(x^+_k)\prod_{0\leq k\leq j-1}\ad(x^+_k)(x^+_j).
\end{equation*}
By \eqref{3.3.4}, it is equal to
\begin{align}
&\ad([x^+_i,J(x^+_{i+1})])\prod_{i+2\leq k\leq m+n-1}\ad(x^+_k)\prod_{0\leq k\leq j-1}\ad(x^+_k)(x^+_j)\nonumber\\
&\qquad\qquad-b_{i,i+1}\dfrac{\ve_1-\ve_2}{2}\ad([x^+_i,x^+_{i+1}])\prod_{i+2\leq k\leq m+n-1}\ad(x^+_k)\prod_{0\leq k\leq j-1}\ad(x^+_k)(x^+_j).\label{eqp}
\end{align}
By the Jacobi identity, the first term of \eqref{eqp} is equal to 
\begin{equation*}
\ad(x^+_i)\ad(J(x^+_{i+1}))\prod_{i+2\leq k\leq m+n-1}\ad(x^+_k)\prod_{0\leq k\leq j-1}\ad(x^+_k)(x^+_j).
\end{equation*}
This completes the proof.

\textup{(3)}\ It is enough to assume that $x_{\pm\alpha}=\displaystyle\prod_{s\leq k\leq t-1}\limits\ad(x^\pm_k)x^\pm_t$. By \eqref{3.3.2}, we have
\begin{align*}
[J(h_i),x_{\pm\alpha}]&=\pm\delta(s\geq i+1\geq t)a_{i,i+1}\prod_{s\leq k\leq i}\ad(x^\pm_k)J(x^\pm_{i+1})\prod_{i+2\leq k\leq t-1}\ad(x^\pm_k)x^\pm_t\\
&\quad\pm\delta(s\geq i\geq t)a_{i,i}\prod_{s\leq k\leq i-1}\ad(x^\pm_k)J(x^\pm_{i})\prod_{i+1\leq k\leq t-1}\ad(x^\pm_k)x^\pm_t\\
&\quad\pm\delta(s\geq i-1\geq t)a_{i,i-1}\prod_{s\leq k\leq i-2}\ad(x^\pm_k)J(x^\pm_{i-1})\prod_{i+1\leq k\leq t-1}\ad(x^\pm_k)x^\pm_t\\
&\quad\pm d^1_i(\alpha_i,\alpha)\prod_{s\leq k\leq t-1}\ad(x^\pm_k)x^\pm_t,
\end{align*}
where $d^1_i$ is a complex number. By a discussion similar to the one in the proof of (2), we find that there exists a complex number $d^2_i$ such that the sum of the first three terms is equal to
\begin{align*}
\pm(\alpha_i,\alpha)\prod_{s\leq k\leq t-1}\ad(x^\pm_k)J(x^\pm_t)\pm d^2_i(\alpha_i,\alpha)\prod_{s\leq k\leq t-1}\ad(x^\pm_k)x^\pm_t.
\end{align*}
Then, we obtain
\begin{align*}
(\alpha_j,\alpha)[J(h_i),x_{\pm\alpha}]-(\alpha_k,\alpha)[J(h_j),x_{\pm\alpha}]=\pm\{(\alpha_j,\alpha)(d_i^1+d^2_i)-(\alpha_i,\alpha)(d_j^1+d^2_j)\}x_{\pm\alpha}.
\end{align*}
We complete the proof.

\textup{(4)}\ It is proven in a similar way to the one in the proof of (3).
\end{proof}
Next, in order to obtain \eqref{992}, we prepare ${\{\tau_i\}}_{i\neq0,m}$, which are automorphisms of the affine super Yangian.
Let us set $\{s_i\}_{i\neq0,m}$ as an automorphism of $\Delta$ such that $s_i(\alpha)=\alpha-\dfrac{2(\alpha_i,\alpha)}{(\alpha_i,\alpha_i)}\alpha_i$.
By the definition of $\widehat{\mathfrak{sl}}(m|n)$, we can rewrite $s_i$ explicitly as follows;
\begin{equation*}
s_i(\alpha_j)=\begin{cases}
-\alpha_j&\text{ if }i=j,\\
\alpha_i+\alpha_j&\text{ if }j=i\pm1,\\
\alpha_j&\text{ otherwise.}
\end{cases}
\end{equation*}
It is called a simple reflection. We also define ${\{\tau_i\}}_{i\neq0,m}$ as an operator on the affine super Yangian determined by
\begin{align}
\tau_i(x)=\exp(\ad(x^+_i))\exp\Big(-\ad(x^-_i)\Big)\exp(\ad(x^+_i))x.\label{al32}
\end{align}
By the defining relation \eqref{eq1.6}, $\tau_i$ is well-defined as an operator on the affine super Yangian. The following lemma is well-known (see \cite{Kac2}).
\begin{Lemma}\label{KAC}
\textup{(1)} The action of $\tau_i$ preserves the inner product $\kappa$.

\textup{(2)} For all $\alpha\in\Delta$, $\tau_i(\mathfrak{g}_\alpha)=\mathfrak{g}_{s_i(\alpha)}$.
\end{Lemma}
Then, in a similar way as that of Lemma 3.17 and 3.19 of \cite{GNW}, we can compute the action of $\tau_i$ on $J(h_j)$ and write it explicitly.  
\begin{Lemma}\label{Lem35}
When $i\neq0,m$, we obtain
\begin{equation*}
\tau_i(J(h_j))=J(h_j)-\dfrac{2(\alpha_i,\alpha_j)}{(\alpha_i,\alpha_i)}J(h_i)+a_{i,j}b_{j,i}(\ve_1-\ve_2)h_i.
\end{equation*}
\end{Lemma}
Since $\text{dim}\mathfrak{g}_\alpha=1$ for all $\alpha\in\Delta^{\re}$, we sometimes denote $x^{k_\alpha}_{-\alpha}$ and $x^{k_\alpha}_{\alpha}$ as $x_{-\alpha}$ and $x_\alpha$ for all $\alpha\in\Delta^{\re}_+$.
\begin{Proposition}\label{Prop1}
For $i,j\in I$ and a positive real root $\alpha$, the following equation holds;
\begin{equation}
(\alpha_j,\alpha)[J(h_i),x_{\alpha}]-(\alpha_i,\alpha)[J(h_j),x_{\alpha}]=c_{i,j}^\alpha x_{\alpha},\label{992}
\end{equation}
where $c_{i,j}^\alpha$ is a complex number such that $c_{i,j}^\alpha=-c_{i,j}^{-\alpha}$.
\end{Proposition}
\begin{proof}
We divide the proof into two cases; one is that $\alpha$ is even, the other is that $\alpha$ is odd..

{\bf Case 1,} $\alpha$ is even.

Suppose that $\alpha$ is even. Then, there exists $s\in\mathbb{Z}$ such that $\alpha$ is an element of $\displaystyle\sum_{1\leq l \leq m-1}\limits\mathbb{Z}\alpha_i+s\delta$ or $\displaystyle\sum_{m+1\leq l \leq m+n-1}\limits\mathbb{Z}\alpha_i+s\delta$. We only prove the case where $\alpha\in\displaystyle\sum_{1\leq l \leq m-1}\limits\mathbb{Z}_{\geq0}\alpha_i+\mathbb{Z}_{\geq0}\delta$. The other cases are proven in a similar way.

First, we prove the case where $\alpha=\alpha_k+s\delta$, where $k\neq0,m$.
\begin{Claim}\label{Claim534}
Suppose that $\alpha=\alpha_k+s\delta$ such that $k\neq0,m$.
Then, we have
\begin{gather*}
[J(h_i),x_{\alpha}]=\dfrac{(\alpha_i,\alpha_k)}{(\alpha_k,\alpha_k)}[J(h_k),x_{\alpha}]+d_\alpha x_{\alpha},\qquad
[J(h_i),x_{-\alpha}]=\dfrac{(\alpha_i,\alpha_k)}{(\alpha_k,\alpha_k)}[J(h_k),x_{-\alpha}]-d_\alpha x_{-\alpha},
\end{gather*}
where $d_\alpha$ is a complex number.
\end{Claim}
\begin{proof}
Let us set 
\begin{gather*}
x^1_{\pm\delta}=[x^\pm_k,\displaystyle\prod_{k+1\leq p\leq m+n-1}\limits\ad(x^\pm_p)\displaystyle\prod_{0\leq p\leq k-2}\limits\ad(x^\pm_p)(x^\pm_{k-1})].
\end{gather*}
It is enough to suppose that $x_{\pm\alpha}=\ad(x^1_{\pm\delta})^sx^\pm_k$ since $\ad(x^1_{\pm\delta})^sx^\pm_k$ is nonzero.
By the Jacobi identity, we obtain
\begin{align}
&\quad[(\alpha_k,\alpha)J(h_i)-(\alpha_i,\alpha)J(h_k),\ad(x^1_{\pm\delta})^sx^\pm_k]\nonumber\\
&=\sum_{0\leq t\leq s-1}\ad(x^1_{\pm\delta})^t\ad([(\alpha_k,\alpha)J(h_i)-(\alpha_i,\alpha)J(h_k),x^1_{\pm\delta}])\ad(x^1_{\pm\delta})^{s-1-t}x^\pm_k\nonumber\\
&\qquad\qquad+\ad(x^1_{\pm\delta})^s[(\alpha_k,\alpha)J(h_i)-(\alpha_i,\alpha)J(h_k),x^\pm_k].\label{equd}
\end{align}
By \eqref{3.3.2}, $[(\alpha_k,\alpha)J(h_i)-(\alpha_i,\alpha)J(h_k),x^\pm_k]$ can be written as $\pm f_kx^\pm_k$, where $f_k$ is a complex number. Then, we have
\begin{align*}
&\quad[J(h_i),x^1_{\pm\delta}]\\
&=[x^\pm_k,[(\alpha_k,\alpha)J(h_i)-(\alpha_i,\alpha)J(h_k),\displaystyle\prod_{k+1\leq p\leq m+n-1}\limits\ad(x^\pm_p)\displaystyle\prod_{0\leq p\leq k-2}\limits\ad(x^\pm_p)(x^\pm_{k-1})]]\pm f_kx^1_{\pm\delta}.
\end{align*}
By Corollary~\ref{Cor297} (4), we can rewrite the first term as 
\begin{align*}
&\pm(\alpha_k,\alpha)(\alpha_i,\alpha)[x^\pm_k,\displaystyle\prod_{k+1\leq p\leq m+n-1}\limits\ad(x^\pm_p)\displaystyle\prod_{0\leq p\leq k-2}\limits\ad(x^\pm_p)J(x^\pm_{k-1})]\\
&\quad\mp(\alpha_k,\alpha)(\alpha_i,\alpha)[x^\pm_k,\displaystyle\prod_{k+1\leq p\leq m+n-1}\limits\ad(x^\pm_p)\displaystyle\prod_{0\leq p\leq k-2}\limits\ad(x^\pm_p)J(x^\pm_{k-1})]\pm g_kx^1_{\pm\delta}\\
&=\pm g_kx^1_{\pm\delta},
\end{align*}
where $g_k$ is a complex number.
We have obtained the statement.
\end{proof}
Now, let us consider the case where $\alpha$ is a general even root. Any even root $\alpha=\displaystyle\sum_{0\leq k\leq l}\limits\alpha_{p+k}$ can be written as $\displaystyle\prod_{0\leq k\leq l-1}\limits s_{p+k}(\alpha_{p+l})$ by the explicit presentation of $s_i$. Let us prove that the statement of Proposition~\ref{Prop1} holds by the induction on $l$. When $l=1$, it is nothing but Claim~\ref{Claim534}. Assume that \eqref{992} holds when $l=q$. We set $\alpha$ and $\beta$ as $\displaystyle\prod_{0\leq k\leq q}\limits s_{p+k}(\alpha_{p+q+1})$ and $\displaystyle\prod_{1\leq k\leq q}\limits s_{p+k}(\alpha_{p+q+1})$. Suppose that $x_{\beta}$ is a nonzero element of $\mathfrak{g}_{\beta}$. By Lemma~\ref{KAC}, $\mathfrak{g}_\alpha$ contains a nonzero element $\tau_{s_p}(x_{\beta})$. Thus, we obtain
\begin{align}
&\quad(\alpha_j,\alpha)[J(h_i),\tau_{s_p}(x_{\pm\beta})]-(\alpha_i,\alpha)[J(h_j),\tau_{s_p}(x_{\pm\beta})]\nonumber\\
&=\tau_{s_p}\Big\{(\alpha_j,\alpha)\big[J(h_i)-\dfrac{2(\alpha_i,\alpha_p)}{(\alpha_p,\alpha_p)}J(h_p),x_{\pm\beta}\big]-(\alpha_i,\alpha)\big[J(h_j)-\dfrac{2(\alpha_i,\alpha_p)}{(\alpha_p,\alpha_p)}J(h_p),(x_{\pm\beta})\big]\Big\}\nonumber\\
&\quad\mp\{(\alpha_j,\alpha)a_{p,i}b_{i,p}-(\alpha_i,\alpha)a_{p,j}b_{j,p}\}(\ve_1-\ve_2)x_{\pm \alpha}\label{aligncde}
\end{align}
by Lemma~\ref{Lem35}. Let us suppose that $(\alpha_t,\beta)\neq0$. Then, by the induction hypothesis, we find the relation
\begin{equation}
[J(h_u),x_{\pm\beta}]=\pm\dfrac{(\alpha_u,\beta)}{(\alpha_t,\beta)}[J(h_t),x_{\pm\beta}]\pm c^\beta_{u,t}x_{\pm\beta}.\label{alignc}
\end{equation}
Applying \eqref{alignc} to \eqref{aligncde}, we obtain
\begin{align*}
\big[J(h_i)-\dfrac{2(\alpha_i,\alpha_p)}{(\alpha_p,\alpha_p)}J(h_p),x_{\pm\beta}\big]
&=\pm\Big\{\dfrac{(\alpha_i,\beta)}{(\alpha_t,\beta)}-\dfrac{2(\alpha_i,\alpha_p)}{(\alpha_p,\alpha_p)}\cdot\dfrac{(\alpha_p,\beta)}{(\alpha_t,\beta)}\Big\}([J(h_t),x_{\pm\beta}]+c^\beta_{i,t}x_{\pm\beta})\\
&=\pm\dfrac{\Big(\alpha_i,\beta-\dfrac{2(\alpha_i,\alpha_p)}{(\alpha_p,\alpha_p)}\alpha_p\Big)}{(\alpha_t,\beta)}([J(h_t),x_{\pm\beta}]+c^\beta_{i,t}x_{\pm\beta}).
\end{align*}
By the definition of $s_p$, $\alpha$ is equal to $\beta-\dfrac{2(\alpha_i,\alpha_p)}{(\alpha_p,\alpha_p)}\alpha_p$. Then, we have 
\begin{equation}
[J(h_i)-\dfrac{2(\alpha_i,\alpha_p)}{(\alpha_p,\alpha_p)}J(h_p),x_{\pm\beta}]=\pm\dfrac{(\alpha,\alpha_i)}{(\alpha_t,\beta)}([J(h_t),x_{\pm\beta}]+c^\beta_{i,t}x_{\pm\beta}).\label{equf}
\end{equation}
Similarly, we find the relation
\begin{equation}
[J(h_j)-\dfrac{2(\alpha_j,\alpha_p)}{(\alpha_p,\alpha_p)}J(h_p),x_{\pm\beta}]=\pm\dfrac{(\alpha,\alpha_j)}{(\alpha_t,\beta)}([J(h_t),x_{\pm\beta}]+c^\beta_{j,t}x_{\pm\beta}).\label{eque}
\end{equation}
Appling \eqref{equf} and \eqref{eque} to the right hand side of \eqref{aligncde}, 
\begin{align*}
&\quad(\alpha_j,\alpha)[J(h_i),\tau_{s_p}(x_{\pm\beta})]-(\alpha_i,\alpha)[J(h_j),\tau_{s_p}(x_{\pm\beta})]\nonumber\\
&=\pm\tau_{s_p}\Big\{(\alpha_j,\alpha)\dfrac{(\alpha,\alpha_i)}{(\alpha_t,\beta)}c^\beta_{i,t}x_{\pm\beta}-(\alpha_i,\alpha)\dfrac{(\alpha,\alpha_j)}{(\alpha_t,\beta)}c^\beta_{j,t}x_{\pm\beta}\}\nonumber\\
&\quad\mp\{(\alpha_j,\alpha)a_{p,i}b_{i,p}-(\alpha_i,\alpha)a_{p,j}b_{j,p}\}(\ve_1-\ve_2)x_{\pm \alpha}.
\end{align*}
This completes the proof of the case where $\alpha$ is even.

{\bf Case 2,} $\alpha$ is odd.

Here after, we suppose that $m$ is greater than 3. The other case is proven in a similar way. First, we consider the case where $\alpha=\displaystyle\sum_{1\leq l\leq m-1}\limits\alpha_i+\alpha_m+s\delta$.
\begin{Claim}\label{Claim32}
\textup{(1)} When $i\neq0,1,m,m+1$, $[J(h_i),x_{\pm\alpha}]=\pm c_\alpha^i x_{\pm\alpha}$, where $c_\alpha$ is a complex number.

\textup{(2)} We obtain the following equations;
\begin{gather}
[J(h_0),x_{\pm\alpha}]=\dfrac{(\alpha_0,\alpha)}{(\alpha_1,\alpha)}[J(h_1),x_{\pm\alpha}]\pm d_{0,1}x_{\pm\alpha},\label{eq995}\\
[J(h_m),x_{\pm\alpha}]=\dfrac{(\alpha_m,\alpha)}{(\alpha_1,\alpha)}[J(h_1),x_{\pm\alpha}]\pm d_{m,1}x_{\pm\alpha},\label{eq996}\\
[J(h_{m+1}),x_{\pm\alpha}]=\dfrac{(\alpha_{m+1},\alpha)}{(\alpha_m,\alpha)}[J(h_m),x_{\pm\alpha}]\pm d_{m,m+1}x_{\pm\alpha},\label{eq997}
\end{gather}
where $d_{0,1}$, $d_{m,1}$, and $d_{m,m+1}$ are complex numbers.
\end{Claim}
\begin{proof}
\textup{(1)} When $i\neq0,1,2,m,m+1$, we set $x^2_{\pm\delta}=[x^\pm_1,\displaystyle\prod_{2\leq p\leq m+n-1}\limits\ad(x^\pm_p)(x^\pm_0)]$. It is sufficient to assume that
\begin{equation*}
x_{\pm\alpha}=\ad(x^2_{\pm\delta})^s\displaystyle\sum_{1\leq l\leq m-1}\limits\ad(x^\pm_i)(x^\pm_m)
\end{equation*} 
since the right hand side is nonzero.
In a similar way as that of Claim~\ref{Claim534}, we also have
\begin{gather}
[J(h_i), x^2_{\pm\delta}]=\pm h_\delta x^2_{\pm\delta},\label{equi}\\
[J(h_i), \displaystyle\sum_{1\leq l\leq m-1}\limits\ad(x^\pm_i)(x^\pm_m)]=\pm i_\alpha\displaystyle\sum_{1\leq l\leq m-1}\limits\ad(x^\pm_i)(x^\pm_m),\label{equj}
\end{gather}
where $h_\delta$ and $i_\alpha$ are complex numbers. Thus, we find the equality
\begin{equation*}
[J(h_i),x^{k_\alpha}_{\alpha}]=\pm(sh_\delta+i_\alpha)\ad(x^2_{\pm\delta})^s\displaystyle\sum_{1\leq l\leq m-1}\limits\ad(x^\pm_i)(x^\pm_m)
\end{equation*}
by the Jacobi identity, \eqref{equi}, and \eqref{equj}.
We have proved the statement when $i\neq0,1,2,m,m+1$. When $i=2$, we set $x^3_{\pm\delta}$ as 
\begin{equation*}
[x^\pm_{m+1},\displaystyle\prod_{m+2\leq p\leq m+n-1}\limits\ad(x^\pm_p)\prod_{0\leq p\leq m-1}\limits\ad(x^\pm_p)\ad(x^\pm_m)].
\end{equation*}
It is enough to assume that
\begin{equation*}
x_{\pm\alpha}=\ad(x^3_{\pm\delta})^s\displaystyle\sum_{1\leq l\leq m-1}\limits\ad(x^\pm_i)(x^\pm_m)
\end{equation*} 
since the right hand side is nonzero.
In a similar way as that of Claim~\ref{Claim534}, we also have
\begin{gather*}
[J(h_i), x^3_{\pm\delta}]=\pm j_\delta x^3_{\pm\delta},\label{equh}\\
[J(h_i),\displaystyle\sum_{1\leq l\leq m-1}\limits\ad(x^\pm_i)(x^\pm_m)]=\pm k_\alpha\displaystyle\sum_{1\leq l\leq m-1}\limits\ad(x^\pm_i)(x^\pm_m),\label{equj2}
\end{gather*}
where $j_\delta$ and $k_\alpha$ are complex numbers. Thus, we find the relation
\begin{equation*}
[J(h_i),x_{\alpha}]=\pm(sj_\delta+k_\alpha)\ad(x^2_{\pm\delta})^s\displaystyle\sum_{1\leq l\leq m-1}\limits\ad(x^\pm_i)(x^\pm_m)
\end{equation*}
by the Jacobi identity, \eqref{equh} and \eqref{equj2}.
We have proved the statement when $i=2$.

\textup{(2)} First, we prove that \eqref{eq995} holds. By the definition of $\alpha$, $x_{\pm\alpha}$ can be written as $[x_{\pm\beta},x^\pm_m]$, where $x_{\pm\beta}$ is a nonzero element of $\mathfrak{g}_{\alpha-\alpha_m}$. Since $[J(h_0),x^\pm_m]$ and $[J(h_1),x^\pm_m]$ is equal to zero by \eqref{3.3.2}, we obtain
\begin{gather}
[J(h_0),x_{\pm\alpha}]=[[J(h_0),x_{\pm\beta}],x^\pm_m],\label{equk}\\
[J(h_1),x_{\pm\alpha}]=[[J(h_1),x_{\pm\beta}],x^\pm_m].\label{equl}
\end{gather}
Then, because $\beta$ is even, we have
\begin{equation}
[[J(h_0),x_{\pm\beta}],x^\pm_m]=\dfrac{(\alpha_0,\beta)}{(\alpha_1,\beta)}[[J(h_1),x_{\pm\beta}],x^\pm_m]+\dfrac{(\alpha_0,\beta)}{(\alpha_1,\beta)}[x_{\pm\beta},x^\pm_m]\label{equm}
\end{equation}
by Case~1.
By\eqref{equk}, \eqref{equl}, and \eqref{equm}, we find the equality
\begin{equation*}
[J(h_0),x_{\pm\alpha}]=\dfrac{(\alpha_0,\beta)}{(\alpha_1,\beta)}[J(h_1),x_{\pm\alpha}]+\dfrac{(\alpha_0,\beta)}{(\alpha_1,\beta)}[x_{\pm\beta},x^\pm_m].
\end{equation*}
Thus we have shown that \eqref{eq995} holds. Similarly, we obtain \eqref{eq996} since $[J(h_m),x^+_m]=0$ holds. 

Finally, we prove that \eqref{eq997} holds. We set $x^4_{\pm\delta}=[x^\pm_1,\prod_{2\leq p\leq m+n-1}\limits\ad(x^\pm_p)(x^\pm_0)]$. It is enough to check the relation under the assumption that $x_{\pm\alpha}=\ad(x^4_{\pm\delta})^s\prod_{1\leq p\leq m-1}\limits\ad(x^\pm_p)(x^\pm_m)$ since the right hand side is nonzero. 
Then, we obtain
\begin{align}
&\quad[J(h_m),x_{\pm\alpha}]\nonumber\\
&=\sum_{1\leq t\leq s}\ad(x^4_{\pm\delta})^{t-1}\ad([J(h_m),x^4_{\pm\delta}])\ad(x^4_{\pm\delta})^{s-t}\prod_{1\leq p\leq m-1}\limits\ad(x^\pm_p)(x^\pm_m)\nonumber\\
&\qquad\qquad+[J(h_m),\ad(x^4_{\pm\delta})^s\prod_{1\leq p\leq m-1}\limits\ad(x^\pm_p)(x^\pm_m)]\label{equt}\\
&\quad[J(h_{m+1}),x_{\pm\alpha}]\nonumber\\
&=\sum_{1\leq t\leq s}\ad(x^4_{\pm\delta})^{t-1}\ad([J(h_{m+1}),x^4_{\pm\delta}])\ad(x^4_{\pm\delta})^{s-t}\prod_{1\leq p\leq m-1}\limits\ad(x^\pm_p)(x^\pm_m)\nonumber\\
&\qquad\qquad+[J(h_{m+1}),\ad(x^4_{\pm\delta})^s\prod_{1\leq p\leq m-1}\limits\ad(x^\pm_p)(x^\pm_m)]\label{equu}
\end{align}
by the Jacobi identity.
First, we rewrite the first term of the right hand side of \eqref{equt} and \eqref{equu}.
By the assumption $m$ is greater than 3, $[J(h_m),x^\pm_1]=0$ holds by \eqref{3.3.2}. Then, in a similar way as that of Claim~\ref{Claim534}, we find the equalities
\begin{gather}
[J(h_m),x^4_{\pm\delta}]=\pm t_\delta x^4_{\pm\delta},\label{equp}\\
[J(h_{m+1}),x^4_{\pm\delta}]=\pm u_\delta x^4_{\pm\delta},\label{equq}
\end{gather}
where $t_\delta$ and $u_\delta$ are complex numbers. Then, we obtain
\begin{align}
&\text{the first term of the right hand side of \eqref{equt}}\nonumber\\
&\qquad=\pm t_\delta\ad(x^4_{\pm\delta})^s\prod_{1\leq p\leq m-1}\limits\ad(x^\pm_p)(x^\pm_m),\label{99999}\\
&\text{the first term of the right hand side of \eqref{equu}}\nonumber\\
&\qquad=\pm u_\delta\ad(x^4_{\pm\delta})^s\prod_{1\leq p\leq m-1}\limits\ad(x^\pm_p)(x^\pm_m)\label{99998}
\end{align}
by \eqref{equp} and \eqref{equq}.
Next, we rewrite the second term of the right hand side of \eqref{equt} and \eqref{equu}. By \eqref{3.3.2}, we obtain
\begin{align}
&\text{the second term of the right hand side of \eqref{equt}}\nonumber\\
&\qquad=\ad(x^4_{\pm\delta})^s\prod_{1\leq p\leq m-2}\ad(x^\pm_p)[J(h_m),[x^\pm_{m-1},x^\pm_m]],\label{99997}\\
&\text{the second term of the right hand side of \eqref{equu}}\nonumber\\
&\qquad=\ad(x^4_{\pm\delta})^s\prod_{1\leq p\leq m-2}\ad(x^\pm_p)[J(h_{m+1}),[x^\pm_{m-1},x^\pm_m]].\label{99996}.
\end{align}
By \eqref{3.3.2} and \eqref{3.3.3}, we find that
\begin{align}
&\quad[J(h_m),[x^\pm_{m-1},x^\pm_m]]\nonumber\\
&=\pm a_{m,m-1}[J(x^\pm_{m-1}),x^\pm_m]\mp a_{m,m-1}b_{m,m-1}\dfrac{\ve_1-\ve_2}{2}[x^\pm_{m-1},x^\pm_m]\nonumber\\
&=\pm a_{m,m-1}[x^\pm_{m-1},J(x^\pm_m)]\mp a_{m,m-1}(b_{m-1,m}+b_{m,m-1})\dfrac{\ve_1-\ve_2}{2}[x^\pm_{m-1},x^\pm_m],\label{99995}\\
&\quad[J(h_{m+1}),[x^\pm_{m-1},x^\pm_m]]\nonumber\\
&=\pm a_{m+1,m}[x^\pm_{m-1},J(x^\pm_m)]\mp a_{m+1,m}b_{m,m+1}\dfrac{\ve_1-\ve_2}{2}[x^\pm_{m-1},x^\pm_m].\label{99994}
\end{align}
Since $a_{m,m-1}=(\alpha,\alpha_m)$ and $a_{m+1,m}=(\alpha,\alpha_{m+1})$, by \eqref{99995} and \eqref{99994}, we obtain
\begin{equation}
(\alpha,\alpha_{m+1})[J(h_m),[x^\pm_{m-1},x^\pm_m]]-(\alpha,\alpha_m)[J(h_{m+1}),[x^\pm_{m-1},x^\pm_m]]=\pm u_\alpha[x^\pm_{m-1},x^\pm_m],\label{99993}
\end{equation}
where $u_\alpha$ is a complex number. Thus, we know that
\begin{align}
&\quad(\alpha,\alpha_{m+1})(\text{the second term of the right hand side of \eqref{equt}})\nonumber\\
&\qquad\qquad-(\alpha,\alpha_m)(\text{the second term of the right hand side of \eqref{equu}})\nonumber\\
&=u_\alpha\ad(x^4_{\pm\delta})^s\prod_{1\leq p\leq m-1}\limits\ad(x^\pm_p)(x^\pm_m).\label{99992}
\end{align}
holds.
By \eqref{99999}, \eqref{99998}, and \eqref{99992}, we have
\begin{align*}
&\quad(\alpha,\alpha_{m+1})[J(h_m),x_{\pm\alpha}]-(\alpha,\alpha_m)[J(h_{m+1}),x_{\pm\alpha}]\\
&=\pm(s(\alpha,\alpha_{m+1})t_\delta-s(\alpha,\alpha_m)u_\delta+u_\alpha)\ad(x^4_{\pm\delta})^s\prod_{1\leq p\leq m-1}\limits\ad(x^\pm_p)(x^\pm_m).
\end{align*}
Then, we have obtained \eqref{eq997}. 
\end{proof}
Next, let us consider the case where $\alpha$ is a general odd root. We only show the case where $\alpha\in\alpha_m+\displaystyle\sum_{\substack{1\leq t \leq m+n-1,\\t\neq m}}\limits\mathbb{Z}_{\geq0}\alpha_t+s\delta$. The other case is proven in a similar way. 

Since $\alpha\in\alpha_m+\displaystyle\sum_{\substack{1\leq t \leq m+n-1,\\t\neq m}}\limits\mathbb{Z}_{\geq0}\alpha_t+s\delta$, $\alpha$ can be written as $\displaystyle\prod_{1\leq t \leq p}\limits s_{i_t}(\sum_{1\leq i\leq m}\alpha_i+\alpha_m)$.
Then, we prove the statement by the induction on $p$. When $p=0$, it is nothing but Claim~\ref{Claim32}. Other cases are proven in a similar way as that of Case~1.
\end{proof}
We easily obtain the following corollary.
\begin{Corollary}\label{C1}
The following equations hold;
\begin{gather}
[J(h_i),\widetilde{v_j}]+[\widetilde{v_i},J(h_j)]=0,\label{rel1}\\
[J(h_i),J(h_j)] + [\widetilde{v_i},\widetilde{v_j}]=0,\label{rel2}
\end{gather}
where $\widetilde{v_i}=v_i+\dfrac{\ve_1+\ve_2}{2}{h_i}^2$.
\end{Corollary}
\begin{proof}
First, we show that \eqref{rel1} holds.
Since $\widetilde{v_i}=\dfrac{\ve_1+\ve_2}{2}\sum_{\alpha\in\Delta^{\re}_+}\limits(\alpha_j,\alpha)x^{k_\alpha}_{-\alpha}x^{k_\alpha}_{\alpha}$ holds, we obtain
\begin{align}
&\phantom{{}={}}[J(h_i),\widetilde{v_j}]+[\widetilde{v_i},J(h_j)]\nonumber\\
&=\dfrac{\ve_1+\ve_2}{2}\sum_{\alpha\in\Delta^{\re}_+}(\alpha_j,\alpha)[J(h_i),x_{-\alpha}]x_{\alpha}+\dfrac{\ve_1+\ve_2}{2}\sum_{\alpha\in\Delta^{\re}_+}(\alpha_j,\alpha)x_{-\alpha}[J(h_i),x_{\alpha}]\nonumber\\
&\quad+\dfrac{\ve_1+\ve_2}{2}\sum_{\alpha\in\Delta^{\re}_+}(\alpha_i,\alpha)[x_{-\alpha},J(h_j)]x_{\alpha}+\dfrac{\ve_1+\ve_2}{2}\sum_{\alpha\in\Delta^{\re}_+}(\alpha_i,\alpha)x_{-\alpha}[x_{\alpha},J(h_j)].\label{eqr1}
\end{align}
By Proposition~\ref{Prop1}, there exists $c_{i,j}^{\alpha}\in\mathbb{C}$ such that
\begin{align}
&\quad\dfrac{\ve_1+\ve_2}{2}\sum_{\alpha\in\Delta^{\re}_+}(\alpha_j,\alpha)[J(h_i),x_{-\alpha}]x_{\alpha}+\dfrac{\ve_1+\ve_2}{2}\sum_{\alpha\in\Delta^{\re}_+}(\alpha_i,\alpha)[x_{-\alpha},J(h_j)]x_{\alpha}\nonumber\\
&=-\dfrac{\ve_1+\ve_2}{2}\sum_{\alpha\in\Delta^{\re}_+}c_{i,j}^{\alpha}x_{-\alpha}x_{\alpha}\label{eqr2}
\end{align}
and
\begin{align}
&\dfrac{\ve_1+\ve_2}{2}\sum_{\alpha\in\Delta^{\re}_+}(\alpha_j,\alpha)x_{-\alpha}[J(h_i),x_{\alpha}]+\dfrac{\ve_1+\ve_2}{2}\sum_{\alpha\in\Delta^{\re}_+}(\alpha_i,\alpha)x_{-\alpha}[x_{\alpha},J(h_j)]\nonumber\\
&=\dfrac{\ve_1+\ve_2}{2}\sum_{\alpha\in\Delta^{\re}_+}c_{i,j}^{\alpha}x_{-\alpha}x_{\alpha}.\label{eqr3}
\end{align}
Therefore, applying \eqref{eqr2} and \eqref{eqr3} to \eqref{eqr1}, we have obtained the relation \eqref{rel1}. By the defining relation \eqref{eq1.2}, we find the equality
\begin{equation}
[J(h_i)-\widetilde{v_i},J(h_j)-\widetilde{v_j}]=[h_{i,1},h_{j,1}]=0.\label{eqr4}
\end{equation}
On the other hand, we find the relation
\begin{equation*}
[J(h_i)-\widetilde{v_i},J(h_j)-\widetilde{v_j}]=[J(h_i),J(h_j)]-[\widetilde{v_i},J(h_j)]-[J(h_i),\widetilde{v_j}]+[\widetilde{v_i},\widetilde{v_j}]\label{eqr5}.
\end{equation*}
By \eqref{rel1}, the right hand side of \eqref{eqr5} is equal to the left hand side of \eqref{rel2}. Thus, by \eqref{eqr4}, we have found that \eqref{rel2} holds.
\end{proof}
Now, we are in position to obtain the proof of Theorem~\ref{Maim}. To simplify the notation, we set
$\square(x)$ as $x\otimes 1+1\otimes x$
for all $x\in Y_{\ve_1,\ve_2}(\widehat{\mathfrak{sl}}(m|n))$.
\begin{proof}[Proof of Theorem \ref{Maim}]
It is enough to check that $\Delta$ is compatible with \eqref{eq2.1}-\eqref{eq2.9}, which are the defining relations of the minimalistic presentation of the affine super Yangian. Since the restriction of $\Delta$ to $\widehat{\mathfrak{sl}}(m|n)$ is nothing but the usual coproduct of $\widehat{\mathfrak{sl}}(m|n)$, $\Delta$ is compatible with \eqref{eq2.2}, \eqref{eq2.7}, \eqref{eq2.8}, and \eqref{eq2.9}. We also know that $\Delta$ is compatible with \eqref{eq2.4} since $\Delta(x^\pm_{i,1})$ is defined as
\begin{equation*}
\begin{cases}
\pm\dfrac{1}{a_{i,i}}[\Delta(\widetilde{h}_{i,1}),\Delta(x^\pm_{i,0})]&\text{if}\quad i\neq m,0,\\
\pm\dfrac{1}{a_{i+1,i}}[\Delta(\widetilde{h}_{i+1,1}),\Delta(x^\pm_{i,0})]+b_{i+1,i}\dfrac{\varepsilon_1 - \varepsilon_2}{2} \Delta(x_{i, 0}^{\pm})&\text{if}\quad i=m,0,
\end{cases}
\end{equation*}
and $\Delta(\widetilde{h}_{i+1,1})$ and $\Delta(\widetilde{h}_{i,1})$ commute with $\Delta(h_{j,0})$ by the definition. We find that the defining relation \eqref{eq2.3} (resp.\ \eqref{eq2.5}, \eqref{eq2.6}) is equivalent to \eqref{3.3.4} (resp.\ \eqref{3.3.2}, \eqref{3.3.3}) by the proof of Lemma~\ref{Lemma39}. It is easy to show that $\Delta$ is compatible with \eqref{3.3.4}, \eqref{3.3.2}, and \eqref{3.3.3} in the same way as that of Theorem 4.9 of \cite{GNW}. Thus, it is enough to show that $\Delta$ is compatible with \eqref{eq2.1}. By the definition of $J(h_i)$, we obtain
\begin{align}
&\phantom{{}={}}[\Delta(h_{i1}), \Delta(h_{j1})]\nonumber\\
&=[\Delta(J(h_i)) - \Delta(\widetilde{v_i}),\Delta(J(h_j)) - \Delta(\widetilde{v_j})]\nonumber\\
&=[\Delta(J(h_i)),\Delta(J(h_j))] + [\Delta(\widetilde{v_i}),\Delta(\widetilde{v_j})] - [\Delta(J(h_i)),\Delta(\widetilde{v_j})] - [\Delta(\widetilde{v_i}),\Delta(J(h_j))],
\end{align}
where $\widetilde{v_i}=v_i+\dfrac{\ve_1+\ve_2}{2}h_i^2$. It is enough to show that 
\begin{equation}
[\Delta(J(h_i)),\Delta(J(h_j))] + [\Delta(\widetilde{v_i}),\Delta(\widetilde{v_j})]=0 \label{JhJh}
\end{equation}
and
\begin{equation} 
[\Delta(J(h_i)),\Delta(\widetilde{v_j})] + [\Delta(\widetilde{v_i}),\Delta(J(h_j))]=0\label{Jhv2} 
\end{equation} 
hold. We only show that \eqref{JhJh} holds. The outline of the proof of \eqref{Jhv2} is the same as that of Theorem 4.9 of \cite{GNW}. In order to simplify the computation, we define
\begin{align*}
\Omega_+&=\sum_{1\leq k \leq\text{dim}\mathfrak{h}}u^k\otimes u_k+\sum_{\alpha\in\Delta_+}\sum_{\substack{1\leq k_\alpha\leq\text{dim}\mathfrak{g}_\alpha}}{(-1)}^{p(\alpha)}x^{k_\alpha}_{\alpha}\otimes x^{k_\alpha}_{-\alpha},\\
\Omega_-&=\sum_{\alpha\in\Delta_+}\sum_{\substack{1\leq k_\alpha\leq\text{dim}\mathfrak{g}_\alpha}}x^{k_\alpha}_{-\alpha}\otimes x^{k_\alpha}_{\alpha},\\
\Omega&=\sum_{1\leq k \leq\text{dim}\mathfrak{h}}u^k\otimes u_k+\sum_{\alpha\in\Delta_+}\sum_{\substack{1\leq k_\alpha\leq\text{dim}\mathfrak{g}_\alpha}}({(-1)}^{p(\alpha)}x^{k_\alpha}_{\alpha}\otimes x^{k_\alpha}_{-\alpha}+x^{k_\alpha}_{-\alpha}\otimes x^{k_\alpha}_{\alpha}),
\end{align*}
where $\{u^k\}$ and $\{u_k\}$ are basis of $\mathfrak{h}$ such that $\kappa(u_k,u^l)=\delta_{k,l}$. By the definition of $J(h_i)$, it is easy to obtain
\begin{equation}
\Delta(J(h_i))=\square(J(h_i))+\dfrac{\ve_1+\ve_2}{2}[h_{i,0}\otimes 1, \Omega]\label{eqr7}
\end{equation}
since we have
\begin{align*}
\Delta(xy)&=(x\otimes1+1\otimes x)(y\otimes 1+1\otimes y)\\
&={(-1)}^{p(1)p(y)}xy\otimes1+{(-1)}^{p(x)p(1)}1\otimes xy+{(-1)}^{p(1)p(1)}x\otimes y+{(-1)}^{p(x)p(y)}y\otimes x
\end{align*}
by the relation $(x\otimes y)(z\otimes w)={(-1)}^{p(y)p(z)}xz\otimes yw$ for all homogeneous elements $x,y,z,w$.
Thus, by \eqref{eqr7}, we obtain
\begin{align*}
&\phantom{{}={}}[\Delta (J(h_i)), \Delta (J(h_j))]\\\nonumber
&= \square ([J(h_i), J(h_j)])+ \dfrac{\ve_1+\ve_2}{2}[\square (J(h_i)), [h_{j,0}\otimes 1,\Omega]]\\\nonumber
&\quad - \dfrac{\ve_1+\ve_2}{2}[\square (J(h_j)), [h_{i,0}\otimes 1, \Omega]]+ \dfrac{(\ve_1+\ve_2)^2}{4} [[h_{i,0}\otimes 1, \Omega], [h_{j,0}\otimes 1, \Omega]].
\end{align*}
First, we prove that
\begin{equation}\label{Jhv3}
\dfrac{\ve_1+\ve_2}{2}[\square (J(h_i)), [h_{j,0}\otimes 1,\Omega]] - \dfrac{\ve_1+\ve_2}{2}[\square (J(h_j)), [h_{i,0}\otimes 1, \Omega]]=0
\end{equation}
holds. Since $[h_{j,0}\otimes 1,\Omega]=\sum_{\alpha\in\Delta^{\re}_+} (\alpha,\alpha_i)(x_{-\alpha}\otimes x_\alpha-x_\alpha\otimes x_{-\alpha})$ holds, we have
\begin{align*}
&\phantom{{}={}}[\square (J(h_i)), [h_{j,0}\otimes 1,\Omega]]-[\square (J(h_j)), [h_{i,0}\otimes 1,\Omega]]\\
&=\sum_{\alpha\in\Delta^{\re}_+} (\alpha,\alpha_j)({(-1)}^{p(\alpha)}[J(h_i),x_\alpha]\otimes x_{-\alpha} -{(-1)}^{p(\alpha)} x_\alpha\otimes [J(h_i),x_{-\alpha}])\\
&\qquad\qquad\qquad\qquad\qquad\qquad+[J(h_i),x_{-\alpha}]\otimes x_\alpha -x_{-\alpha}\otimes [J(h_i),x_\alpha])\\
&\quad-\sum_{\alpha\in\Delta^{\re}_+} (\alpha,\alpha_i)({(-1)}^{p(\alpha)}[J(h_j),x_\alpha]\otimes x_{-\alpha} - {(-1)}^{p(\alpha)}x_\alpha\otimes [J(h_j),x_{-\alpha}])\\
&\qquad\qquad\qquad\qquad\qquad\qquad+[J(h_j),x_{-\alpha}]\otimes x_\alpha -x_{-\alpha}\otimes [J(h_j),x_\alpha])\\
&=\sum_{\alpha\in\Delta^{\re}_+}(\alpha,\alpha_j)(\alpha,\alpha_i)c_{i,j}^\alpha({(-1)}^{p(\alpha)}x_\alpha\otimes x_{-\alpha}-{(-1)}^{p(\alpha)}x_\alpha\otimes x_{-\alpha}+x_{-\alpha}\otimes x_\alpha-x_{-\alpha}\otimes x_\alpha)\\
&\quad-\sum_{\alpha\in\Delta^{\re}_+}(\alpha,\alpha_i)(\alpha,\alpha_j)c_{j,i}^\alpha({(-1)}^{p(\alpha)}x_\alpha\otimes x_{-\alpha}-{(-1)}^{p(\alpha)}x_\alpha\otimes x_{-\alpha}+x_{-\alpha}\otimes x_\alpha-x_{-\alpha}\otimes x_\alpha)\\
&=0.
\end{align*}
where the third equality is due to Proposition~\ref{Prop1}.
Therefore \eqref{Jhv3} holds.
Since $\Delta(\widetilde v_i)=\square(\widetilde v_i)-\dfrac{\ve_1+\ve_2}{2}[h_{i,0}\otimes 1, \Omega_+-\Omega_-]$ holds,
we obtain
\begin{align*}
&\phantom{{}={}}[\Delta (\widetilde v_i), \Delta (\widetilde v_j)]\\
&=\square ([\widetilde v_i, \widetilde v_j])+ \dfrac{\ve_1+\ve_2}{2}( - [\square (\widetilde v_i), [h_{j,0}\otimes 1, \Omega_+-\Omega_-]]+ [\square (\widetilde v_j), [h_{i,0}\otimes 1, \Omega_+-\Omega_-]])\\
&\quad + \dfrac{(\ve_1+\ve_2)^2}{4} [[h_{i,0}\otimes 1, \Omega_+-\Omega_-],[h_{j,0}\otimes 1, \Omega_+-\Omega_-]].
\end{align*}
Using this along with $\Omega = \Omega_+ + \Omega_-$ and \eqref{Jhv3}, we find the equality
\begin{align}
&\phantom{{}={}}[\Delta (J(h_i)), \Delta (J(h_j))] + [\Delta (\widetilde v_i), \Delta (\widetilde v_j)]\nonumber\\
&=\square([J(h_i), J(h_j)] + [\widetilde v_i, \widetilde v_j])\nonumber\\
&\quad+ \dfrac{\ve_1+\ve_2}{2}(- [\square (\widetilde v_i), [h_{j,0}\otimes 1, \Omega_+-\Omega_-]]+ [\square (\widetilde v_j), [h_{i,0}\otimes 1, \Omega_+-\Omega_-]])\nonumber\\
&\quad+ \dfrac{(\ve_1+\ve_2)^2}{2}([[h_{i,0}\otimes 1, \Omega_+],[h_{j,0}\otimes 1, \Omega_+]] +[[h_{i,0}\otimes 1, \Omega_-], [h_{j,0}\otimes 1, \Omega_-]]).\label{6660}
\end{align}
By the same way as the one of Theorem~4.9 in \cite{GNW}, we can check that the sum of the last four terms of the right hand side of \eqref{6660} vanishes. 
By Corollary~\ref{C1}, $\square([J(h_i), J(h_j)] + [\widetilde v_i, \widetilde v_j])=0$ holds. The coassocitivity is proven in a similar way to the one of \cite{GNW}. We complete the proof.
\end{proof}
By setting the degree on $Y_{\ve_1,\ve_2}(\widetilde{\mathfrak{sl}}(m|n))$ determined by \eqref{degree} and $\text{deg}(d)=0$, we can define the $\widehat{Y}_{\ve_1,\ve_2}(\widetilde{\mathfrak{sl}}(m|n))$ (resp. $Y_{\ve_1,\ve_2}(\widetilde{\mathfrak{sl}}(m|n)) \widehat{\otimes}  Y_{\ve_1,\ve_2}(\widetilde{\mathfrak{sl}}(m|n))$) as the degreewise completion of $Y_{\ve_1,\ve_2}(\widetilde{\mathfrak{sl}}(m|n))$ (resp. $Y_{\ve_1,\ve_2}(\widetilde{\mathfrak{sl}}(m|n))^{\otimes 2}$) in the sense of \cite{MNT}. We regard a represenation of $Y_{\ve_1,\ve_2}(\widetilde{\mathfrak{sl}}(m|n))$ as that of $\widetilde{\mathfrak{sl}}(m|n)$ via $\Phi$. 
By Theorem~\ref{Maim}, we easily obtain the following corollary.
\begin{Corollary}
The linear map
${\Delta} \colon Y_{\ve_1,\ve_2}(\widetilde{\mathfrak{sl}}(m|n))\rightarrow Y_{\ve_1,\ve_2}(\widetilde{\mathfrak{sl}}(m|n)) \widehat{\otimes}  Y_{\ve_1,\ve_2}(\widetilde{\mathfrak{sl}}(m|n))$
uniquely determined by
\begin{gather*}
{\Delta}(h_{i,0})={h_{i,0}}{\otimes}1+1{\otimes}{h_{i,0}},\quad{\Delta}(x^{\pm}_{i,0})={x^{\pm}_{i,0}}{\otimes}1+1{\otimes}x^{\pm}_{i,0}\quad\Delta(d)=d\otimes 1+1\otimes d,\nonumber\\
{\Delta}(h_{i,1})={h_{i,1}}{\otimes}1+1{\otimes}{h_{i,1}}+(\ve_1+\ve_2){h_{i,0}}{\otimes}{h_{i,0}}-(\ve_1+\ve_2)\sum_{{\alpha}\in{\Delta}_{+}}\sum_{1\leq k_\alpha\leq\text{dim}\mathfrak{g}_\alpha}(\alpha, {\alpha}_i)x^{k_\alpha}_{-\alpha}{\otimes}x^{k_\alpha}_{\alpha}
\end{gather*}
is an algebra homomorphism. Moreover, $\Delta$ satisfies the coassociativity. 

  In particular,
 $\Delta$ defines an action on
 $Y_{\ve_1,\ve_2}(\widetilde{\mathfrak{sl}}(m|n))$
 on $V\otimes W$
 for any
$Y_{\ve_1,\ve_2}(\widetilde{\mathfrak{sl}}(m|n))$-modules $V, W$ which are in the category $\mathcal{O}$ as $\widetilde{\mathfrak{sl}}(m|n)$-modules.
\end{Corollary}
\section{Evaluation map for the Affine Super Yangian}
Since the definition of the affine super Yangian is very complicated, it is not clear whether the affine super Yangian is trivial or not. In this section, we construct the non-trivial homomorphism from the affine super Yangian to the completion of $U(\widehat{\mathfrak{gl}}(m|n))$. 
In this section, we set $\widehat{\mathfrak{gl}}(m|n)=\mathfrak{gl}(m|n)\otimes\mathbb{C}[t^{\pm1}]\oplus\mathbb{C}c\oplus\mathbb{C}z$ as a Lie superalgebra whose defining relations are 
\begin{gather*}
\text{$c,z$ are a central elements},\\
[x\otimes t^s, y\otimes t^u]=\begin{cases}
[x,y]\otimes t^{s+u}+s\delta_{s+u,0}\text{str}(xy)c\quad\text{ if }x,\ y\in\mathfrak{sl}(m|n),\\
[e_{a,b},e_{i,i}]\otimes t^{s+u}+s\delta_{s+u,0}\text{str}(e_{a,b}e_{i,i})c+s\delta_{a,b}{(-1)}^{p(a)+p(i)}z\\
\qquad\qquad\qquad\qquad\qquad\qquad\qquad\text{ if }x=e_{a,b},\ y=e_{i,i}.
\end{cases}
\end{gather*}
For all $s\in\mathbb{Z}$, we denote $E_{i,j}\otimes t^s$ by $E_{i,j}(s)$. 
We also set the grading of $U(\widehat{\mathfrak{gl}}(m|n))/U(\widehat{\mathfrak{gl}}(m|n))(z-1)$ as $\text{deg}(X(s))=s$ and $\text{deg}(c)=0$. We define $U(\widehat{\mathfrak{gl}}(m|n))_{{\rm comp},+}$ as the standard degreewise completion of $U(\widehat{\mathfrak{gl}}(m|n))/U(\widehat{\mathfrak{gl}}(m|n))(z-1)$ in the sense of Definition~\ref{Defin}. 

Let us state the main result of this section. In order to simplify the notation, we denote $\ve_1+\ve_2$ as $\hbar$.
\begin{Theorem}\label{thm:main}
Assume $c\hbar=( -m+n) \ve_1$ and $z=1$. Let $\alpha$ be a complex number.
Then, there exists an algebra homomorphism $\ev \colon Y_{\ve_1,\ve_2}(\widehat{\mathfrak{sl}}(m|n)) \to U(\widehat{\mathfrak{gl}}(m|n))_{{\rm comp},+}$ uniquely determined by 
\begin{gather}
	\ev(x_{i,0}^{+}) = x_{i}^{+}, \quad \ev(x_{i,0}^{-}) = x_{i}^{-},\quad \ev(h_{i,0}) = h_{i},\label{evalu1}
\end{gather}
\begin{align}
	\ev(x_{i,1}^{+}) = \begin{cases}
		(\alpha - (m-n) \varepsilon_1) x_{0}^{+} + \hbar \displaystyle\sum_{s \geq 0} \limits\displaystyle\sum_{k=1}^{m+n}\limits {(-1)}^{p(k)}E_{m+n,k}(-s) E_{k,1}(s+1) \text{ if $i = 0$},\\
		\\
		(\alpha - (i-2\delta(i\geq m+1)(i-m)) \varepsilon_1) x_{i}^{+}\\
		 \quad+ \hbar \displaystyle\sum_{s \geq 0}\limits\displaystyle\sum_{k=1}^i\limits {(-1)}^{p(k)}E_{i,k}(-s) E_{k,i+1}(s)\\
		 \qquad +\hbar \displaystyle\sum_{s \geq 0}\limits\displaystyle\sum_{k=i+1}^{m+n}\limits {(-1)}^{p(k)}E_{i,k}(-s-1) E_{k,i+1}(s+1)\\
		\qquad\qquad\qquad\qquad\qquad\qquad\qquad\qquad\qquad\qquad\qquad\qquad\qquad\qquad \text{ if $i \neq 0$},
	\end{cases}\label{evalu2}
\end{align}
\begin{gather}
	\ev(x_{i,1}^-) = \begin{cases}
		(\alpha - (m-n) \varepsilon_1) x_{0}^{-} - \hbar \displaystyle\sum_{s \geq 0} \limits\displaystyle\sum_{k=1}^{m+n}\limits {(-1)}^{p(k)}E_{1,k}(-s-1) E_{k,m+n}(s) \text{ if $i = 0$},\\
		\\
		(\alpha - (i-2\delta(i\geq m+1)(i-m)) \varepsilon_1) x_{i}^{-}\\
		 \quad+ {(-1)}^{p(i)}\hbar \displaystyle\sum_{s \geq 0}\limits\displaystyle\sum_{k=1}^i\limits {(-1)}^{p(k)}E_{i+1,k}(-s) E_{k,i}(s)\\
		 \qquad+ {(-1)}^{p(i)}\hbar \displaystyle\sum_{s \geq 0}\limits\displaystyle\sum_{k=i+1}^{m+n}\limits {(-1)}^{p(k)}E_{i+1,k}(-s-1) E_{k,i}(s+1) \\
		\qquad\qquad\qquad\qquad\qquad\qquad\qquad\qquad\qquad\qquad\qquad\qquad\qquad\qquad\text{ if $i \neq 0$},
	\end{cases}\label{evalu3}
\end{gather}
\begin{gather}
	\ev(h_{i,1}) = \begin{cases}
		(\alpha - (m-n) \varepsilon_1)h_{0} + \hbar E_{m+n,m+n} (E_{1,1}-c) \\
		\ -\hbar \displaystyle\sum_{s \geq 0} \limits\displaystyle\sum_{k=1}^{m+n}\limits{(-1)}^{p(k)}E_{m+n,k}(-s) E_{k,m+n}(s)\\
		\quad-\hbar \displaystyle\sum_{s \geq 0} \quad\displaystyle\sum_{k=1}^{m+n}\limits{(-1)}^{p(k)}E_{1,k}(-s-1) E_{k,1}(s+1)\\
		 \qquad\qquad\qquad\qquad\qquad\qquad\qquad\qquad\qquad\qquad\qquad\qquad\qquad\qquad\text{ if $i = 0$},\\
\\
		(\alpha - (i-2\delta(i\geq m+1)(i-m))\varepsilon_1) h_{i} -{(-1)}^{p(E_{i,i+1})} \hbar E_{i,i}E_{i+1,i+1} \\
		\quad + \hbar{(-1)}^{p(i)} \displaystyle\sum_{s \geq 0}  \limits\displaystyle\sum_{k=1}^{i}\limits{(-1)}^{p(k)} E_{i,k}(-s) E_{k,i}(s)\\
		\qquad +\hbar{(-1)}^{p(i)} \displaystyle\sum_{s \geq 0} \limits\displaystyle\sum_{k=i+1}^{m+n}\limits {(-1)}^{p(k)}E_{i,k}(-s-1) E_{k,i}(s+1) \\
		\qquad\quad -\hbar{(-1)}^{p(i+1)}\displaystyle\sum_{s \geq 0}\limits\displaystyle\sum_{k=1}^{i}\limits{(-1)}^{p(k)}E_{i+1,k}(-s) E_{k,i+1}(s)\\
		 \qquad\qquad-\hbar{(-1)}^{p(i+1)}\displaystyle\sum_{s \geq 0}\limits\displaystyle\sum_{k=i+1}^{m+n} \limits{(-1)}^{p(k)}E_{i+1,k}(-s-1) E_{k,i+1}(s+1)\\
		\qquad\qquad\qquad\qquad\qquad\qquad\qquad\qquad\qquad\qquad\qquad\qquad\qquad\qquad \text{ if $i \neq 0$}.
	\end{cases}\label{evalu4}
\end{gather}
\end{Theorem}
The outline of the proof is the same as that of \cite{K1}. It is enough to check that $\ev$ is compatible with \eqref{eq2.1}-\eqref{eq2.9}, which are the defining relations of the minimalistic presentation of the affine super Yangian. When we restrict $\ev$ to $\widehat{\mathfrak{sl}}(m|n)$, $\ev$ is an identity map on $\widehat{\mathfrak{sl}}(m|n)$. Thus, $\ev$ is compatible with \eqref{eq2.2}, \eqref{eq2.4}, \eqref{eq2.7}-\eqref{eq2.9}.

We set a anti-automorphism $\omega\colon U(\widehat{\mathfrak{gl}}(m|n))\to U(\widehat{\mathfrak{gl}}(m|n))$ as
\begin{gather*}
\omega(X\otimes t^r)={(-1)}^rX^{T}\otimes t^r,\qquad\omega(c)=c,
\end{gather*}
where $X^{T}$ is a transpose of a matrix $X$. Then, the compatibility of $\ev$ with \eqref{eq2.5} and \eqref{eq2.6} for $-$ are deduced from those for $+$ by applying the anti-automorphism $\omega$ since we have $\omega(\ev(h_{i,1}))=\ev(h_{i,1})$ and $\omega(\ev(x_{i,1}^+))={(-1)}^{p(i)}\ev(x_{i,1}^{-})$.
Therefore, it is enough to check the following lemma.
\begin{Lemma}\label{Lemma1}
The following equations hold;
\begin{gather}
[\ev(x^+_{i,1}),\ev(x^-_{j,0})]=\delta_{i,j}\ev(h_{i,1}),\label{rel4}\\
[\ev(\widetilde{h}_{i,1}),x^+_j]=a_{i,j}(\ev(x^+_{j,1})-b_{i,j}\dfrac{\ve_1-\ve_2}{2}x^+_j),\label{rel5}\\
[\ev(x^+_{i,1}),x^+_j]-[x^+_i,\ev(x^+_{j,1})]=a_{i,j}\dfrac{\ve_1+\ve_2}{2}\{x^+_i, x^+_j\}-b_{i,j}\dfrac{\ve_1-\ve_2}{2}[x^+_i, x^+_j],\label{rel6}\\
[\ev(h_{i,1}),\ev(h_{j,1})]=0.\label{rel7}
\end{gather}
\end{Lemma}
The rest of the paper is devoted to the proof of Lemma~\ref{Lemma1}.
\subsection{The proof of \eqref{rel4}}
We prepare one claim before starting the proof. 
\begin{Claim}
The following relations hold;
\begin{align}
&\quad[\displaystyle\sum_{s \geq p}  \limits\displaystyle\sum_{k=1}^{a}\limits{(-1)}^{p(k)} E_{i,k}(-s) E_{k,j}(s),E_{x,y}]\\
&=\delta_{j,x}\displaystyle\sum_{s \geq p}  \limits\displaystyle\sum_{k=1}^{a}\limits{(-1)}^{p(k)} E_{i,k}(-s) E_{k,y}(s)-{(-1)}^{p(E_{i,j})p(E_{x,y})}\displaystyle\sum_{s \geq p}  \limits\displaystyle\sum_{k=1}^{a}\limits{(-1)}^{p(k)} E_{x,k}(-s) E_{k,j}(s)\nonumber\\
&\quad+\{\delta(x\leq a<y)-\delta(x>a\geq y)\}\displaystyle\sum_{s \geq p}  \limits{(-1)}^{p(x)+p(E_{x,j})p(E_{x,y})} E_{i,y}(-s) E_{x,j}(s),\label{equation154}\\
&\quad[\displaystyle\sum_{s \geq p}  \limits\displaystyle\sum_{k=a}^{m+n}\limits{(-1)}^{p(k)} E_{i,k}(-s) E_{k,j}(s),E_{x,y}]\nonumber\\
&=\delta_{j,x}\displaystyle\sum_{s \geq p}  \limits\displaystyle\sum_{k=a}^{m+n}\limits{(-1)}^{p(k)} E_{i,k}(-s) E_{k,y}(s)-{(-1)}^{p(E_{i,j})p(E_{x,y})}\displaystyle\sum_{s \geq p}  \limits\displaystyle\sum_{k=a}^{m+n}\limits{(-1)}^{p(k)} E_{x,k}(-s) E_{k,j}(s)\nonumber\\
&\quad+\{\delta(x\geq a>y)-\delta(x<a\leq y)\}\displaystyle\sum_{s \geq p}  \limits{(-1)}^{p(x)+p(E_{x,j})p(E_{x,y})} E_{i,y}(-s) E_{x,j}(s).\label{equation155}
\end{align}
\end{Claim}
\begin{proof}
We prove only \eqref{equation154} since \eqref{equation155} is proven in a similar way. By direct computation, the first term of \eqref{equation154} is equal to
\begin{align}
&\delta_{j,x}\displaystyle\sum_{s \geq p}  \limits\displaystyle\sum_{k=1}^{a}\limits{(-1)}^{p(k)} E_{i,k}(-s) E_{k,y}(s)\nonumber\\
&\quad-\delta(y\leq a)\displaystyle\sum_{s \geq p}  \limits{(-1)}^{p(y)+p(E_{y,j})p(E_{x,y})} E_{i,y}(-s) E_{x,j}(s)\nonumber\\
&\quad+\delta(x\leq a)\displaystyle\sum_{s \geq p}  \limits{(-1)}^{p(x)+p(E_{x,j})p(E_{x,y})} E_{i,y}(-s) E_{x,j}(s)\nonumber\\
&\quad-{(-1)}^{p(E_{i,j})p(E_{x,y})}\displaystyle\sum_{s \geq p}  \limits\displaystyle\sum_{k=1}^{a}\limits{(-1)}^{p(k)} E_{x,k}(-s) E_{k,j}(s).\label{equation156}
\end{align}
Since $p(y)+p(E_{y,j})p(E_{x,y})=p(x)+p(E_{x,j})p(E_{x,y})$, the sum of the second and third terms of \eqref{equation156} is equal to
\begin{align*}
&\{\delta(x\leq a<y)-\delta(x>a\geq y)\}\displaystyle\sum_{s \geq p}  \limits{(-1)}^{p(x)+p(E_{x,j})p(E_{x,y})} E_{i,y}(-s) E_{x,j}(s).
\end{align*}
Then, we obtain \eqref{equation155}.
\end{proof}
Suppose that $i, j\neq0$. Other cases are proven in a similar way. By the definition of $\ev(x^+_{i,1})$, we obtain
\begin{align}
&\phantom{{}={}}[\ev(x^+_{i,1}),\ev(x^-_{j,0})]\nonumber\\
&=[(\alpha - (i-2\delta(i\geq m+1)(i-m))\varepsilon_1) x_{i}^{+}, {(-1)}^{p(j)}E_{j+1,j}]\nonumber\\
&\quad+ [\hbar \sum_{s \geq 0}\sum_{k=1}^i {(-1)}^{p(k)}E_{i,k}(-s) E_{k,i+1}(s), {(-1)}^{p(j)}E_{j+1,j}]\nonumber\\
&\quad+ [\hbar \sum_{s \geq 0}\sum_{k=i+1}^{m+n} {(-1)}^{p(k)}E_{i,k}(-s-1) E_{k,i+1}(s+1), {(-1)}^{p(j)}E_{j+1,j}].\label{r1}
\end{align}
By \eqref{equation154}, $[\hbar \displaystyle\sum_{s \geq 0}\limits\displaystyle\sum_{k=1}^i\limits {(-1)}^{p(k)}E_{i,k}(-s) E_{k,i+1}(s), {(-1)}^{p(j)}E_{j+1,j}]$, the second term of the right hand side of \eqref{r1}, is equal to
\begin{align}
&\phantom{{}={}}[\hbar \sum_{s \geq 0}\sum_{k=1}^i{(-1)}^{p(k)} E_{i,k}(-s) E_{k,i+1}(s), {(-1)}^{p(j)}E_{j+1,j}]\nonumber\\
&=\delta_{i,j}\hbar\sum_{s \geq 0}\sum_{k=1}^i{(-1)}^{p(i)+p(k)}E_{i,k}(-s) E_{k,i}(s)\nonumber\\
&\quad-\delta_{i,j}\hbar\sum_{s \geq 0}\sum_{k=1}^i{(-1)}^{p(k)+p(i)+p(E_{i,i+1})}E_{i+1,k}(-s) E_{k,i+1}(s)\nonumber\\
&\quad-\delta_{i,j}\hbar\sum_{s \geq 0}{(-1)}^{p(E_{i,i+1})p(E_{i,i+1})}E_{i,i}(-s) E_{i+1,i+1}(s).\label{813}
\end{align}
Similarly, by \eqref{equation155}, $[\hbar \displaystyle\sum_{s \geq 0}\limits\displaystyle\sum_{k=i+1}^{m+n}\limits{(-1)}^{p(k)} E_{i,k}(-s-1) E_{k,i+1}(s+1), {(-1)}^{p(j)}E_{j+1,j}]$, the third term of the right hand side of \eqref{r1}, is equal to 
\begin{align}
&\phantom{{}={}}[\hbar \sum_{s \geq 0}\sum_{k=i+1}^{m+n} {(-1)}^{p(k)}E_{i,k}(-s-1) E_{k,i+1}(s+1), {(-1)}^{p(j)}E_{j+1,j}]\nonumber\\
&=\hbar\sum_{s \geq 0}\sum_{k=i+1}^{m+n} \delta_{i,j}{(-1)}^{p(k)+p(i)}E_{i,k}(-s-1) E_{k,i}(s+1)\nonumber\\
&\quad-\hbar\sum_{s \geq 0}\sum_{k=i+1}^{m+n}\delta_{i,j}{(-1)}^{p(k)+p(i)+p(E_{i,i+1})}E_{i+1,k}(-s-1) E_{k,i+1}(s+1)\nonumber\\
&\quad+\hbar\sum_{s \geq 0}\delta_{i,j}{(-1)}^{p(i+1)+p(i)}E_{i,i}(-s-1) E_{i+1,i+1}(s+1).\label{814}
\end{align}
We can rewrite the sum of the last term of \eqref{813} and the last term of \eqref{814}. Since $p(E_{i,i+1})=p(i)+p(i+1)$ holds, we obtain
\begin{align}
&\phantom{{}={}}-\hbar\sum_{s \geq 0}{(-1)}^{p(E_{i,i+1})p(E_{i,i+1})}E_{i,i}(-s) E_{i+1,i+1}(s)\nonumber\\
&\phantom{{}={}}+\hbar\sum_{s \geq 0}{(-1)}^{p(i+1)+p(i)+p(E_{i+1,i+1})p(E_{i,i+1})}E_{i,i}(-s-1) E_{i+1,i+1}(s+1)\nonumber\\
&=-\hbar\sum_{s \geq 0}{(-1)}^{p(E_{i,i+1})}E_{i,i}(-s) E_{i+1,i+1}(s)+\hbar\sum_{s \geq 0}{(-1)}^{p(E_{i,i+1})}E_{i,i}(-s-1) E_{i+1,i+1}(s+1)\nonumber\\
&=-\hbar{(-1)}^{p(E_{i,i+1})}E_{i,i} E_{i+1,i+1}.\label{QUO}
\end{align}
Thus, we have shown that $[\ev(x^+_{i,1}),\ev(x^-_{j,0})]=\delta_{i,j}\ev(h_{i,1})$ holds by \eqref{813}, \eqref{814} and \eqref{QUO}. 

\subsection{The proof of \eqref{rel5}}
We only show the case where $i,j\neq0$ and when $i=0$ and $j\neq0$. The other case is proven in a similar way.

{\bf Case 1,}\ $i,j\neq0$.

First, we show the case where $i, j\neq0$. By the definition of $\ev(h_{i,1})$, we obtain
\begin{align}
&[\ev(\widetilde{h}_{i,1}),\ev(x^+_{j,0})]\nonumber\\
&=[(\alpha - (i-2\delta(i\geq m+1)(i-m))\varepsilon_1) h_{i} - \dfrac{1}{2}\hbar((E_{i,i})^2+(E_{i+1,i+1})^2), E_{j,j+1}]\nonumber\\
&\quad + [\hbar{(-1)}^{p(i)} \sum_{s \geq 0} \sum_{k=1}^{i} {(-1)}^{p(k)}E_{i,k}(-s) E_{k,i}(s), E_{j,j+1}]\nonumber\\
&\quad +[\hbar{(-1)}^{p(i)}\sum_{s \geq 0}\sum_{k=i+1}^{m+n} {(-1)}^{p(k)}E_{i,k}(-s-1) E_{k,i}(s+1), E_{j,j+1}]\nonumber\\
&\quad -[ \hbar{(-1)}^{p(i+1)}\sum_{s \geq 0}\sum_{k=1}^{i}{(-1)}^{p(k)} E_{i+1,k}(-s) E_{k,i+1}(s), E_{j,j+1}]\nonumber\\
&\quad -[\hbar{(-1)}^{p(i+1)}\sum_{s \geq 0}\sum_{k=i+1}^{m+n} {(-1)}^{p(k)}E_{i+1,k}(-s-1) E_{k,i+1}(s+1), E_{j,j+1}].\label{r2}
\end{align}
Let us compute these terms respectively. By direct computation, the first term of the right hand side of \eqref{r2} is equal to
\begin{align}
&\phantom{{}={}}[(\alpha - (i-2\delta(i\geq m+1)(i-m))\varepsilon_1) h_{i} - \dfrac{1}{2}\hbar((E_{i,i})^2+(E_{i+1,i+1})^2), E_{j,j+1}]\nonumber\\
&=(\alpha - (i-2\delta(i\geq m+1)(i-m))\varepsilon_1)a_{i,j}x^+_j\nonumber\\
&\quad-\dfrac{\hbar}{2} (\delta_{i,j}(\{E_{i,i+1}, E_{i,i}\}-\{E_{i,i+1}, E_{i+1,i+1}\})-\delta_{i,j+1}\{E_{i-1,i}, E_{i,i}\}+\delta_{i+1,j}\{E_{i+1,i+2}, E_{i+1,i+1}\}).\label{411}
\end{align}
By \eqref{equation154} and \eqref{equation155}, we also find that the sum of the second and third terms of the right hand side of \eqref{r2} is equal to
\begin{align}
&\phantom{{}={}}[\hbar{(-1)}^{p(i)} \sum_{s \geq 0} \sum_{k=1}^{i} {(-1)}^{p(k)}E_{i,k}(-s) E_{k,i}(s), E_{j,j+1}]\nonumber\\
&\qquad\qquad\qquad\qquad\qquad+[\hbar{(-1)}^{p(i)}\sum_{k=i+1}^{m+n} {(-1)}^{p(k)}E_{i,k}(-s-1) E_{k,i}(s+1), E_{j,j+1}]\nonumber\\
&=\hbar{(-1)}^{p(i)} \sum_{s \geq 0}\sum_{k=1}^{i}\delta_{i,j}{(-1)}^{p(k)}E_{i,k}(-s) E_{k,i+1}(s)\nonumber\\
&\quad-\hbar{(-1)}^{p(i)} \sum_{s \geq 0}\sum_{k=1}^{i}\delta_{i,j+1}{(-1)}^{p(k)}E_{j,k}(-s) E_{k,i}(s)\nonumber\\
&\quad+\hbar{(-1)}^{p(i)} \sum_{s \geq 0}\delta_{i,j}{(-1)}^{p(i)}E_{i,i+1}(-s) E_{i,i}(s)\nonumber\\
&\quad+\hbar{(-1)}^{p(i)} \sum_{s \geq 0}\sum_{k=i+1}^{m+n}\delta_{i,j}{(-1)}^{p(k)}E_{i,k}(-s-1) E_{k,j+1}(s+1)\nonumber\\
&\quad- \hbar{(-1)}^{p(i)} \sum_{s \geq 0}\sum_{k=i+1}^{m+n}\delta_{i,j+1}{(-1)}^{p(k)}E_{j,k}(-s-1) E_{k,i}(s+1)\nonumber\\
&\quad-\hbar{(-1)}^{p(i)} \sum_{s \geq 0}\delta_{i,j}{(-1)}^{p(i+1)+p(E_{i,i+1})p(E_{i+1,i})}E_{i,i+1}(-s-1) E_{i,i}(s+1).\label{413}
\end{align}
By a direct computation, we obtain
\begin{align}
\text{the sum of the third and 6-th terms of \eqref{413}}=\hbar\delta_{i,j}E_{i,i+1}E_{i,i}.\label{r5.5}
\end{align}
Next, let us rewrite the sum of the first and 4-th terms of \eqref{413}. By the definition of $\ev(x_{i,1}^{+})$, we obtain
\begin{align}
&\phantom{{}={}}\text{the first term of \eqref{413}}+\text{the 4-th term of \eqref{413}}\nonumber\\
&=\delta_{i,j}(\ev(x_{i,1}^{+})-(\alpha - (i-2\delta(i\geq m+1)(i-m)) \varepsilon_1) x_{i}^{+}).\label{r8.5}
\end{align}
By the definition of $\ev(x_{i,1}^{+})$, we also obtain
\begin{align}
&\phantom{{}={}}\text{the second term of \eqref{413}}+\text{the 5-th term of \eqref{413}}\nonumber\\
&=-\delta_{i,j+1}\hbar{(-1)}^{p(i)} \sum_{s \geq 0}\sum_{k=1}^{j}{(-1)}^{p(k)}E_{j,k}(-s) E_{k,i}(s)\nonumber\\
&\quad-\delta_{i,j+1}\hbar{(-1)}^{p(i)} \sum_{s \geq 0}\sum_{k=j+1}^{m+n}{(-1)}^{p(k)}E_{j,k}(-s-1) E_{k,i}(s+1)\nonumber\\
&\quad-\hbar\delta_{i,j+1}E_{j,i} E_{i,i}\nonumber\\
&=-\delta_{i,j+1}{(-1)}^{p(i)}(\ev(x_{j,1}^{+})-(\alpha - (j-2\delta(i\geq m+1)(j-m)) \varepsilon_1) x_{j}^{+})-\hbar\delta_{i,j+1}E_{j,i} E_{i,i}.\label{r9}
\end{align}
Therefore, by \eqref{r5.5}, \eqref{r8.5} and \eqref{r9}, the sum of first, second and third terms of the right hand side of \eqref{r2} is equal to
\begin{align}
&(\alpha - (i-2\delta(i\geq m+1)(i-m))\varepsilon_1)a_{i,j}x^+_j\nonumber\\
&\quad-\dfrac{\hbar}{2} (\delta_{i,j}(\{E_{i,i+1}, E_{i,i}\}-\{E_{i,i+1}, E_{i+1,i+1}\})-\delta_{i,j+1}\{E_{i-1,i}, E_{i,i}\}+\delta_{i+1,j}\{E_{i+1,i+2}, E_{i+1,i+1}\})\nonumber\\
&\quad+\hbar\delta_{i,j}E_{i,i+1}E_{i,i}+{(-1)}^{p(i)}\delta_{i,j}(\ev(x_{i,1}^{+})-(\alpha - (i-2\delta(i\geq m+1)(i-m)) \varepsilon_1) x_{i}^{+})\nonumber\\
&\quad-{(-1)}^{p(i)}\delta_{i,j+1}(\ev(x_{j,1}^{+})-(\alpha - (j-2\delta(i\geq m+1)(j-m)) \varepsilon_1) x_{j}^{+}).\label{super11}
\end{align}

Similarly to \eqref{super11}, we find that the sum of the 4-th and 5-th terms of the right hand side of \eqref{r2} is equal to
\begin{align}
&-\hbar\delta_{j,i}E_{i+1,i+1}E_{i,i+1}-{(-1)}^{p(i+1)}\delta_{i+1,j}(\ev(x_{j,1}^{+})-(\alpha - (i-2\delta(j\geq m+1)(j-m)) \varepsilon_1) x_{j}^{+})\nonumber\\
&\quad+\delta_{i+1,j}\hbar E_{i+1,i+1} E_{i+1,j+1}+\delta_{i,j}{(-1)}^{p(i+1)}(\ev(x^+_{i,1})-(\alpha - (i-2\delta(i\geq m+1)(i-m)) \varepsilon_1) x_{i}^{+}).\label{super12}
\end{align}
Then, $[\ev(\widetilde{h}_{i,1}),\ev(x^+_{j,0})]$ is equal to the sum of \eqref{411}, \eqref{super11} and \eqref{super12}.
\begin{align*}
&(\alpha - (i-2\delta(i\geq m+1)(i-m))\varepsilon_1)a_{i,j}x^+_j\nonumber\\
&\quad-\dfrac{\hbar}{2} (\delta_{i,j}(\{E_{i,i+1}, E_{i,i}\}-\{E_{i,i+1}, E_{i+1,i+1}\})-\delta_{i,j+1}\{E_{i-1,i}, E_{i,i}\}+\delta_{i+1,j}\{E_{i+1,i+2}, E_{i+1,i+1}\})\\
&\quad+(\alpha - (i-2\delta(i\geq m+1)(i-m))\varepsilon_1)a_{i,j}x^+_j\nonumber\\
&\quad-\dfrac{\hbar}{2} (\delta_{i,j}(\{E_{i,i+1}, E_{i,i}\}-\{E_{i,i+1}, E_{i+1,i+1}\})-\delta_{i,j+1}\{E_{i-1,i}, E_{i,i}\}+\delta_{i+1,j}\{E_{i+1,i+2}, E_{i+1,i+1}\})\nonumber\\
&\quad+\hbar\delta_{i,j}E_{i,i+1}E_{i,i}+{(-1)}^{p(i)}\delta_{i,j}(\ev(x_{i,1}^{+})-(\alpha - (i-2\delta(i\geq m+1)(i-m)) \varepsilon_1) x_{i}^{+})\nonumber\\
&\quad-{(-1)}^{p(i)}\delta_{i,j+1}(\ev(x_{j,1}^{+})-(\alpha - (j-2\delta(i\geq m+1)(j-m)) \varepsilon_1) x_{j}^{+})\\
&\quad-\hbar\delta_{j,i}E_{i+1,i+1}E_{i,i+1}-{(-1)}^{p(i+1)}\delta_{i+1,j}(\ev(x_{j,1}^{+})-(\alpha - (i-2\delta(j\geq m+1)(j-m)) \varepsilon_1) x_{j}^{+})\nonumber\\
&\quad+\delta_{i+1,j}\hbar E_{i+1,i+1} E_{i+1,j+1}+\delta_{i,j}{(-1)}^{p(i+1)}(\ev(x^+_{i,1})-(\alpha - (i-2\delta(i\geq m+1)(i-m)) \varepsilon_1) x_{i}^{+}).
\end{align*}
By \eqref{411}, \eqref{super11} and \eqref{super12}, when $i\neq j, j\pm1$, $[\ev(\widetilde{h}_{i,1}),\ev(x^+_{j,0})]$ is zero. Provided that $i=j$, $[\ev(\widetilde{h}_{i,1}),\ev(x^+_{i,0})]$ is equal to
\begin{align}
&(\alpha - (i-2\delta(i\geq m+1)(i-m))\varepsilon_1)a_{i,i}x^+_i-\dfrac{\hbar}{2}(\{E_{i,i+1}, E_{i,i}\}-\{E_{i,i+1}, E_{i+1,i+1}\})\nonumber\\
&\quad+\hbar E_{i,i+1}E_{i,i}+{(-1)}^{p(i)}(\ev(x_{i,1}^{+})-(\alpha - (i-2\delta(i\geq m+1)(i-m)) \varepsilon_1) x_{i}^{+})\nonumber\\
&\quad-\hbar E_{i+1,i+1}E_{i,i+1}+{(-1)}^{p(i+1)}(\ev(x^+_{i,1})-(\alpha - (i-2\delta(i\geq m+1)(i-m)) \varepsilon_1) x_{i}^{+}).\label{super210}
\end{align}
Since $a_{i,i}={(-1)}^{p(i)}+{(-1)}^{p(i+1)}$ holds, we have
\begin{align*}
&(\alpha - (i-2\delta(i\geq m+1)(i-m))\varepsilon_1)a_{i,i}x^+_i-{(-1)}^{p(i)}(\alpha - (i-2\delta(i\geq m+1)(i-m)) \varepsilon_1) x_{i}^{+}\\
&\qquad\qquad-{(-1)}^{p(i+1)}(\alpha - (i-2\delta(i\geq m+1)(i-m)) \varepsilon_1) x_{i}^{+}=0
\end{align*}
and
\begin{align*}
&\phantom{{}={}}-\dfrac{\hbar}{2} (\{E_{i,i+1}, E_{i,i}\}-\{E_{i,i+1}, E_{i+1,i+1}\})+\hbar E_{i,i+1} E_{i,i}-\hbar E_{i+1,i+1}E_{i,i+1}\\
&=-\dfrac{\hbar}{2}(E_{i,i}E_{i,i+1}-E_{i,i+1} E_{i,i}+E_{i+1,i+1}E_{i,i+1}-E_{i,i+1}E_{i+1,i+1})\\
&=-\dfrac{\hbar}{2}(E_{i,i+1}-E_{i,i+1})=0.
\end{align*}
Then, we find that $[\ev(\widetilde{h}_{i,1}),\ev(x^+_{i,0})]$ is equal to $a_{i,i}\ev(x^+_{i,1})$.

When $i=j+1$, $[\ev(\widetilde{h}_{i,1}),\ev(x^+_{j,0})]$ is equal to
\begin{align}
&(\alpha - (i-2\delta(i\geq m+1)(i-m))\varepsilon_1)a_{i,j}x^+_j+\dfrac{\hbar}{2}\{E_{i-1,i}, E_{i,i}\}\nonumber\\
&-{(-1)}^{p(i)}(\ev(x_{j,1}^{+})-(\alpha - (j-2\delta(i\geq m+1)(j-m)) \varepsilon_1) x_{j}^{+})-\hbar E_{j,i}E_{i,i}.
\end{align}
Since $a_{i,j}=-{(-1)}^{p(i)}$ holds, we have
\begin{align*}
\dfrac{\hbar}{2}\{E_{i-1,i}, E_{i,i}\}-\hbar E_{j,i}E_{i,i}  &=\dfrac{\hbar}{2}[E_{i,i},E_{i-1,i}]=-\dfrac{\hbar}{2}E_{i-1,i}
\end{align*}
and
\begin{align*}
&\quad(\alpha - (i-2\delta(i\geq m+1)(i-m))\varepsilon_1)a_{i,j}x^+_j+{(-1)}^{p(i)}(\alpha-(j-2\delta(i\geq m+1)(j-m)\ve_1)x^+_j\\
&=\ve_1x^+_j.
\end{align*}
Then, we find that $[\ev(\widetilde{h}_{i,1}),\ev(x^+_{j,0})]$ is equal to $a_{i,i-1}(\ev(x^+_{i-1})+a_{i,i-1}\dfrac{\ve_1-\ve_2}{2}E_{i-1,i})$. 

When $i=j-1$,  $[\ev(\widetilde{h}_{i,1}),\ev(x^+_{j,0})]$ is equal to
\begin{align}
&(\alpha - (i-2\delta(i\geq m+1)(i-m))\varepsilon_1)a_{i,j}x^+_j-\dfrac{\hbar}{2} (\{E_{i+1,i+2}, E_{i+1,i+1}\})\nonumber\\
&-{(-1)}^{p(i+1)}(\ev(x_{j,1}^{+})-(\alpha - (i-2\delta(j\geq m+1)(j-m)) \varepsilon_1) x_{j}^{+}+\hbar E_{i+1,i+1} E_{i+1,j+1}.
\end{align}
Since $a_{i,j}=-{(-1)}^{p(j)}$ holds, we have
\begin{align*}
-\dfrac{\hbar}{2}\{E_{i+1,i+2}, E_{i+1,i+1}\}+\hbar E_{i+1,i+1} E_{i+1,j+1}=\dfrac{\hbar}{2}[E_{i+1,i+1}, E_{i+1,i+2}]=\dfrac{\hbar}{2}E_{i+1,i+2}
\end{align*}
and
\begin{align*}
&\quad(\alpha - (i-2\delta(i\geq m+1)(i-m))\varepsilon_1)a_{i,j}x^+_j+{(-1)}^{p(i+1)}(\alpha - (i-2\delta(j\geq m+1)(j-m)) \varepsilon_1) x_{j}^{+}\\
&=-\ve_1E_{i+1,i+2}
\end{align*}
Then, $[\ev(\widetilde{h}_{i,1}),\ev(x^+_{i+1,0})]$ is equal to $a_{i,i+1}(\ev(x^+_{i+1,1})-a_{i,i+1}\dfrac{\ve_1-\ve_2}{2}E_{i+1,i+2})$.

{\bf Case 2,}\ $i=0$ and $j\neq0$.

By the definition of $\ev$, we obtain
\begin{align}
&\phantom{{}={}}[\ev(\widetilde{h}_{0,1}),\ev(x^+_{j,0})]\nonumber\\
&=[(\alpha - (m-n) \varepsilon_1) h_{0}- \dfrac{1}{2}\hbar((E_{m+n,m+n})^2+(E_{1,1}-c)^2), E_{j,j+1}]\nonumber\\
&\quad- [\hbar \sum_{s \geq 0} \sum_{k=1}^{m+n}{(-1)}^{p(k)}E_{m+n,k}(-s) E_{k,m+n}(s), E_{j,j+1}]\nonumber\\
&\quad - [\hbar \sum_{s \geq 0} \sum_{k=1}^{m+n}{(-1)}^{p(k)}E_{1,k}(-s-1) E_{k,1}(s+1), E_{j,j+1}].\label{r3}
\end{align}
By direct computation, the first term of \eqref{r3} is equal to
\begin{align}
&(\alpha - (m-n) \varepsilon_1)a_{0,j}x^+_j\nonumber\\
&\quad- \dfrac{\hbar}{2} \Big( -\delta_{m+n-1,j}(\{E_{m+n,m+n}, E_{m+n-1,m+n}\}+\delta_{1,j}\{E_{1,2}, (E_{1,1}-c)\})\Big)\label{311}
\end{align}
We also find that the second term of \eqref{r3} is equal to
\begin{align}
&\hbar \sum_{s \geq 0}{(-1)}^{p(j+1)+p(E_{j,j+1})p(E_{j+1,m+n})}E_{m+n,j+1}(-s) E_{j,m+n}(s)\nonumber\\
&\quad- \hbar \sum_{s \geq 0}{(-1)}^{p(j)+p(E_{j,j+1})p(E_{j,m+n})}E_{m+n,j+1}(-s) E_{j,m+n}(s)\nonumber\\
&\quad+ \hbar \sum_{s \geq 0} \sum_{k=1}^{m+n}\delta_{m+n,j+1}{(-1)}^{p(k)}E_{m+n-1,k}(-s) E_{k,m+n}(s).\label{312}
\end{align}
By direct computation, we also know that the third term of \eqref{r3} is equal to
\begin{align}
&- \hbar \sum_{s \geq 0} \sum_{k=1}^{m+n}\delta_{1,j}{(-1)}^{p(k)}E_{1,k}(-s-1) E_{k,2}(s+1)\nonumber\\
&\quad+ \hbar \sum_{s \geq 0}{(-1)}^{p(j+1)+p(E_{j,j+1})p(E_{j+1,1})}E_{1,j+1}(-s-1) E_{j,1}(s+1)\nonumber\\
&\quad- \hbar \sum_{s \geq 0}{(-1)}^{p(j)+p(E_{j,j+1})p(E_{j,1})}E_{1,j+1}(-s-1) E_{j,1}(s+1).\label{313}
\end{align}
First, we show that the sum of the first and second terms of \eqref{312} is equal to zero. By direct computation, we have
\begin{align}
&\phantom{{}={}}\text{the first term of \eqref{312}}+\text{the second term of \eqref{312}}\nonumber\\
&=\hbar \sum_{s \geq 0} {(-1)}^{p(j+1)+p(E_{j,j+1})p(E_{j+1,m+n})}E_{m+n,j+1}(-s) E_{j,m+n}(s)\nonumber\\
&\quad- \hbar \sum_{s \geq 0}{(-1)}^{p(j)+p(E_{j,j+1})p(E_{j,m+n})}E_{m+n,j+1}(-s) E_{j,m+n}(s)\nonumber\\
&=0.\label{super60}
\end{align}
Similarly, by direct computation, we also obtain
\begin{align}
&\phantom{{}={}}\text{the second term of \eqref{313}}+\text{the third term of \eqref{313}}\nonumber\\
&=\hbar \sum_{s \geq 0}{(-1)}^{p(j+1)+p(E_{j,j+1})p(E_{j+1,1})}E_{1,j+1}(-s-1) E_{j,1}(s+1)\nonumber\\
&\quad- \hbar \sum_{s \geq 0}{(-1)}^{p(j)+p(E_{j,j+1})p(E_{j,1})}E_{1,j+1}(-s-1) E_{j,1}(s+1)\nonumber\\
&=0.\label{super61}
\end{align}
Next, we rewrite the third term of \eqref{312}. By direct computation, we have
\begin{align}
&\phantom{{}={}}\text{the third term of \eqref{312}}\nonumber\\
&=\hbar \sum_{s \geq 0} \sum_{k=1}^{m+n-1}\delta_{m+n,j+1}{(-1)}^{p(k)}E_{m+n-1,k}(-s) E_{k,m+n}(s)\nonumber\\
&\quad+\hbar \sum_{s \geq 0} \delta_{m+n,j+1}{(-1)}^{p(m+n)}E_{m+n-1,m+n}(-s-1) E_{m+n,m+n}(s+1)\nonumber\\
&\quad+\hbar\delta_{m+n,j+1}{(-1)}^{p(m+n)}E_{m+n-1,m+n} E_{m+n,m+n}\nonumber\\
&=\delta_{m+n,j+1}(\ev(x_{m+n-1,1}^{+})-(\alpha - (m-n+1) \varepsilon_1)x^+_{m+n-1})\nonumber\\
&\quad+\hbar\delta_{m+n,j+1}{(-1)}^{p(m+n)}E_{m+n-1,m+n} E_{m+n,m+n}.\label{super62}
\end{align}
Similarly, we rewrite the first term of \eqref{313} as follows;
\begin{align}
&\phantom{{}={}}\text{the first term of \eqref{313}}\nonumber\\
&=- \hbar \sum_{s \geq 0} \delta_{1,j}E_{1,1}(-s) E_{1,2}(s)- \hbar \sum_{s \geq 0} \sum_{k=2}^{m+n}\delta_{1,j}{(-1)}^{p(k)}E_{1,k}(-s-1) E_{k,2}(s+1)\nonumber\\
&\quad+ \hbar \sum_{s \geq 0}\delta_{1,j}E_{1,1} E_{1,2}\nonumber\\
&=-\delta_{j,1}(\ev(x_{1,1}^{+})-(\alpha - \varepsilon_1)x^+_{1})+\hbar\delta_{j,1}E_{1,1} E_{1,2}.\label{super63}
\end{align}
Then, by \eqref{r3}, \eqref{super60}, \eqref{super61}, \eqref{super62}, and \eqref{super63}, we can rewrite $[\ev(\widetilde{h}_{0,1}),x^+_{j,0}]$ as
\begin{align}
&(\alpha - (m-n) \varepsilon_1)a_{0,j}x^+_j- \dfrac{\hbar}{2}( -\delta_{m+n,j+1}\{E_{m+n,m+n}, E_{j,m+n}\}+\delta_{1,j}\{E_{1,j+1}, (E_{1,1}-c)\})\nonumber\\
&\quad+\delta_{m+n,j+1}(\ev(x_{m+n-1,1}^{+})-(\alpha - (m-n-1) \varepsilon_1)x^+_{m+n-1})\nonumber\\
&\quad+\hbar\delta_{m+n,j+1}{(-1)}^{p(m+n)}E_{m+n-1,m+n} E_{m+n,m+n}\nonumber\\
&\quad-\delta_{j,1}(\ev(x_{1,1}^{+}-(\alpha - \varepsilon_1)x^+_{1})+\hbar\delta_{j,1}E_{1,1} E_{1,2}.\label{super20}
\end{align}
By \eqref{super20}, when $j\neq 0, 1,m+n-1$, $[\ev(\widetilde{h}_{0,1}),\ev(x^+_{j,0})]$ is equal to zero. When $j=m+n-1$, $[\ev(\widetilde{h}_{0,1}),\ev(x^+_{j,0})]$ is equal to
\begin{align*}
&(\alpha - (m-n) \varepsilon_1)x^+_{m+n-1}+\dfrac{\hbar}{2}\{E_{m+n,m+n}, E_{j,m+n}\}\nonumber\\
&\quad+\ev(x_{m+n-1,1}^{+})-(\alpha - (m-n+1) \varepsilon_1)x^+_{m+n-1}+\hbar{(-1)}^{p(m+n)}E_{m+n-1,m+n} E_{m+n,m+n}.
\end{align*}
Since
\begin{align*}
&\quad\dfrac{\hbar}{2}\{E_{m+n,m+n}, E_{m+n-1,m+n}\}+\hbar{(-1)}^{p(m+n)}E_{m+n-1,m+n} E_{m+n,m+n}\\
&=\dfrac{\hbar}{2}[E_{m+n,m+n}, E_{m+n-1,m+n}]-\dfrac{\hbar}{2}E_{m+n-1,m+n}.
\end{align*}
holds, $[\ev(\widetilde{h}_{0,1}),\ev(x^+_{m+n-1,0})]$ is equal to 
\begin{equation*}
a_{m+n-1,0}(\ev(x^+_{m+n-1,1})+a_{m+n-1,0}\dfrac{\ve_1-\ve_2}{2}E_{m+n-1,m+n}).
\end{equation*}

By \eqref{super20}, when $j=1$, $[\ev(\widetilde{h}_{0,1}),\ev(x^+_{1,0})]$ can be written as
\begin{align*}
&-(\alpha - (m-n) \varepsilon_1)x^+_1- \dfrac{\hbar}{2}\{E_{1,j+1}, (E_{1,1}-c)\}-\ev(x_{1,1}^{+})-(\alpha - \varepsilon_1)x^+_{1}+\hbar E_{1,1} E_{1,2}.
\end{align*}
Since
\begin{align*}
\hbar E_{1,1} E_{1,2}- \dfrac{\hbar}{2}\{E_{1,2}, (E_{1,1}-c)\}&=\dfrac{\hbar}{2}[E_{1,1},E_{1,2}]+\hbar cE_{1,2}=(\dfrac{\hbar}{2}+\hbar c)E_{1,2}
\end{align*}
holds, $[\ev(\widetilde{h}_{0,1}),x^+_{1,0}]=a_{0,1}(\ev(x^+_{1,1})-a_{0,1}\dfrac{\ve_1-\ve_2}{2}\ev(x^+_{1,0}))$ is equivalent to the relation $c\hbar=(m-n)\ve_1$. It is nothing but assumption. This completes the proof of the case $j\neq0$ and $i=0$.

Other cases are proven in the same way. Thus, we show that $[\ev(\widetilde{h}_{i,1}),\ev(x^+_{j,0})]=a_{i,j}(\ev(x^+_{j,1})-b_{i,j}\dfrac{\ve_1-\ve_2}{2}\ev(x^+_{j,0}))$ holds.

\subsection{The proof of \eqref{rel6}}
We only show the case where $i,j\neq0$ and $i=0,j\neq0$. The other case is proven in a similar way.

{\bf Case 1,} $i, j\neq0$

Suppose that $i, j\neq0$. First, we let us compute $[\ev(x^+_{i,1}),\ev(x^+_{j,0})]$. By the definition of $\ev(x^+_{i,1})$, we have
\begin{align}
&\phantom{{}={}}[\ev(x^+_{i,1}),\ev(x^+_{j,0})]\nonumber\\
&=[(\alpha - (i-2\delta(i\geq m+1)(i-m)) \varepsilon_1) x_{i}^{+}, E_{j,j+1}]\nonumber\\
&\quad + [\hbar\sum_{s \geq 0}\sum_{k=1}^i {(-1)}^{p(k)}E_{i,k}(-s) E_{k,i+1}(s), E_{j,j+1}]\nonumber\\
&\quad + [\hbar\sum_{s \geq 0}\sum_{k=i+1}^{m+n} {(-1)}^{p(k)}E_{i,k}(-s-1) E_{k,i+1}(s+1), E_{j,j+1}].\label{r11}
\end{align}
By direct computation, the second term of \eqref{r11} is equal to
\begin{align}
&\phantom{{}={}}[\hbar\sum_{s \geq 0}\sum_{k=1}^i {(-1)}^{p(k)}E_{i,k}(-s) E_{k,i+1}(s), E_{j,j+1}]\nonumber\\
&=\hbar\sum_{s \geq 0}\sum_{k=1}^i\delta_{i+1,j}{(-1)}^{p(k)}E_{i,k}(-s) E_{k,j+1}(s)\nonumber\\
&\quad-\hbar\sum_{s \geq 0}\sum_{k=1}^i\delta_{i,j+1}{(-1)}^{p(k)+p(E_{i+1,i})p(E_{j,j+1})}E_{j,k}(-s) E_{k,i+1}(s)\nonumber\\
&\quad+\hbar\sum_{s \geq 0}\delta_{j,i}{(-1)}^{p(i)+p(E_{i,i+1})p(E_{i,i+1})}E_{i,i+1}(-s) E_{i,i+1}(s).\label{211}
\end{align}
We also find that the third term of \eqref{r11} is equal to
\begin{align}
&\phantom{{}={}}[\hbar\sum_{s \geq 0}\sum_{k=i+1}^{m+n} {(-1)}^{p(k)}E_{i,k}(-s-1) E_{k,i+1}(s+1), E_{j,j+1}]\nonumber\\
&=\hbar\sum_{s \geq 0}\sum_{k=i+1}^{m+n}\delta_{i+1,j}{(-1)}^{p(k)}E_{i,k}(-s-1) E_{k,j+1}(s+1)\nonumber\\
&\quad-\hbar\sum_{s \geq 0}\sum_{k=i+1}^{m+n}\delta_{i,j+1}{(-1)}^{p(k)+p(E_{i+1,i})p(E_{j,j+1})}E_{j,k}(-s-1) E_{k,i+1}(s+1)\nonumber\\
&\quad-\hbar\sum_{s \geq 0}\delta_{i,j}{(-1)}^{p(i+1)+p(E_{i+1,i})p(E_{i,i+1})}E_{i,i+1}(-s-1) E_{i,i+1}(s+1).\label{212}
\end{align}
Thus, we can rewrite $[\ev(x^+_{i,1}),\ev(x^+_{j,0})]$ as
\begin{align}
&[(\alpha - (i-2\delta(i\geq m+1)(i-m)) \varepsilon_1) x_{i}^{+}, E_{j,j+1}]\nonumber\\
&+\hbar\sum_{s \geq 0}\sum_{k=1}^i\delta_{i+1,j}{(-1)}^{p(k)}E_{i,k}(-s) E_{k,j+1}(s)\nonumber\\
&-\hbar\sum_{s \geq 0}\sum_{k=1}^i\delta_{i,j+1}{(-1)}^{p(k)+(p(E_{i+1,k})+p(E_{k,i}))p(E_{j,j+1})}E_{j,k}(-s) E_{k,i+1}(s)\nonumber\\
&+\hbar\sum_{s \geq 0}\sum_{k=i+1}^{m+n}\delta_{i+1,j}{(-1)}^{p(k)}E_{i,k}(-s-1) E_{k,j+1}(s+1)\nonumber\\
&-\hbar\sum_{s \geq 0}\sum_{k=i+1}^{m+n}\delta_{i,j+1}{(-1)}^{p(k)+(p(E_{i+1,k})+p(E_{k,i}))p(E_{j,j+1})}E_{j,k}(-s-1) E_{k,i+1}(s+1)\nonumber\\
&+\hbar\sum_{s \geq 0}\delta_{j,i}{(-1)}^{p(i)+p(E_{i,i+1})p(E_{j,i+1})}E_{i,i+1}(-s) E_{i,i+1}(s)\nonumber\\
&-\hbar\sum_{s \geq 0}\delta_{i,j}{(-1)}^{p(i+1)+p(E_{i+1,i})p(E_{i,i+1})}E_{i,i+1}(-s-1) E_{i,i+1}(s+1).\label{super501}
\end{align}
Next, let us compute $[\ev(x^+_{i,0}),\ev(x^+_{j,1})]$. Since it is equal to 
\begin{equation*}
-{(-1)}^{p(E_{i,i+1})p(E_{j,j+1})}[\ev(x^+_{j,1}),\ev(x^+_{i,0})],
\end{equation*}
we can rewrite $[\ev(x^+_{i,0}),\ev(x^+_{j,1})]$ as
\begin{align}
&[E_{i,i+1},(\alpha - (j-2\delta(j\geq m+1)(j-m)) \varepsilon_1) x_{j}^{+}]\nonumber\\
&\quad-\hbar\sum_{s \geq 0}\sum_{k=1}^j\delta_{i,j+1}{(-1)}^{p(k)+p(E_{i,i+1})(p(E_{j,k})+p(E_{k,j+1}))}E_{j,k}(-s) E_{k,i+1}(s)\nonumber\\
&\quad+\hbar\sum_{s \geq 0}\sum_{k=1}^j\delta_{i+1,j}{(-1)}^{p(k)}E_{i,k}(-s) E_{k,j+1}(s)\nonumber\\
&\quad-\hbar\sum_{s \geq 0}\sum_{k=j+1}^{m+n}\delta_{i,j+1}{(-1)}^{p(k)+p(E_{i,i+1})(p(E_{j,k})+p(E_{k,j+1}))}E_{j,k}(-s-1) E_{k,i+1}(s+1)\nonumber\\
&\quad+\hbar\sum_{s \geq 0}\sum_{k=j+1}^{m+n}\delta_{i+1,j}{(-1)}^{p(k)}E_{i,k}(-s-1) E_{k,j+1}(s+1)-\hbar\sum_{s \geq 0}\delta_{i,j}{(-1)}^{p(i)}E_{i,i+1}(-s) E_{i,i+1}(s)\nonumber\\
&\quad+\hbar \sum_{s \geq 0}\delta_{i,j}{(-1)}^{p(i+1)}E_{i,i+1}(-s-1) E_{i,i+1}(s+1).\label{super500}
\end{align}
By \eqref{super501} and \eqref{super500}, when $i\neq j, j\pm1$, $[\ev(x^+_{i,1}),\ev(x^+_{j,0})]-[\ev(x^+_{i,0}),\ev(x^+_{j,1})]$ is equal to zero.

When $i=j$, $[\ev(x^+_{i,1}),\ev(x^+_{j,0})]-[\ev(x^+_{i,0}),\ev(x^+_{j,1})]$ is equal to
\begin{align}
&[(\alpha - (i-2\delta(i\geq m+1)(i-m)) \varepsilon_1) x_{i}^{+}, E_{j,j+1}]\nonumber\\
&\quad-[E_{i,i+1},(\alpha - (j-2\delta(j\geq m+1)(j-m)) \varepsilon_1) x_{j}^{+}]\nonumber\\
&\quad+\hbar\sum_{s \geq 0}{(-1)}^{p(i+1)}E_{i,i+1}(-s) E_{i,i+1}(s)-\hbar\sum_{s \geq 0}{(-1)}^{p(i)}E_{i,i+1}(-s-1) E_{i,i+1}(s+1)\nonumber\\
&\quad+\hbar\sum_{s \geq 0}{(-1)}^{p(i)}E_{i,i+1}(-s) E_{i,i+1}(s)-\hbar\sum_{s \geq 0}{(-1)}^{p(i+1)}E_{i,i+1}(-s-1) E_{i,i+1}(s+1).\label{super1024}
\end{align}
Since $[x^+_i, x^+_i]=0$ holds, the first and second term are zero. We also obtain
\begin{align*}
&\phantom{{}={}}\text{the third term of \eqref{super1024}}+\text{the 4-th term of \eqref{super1024}}=\hbar{(-1)}^{p(i)}E_{i,i+1} E_{i,i+1}
\end{align*}
and
\begin{align*}
&\phantom{{}={}}\text{the 5-th term of \eqref{super1024}}+\text{the 6-th term of \eqref{super1024}}=\hbar{(-1)}^{p(i+1)}E_{i,i+1} E_{i,i+1}
\end{align*}
by direct computation. Thus, $[\ev(x^+_{i,1}),\ev(x^+_{j,0})]-[\ev(x^+_{i,0}),\ev(x^+_{j,1})]$ is equal to $\hbar a_{i,i}E_{i,i+1} E_{i,i+1}$ since $a_{i,i}={(-1)}^{p(i)}+{(-1)}^{p(i+1)}$ holds.

When $i=j-1$, $[\ev(x^+_{i,1}),\ev(x^+_{j,0})]-[\ev(x^+_{i,0}),\ev(x^+_{j,1})]$ is equal to
\begin{align}
&[(\alpha - (i-2\delta(i\geq m+1)(i-m)) \varepsilon_1) x_{i}^{+}, E_{j,j+1}]-[E_{i,i+1},(\alpha - (j-2\delta(j\geq m+1)(j-m)) \varepsilon_1) x_{j}^{+}]\nonumber\\
&\quad+\hbar \sum_{s \geq 0}\sum_{k=1}^i\delta_{i+1,j}{(-1)}^{p(k)}E_{i,k}(-s) E_{k,j+1}(s)\nonumber\\
&\quad+\hbar\sum_{s \geq 0}\sum_{k=i+1}^{m+n}\delta_{i+1,j}{(-1)}^{p(k)}E_{i,k}(-s-1) E_{k,j+1}(s+1)\nonumber\\
&\quad-\hbar\sum_{s \geq 0}\sum_{k=1}^j\delta_{i+1,j}{(-1)}^{p(k)}E_{i,k}(-s) E_{k,j+1}(s)\nonumber\\
&\quad-\hbar\sum_{s \geq 0}\sum_{k=j+1}^{m+n}\delta_{i+1,j}{(-1)}^{p(k)}E_{i,k}(-s-1) E_{k,j+1}(s+1).\label{super1025}
\end{align}
By direct computation, we obtain
\begin{align*}
\text{the third term of \eqref{super1025}}+\text{the 5-th term of \eqref{super1025}}=-\hbar\sum_{s \geq 0}{(-1)}^{p(i+1)}E_{i,i+1}(-s) E_{i+1,i+2}(s)
\end{align*}
and
\begin{align*}
\text{the 4-th term of \eqref{super1024}}+\text{the 6-th term of \eqref{super1024}}=\hbar\sum_{s \geq 0}{(-1)}^{p(i+1)}E_{i,i+1}(-s-1) E_{i+1,i+2}(s+1).
\end{align*}
Then, $[\ev(x^+_{i,1}),\ev(x^+_{j,0})]-[\ev(x^+_{i,0}),\ev(x^+_{j,1})]$ is equal to
\begin{align*}
&[(\alpha - (i-2\delta(i\geq m+1)(i-m)) \varepsilon_1) x_{i}^{+}, E_{j,j+1}]-[E_{i,i+1},(\alpha - (j-2\delta(j\geq m+1)(j-m)) \varepsilon_1) x_{j}^{+}]\nonumber\\
&-\hbar\sum_{s \geq 0}{(-1)}^{p(i+1)}E_{i,i+1}(-s) E_{i+1,i+2}(s)+\hbar\sum_{s \geq 0}{(-1)}^{p(i+1)}E_{i,i+1}(-s-1) E_{i+1,i+2}(s+1).
\end{align*}
Since $a_{i,i+1}=-{(-1)}^{p(i+1)}$ holds, we have
\begin{align*}
&-\hbar\sum_{s \geq 0}{(-1)}^{p(i+1)}E_{i,i+1}(-s) E_{i+1,i+2}(s)+\hbar\sum_{s \geq 0}{(-1)}^{p(i+1)}E_{i,i+1}(-s-1) E_{i+1,i+2}(s+1)\\
&=-{(-1)}^{p(i+1)}\hbar E_{i,i+1} E_{i+1,i+2}=a_{i,i+1}\dfrac{\hbar}{2}\{E_{i,i+1},E_{i+1,i+2}\}+a_{i,i+1}\dfrac{\hbar}{2}[E_{i,i+1},E_{i+1,i+2}]
\end{align*}
and
\begin{align*}
&[(\alpha - (i-2\delta(i\geq m+1)(i-m)) \varepsilon_1) x_{i}^{+}, E_{j,j+1}]-[E_{i,i+1},(\alpha - (j-2\delta(j\geq m+1)(j-m)) \varepsilon_1) x_{j}^{+}]\\
&\qquad\qquad={(-1)}^{p(i+1)}\ve_1[E_{i,i+1},E_{i+1,i+2}]=-a_{i,i+1}\ve_1[E_{i,i+1},E_{i+1,i+2}].
\end{align*}
Then, $[\ev(x^+_{i,1}),\ev(x^+_{j,0})]-[\ev(x^+_{i,0}),\ev(x^+_{j,1})]$ is equal to 
\begin{equation*}
a_{i,i+1}\dfrac{\hbar}{2}\{E_{i,i+1},E_{i+1,i+2}\}-a_{i,i+1}\dfrac{\ve_1-\ve_2}{2}[E_{i,i+1},E_{i+1,i+2}].
\end{equation*}
When $i=j+1$, $[\ev(x^+_{i,1}),\ev(x^+_{j,0})]-[\ev(x^+_{i,0}),\ev(x^+_{j,1})]$ is equal to
\begin{align}
&[(\alpha - (i-2\delta(i\geq m+1)(i-m)) \varepsilon_1) x_{i}^{+}, E_{j,j+1}]-[E_{i,i+1},(\alpha - (j-2\delta(j\geq m+1)(j-m)) \varepsilon_1) x_{j}^{+}]\nonumber\\
&\quad-\hbar\sum_{s \geq 0}\sum_{k=1}^i\delta_{i,j+1}{(-1)}^{p(k)+p(E_{i+1,i})p(E_{j,j+1})}E_{j,k}(-s) E_{k,i+1}(s)\nonumber\\
&\quad-\hbar\sum_{s \geq 0}\sum_{k=i+1}^{m+n}\delta_{i,j+1}{(-1)}^{p(k)+p(E_{i+1,i})p(E_{j,j+1})}E_{j,k}(-s-1) E_{k,i+1}(s+1)\nonumber\\
&\quad+\hbar\sum_{s \geq 0}\sum_{k=1}^j\delta_{i,j+1}{(-1)}^{p(k)+p(E_{i+1,i})p(E_{j,j+1})}E_{j,k}(-s) E_{k,i+1}(s)\nonumber\\
&\quad+\hbar\sum_{s \geq 0}\sum_{k=j+1}^{m+n}\delta_{i,j+1}{(-1)}^{p(k)+p(E_{i+1,i})p(E_{j,j+1})}E_{j,k}(-s-1) E_{k,i+1}(s+1).\label{super1026}
\end{align}
By direct computation, we find that
\begin{align*}
\text{the 4-th term of \eqref{super1024}}+\text{the 6-th term of \eqref{super1024}}=-\hbar\sum_{s \geq 0}{(-1)}^{p(i)}E_{i-1,i}(-s) E_{i,i+1}(s)
\end{align*}
and
\begin{align*}
\text{the 4-th term of \eqref{super1024}}+\text{the 6-th term of \eqref{super1024}}=\hbar\sum_{s \geq 0}{(-1)}^{p(i)}E_{i-1,i}(-s-1) E_{i,i+1}(s+1).
\end{align*}
hold. Since $a_{i,i-1}=-{(-1)}^{p(i)}$ holds, we have
\begin{align*}
&-\hbar\sum_{s \geq 0}{(-1)}^{p(i)}E_{i-1,i}(-s) E_{i,i+1}(s)+\hbar\sum_{s \geq 0}{(-1)}^{p(i)}E_{i-1,i}(-s-1) E_{i,i+1}(s+1)\\
&=-\hbar{(-1)}^{p(i)}E_{i-1,i} E_{i,i+1}=\dfrac{\hbar}{2} a_{i-1,i}\{E_{i,i+1},E_{i-1,i}\}-\dfrac{\hbar}{2} a_{i-1,i}[E_{i,i+1},E_{i-1,i}]
\end{align*}
and
\begin{align*}
&[(\alpha - (i-2\delta(i\geq m+1)(i-m)) \varepsilon_1) x_{i}^{+}, E_{j,j+1}]-[E_{i,i+1},(\alpha - (j-2\delta(j\geq m+1)(j-m)) \varepsilon_1) x_{j}^{+}]\nonumber\\
&\qquad\qquad=-{(-1)}^{p(i)}\ve_1[E_{i,i+1},E_{i-1,i}]=a_{i,i-1}\ve_1[E_{i,i+1},E_{i-1,i}].
\end{align*}
holds, $[\ev(x^+_{i,1}),\ev(x^+_{j,0})]-[\ev(x^+_{i,0}),\ev(x^+_{j,1})]$ is equal to
\begin{align*}
&[(\alpha - (i-2\delta(i\geq m+1)(i-m)) \varepsilon_1) x_{i}^{+}, E_{j,j+1}]-[E_{i,i+1},(\alpha - (j-2\delta(j\geq m+1)(j-m)) \varepsilon_1) x_{j}^{+}]\nonumber\\
&-\hbar{(-1)}^{p(i)}\{E_{i-1,i}, E_{i,i+1}\}+{(-1)}^{p(i)}\dfrac{\hbar}{2}E_{i-1,i+1}.
\end{align*}
Therefore, it is equal to $-a_{i,i-1}\dfrac{\hbar}{2}\{E_{i-1,i},E_{i,i+1}\}+a_{i,i-1}\dfrac{\ve_1-\ve_2}{2} E_{i-1,i+1}$.

{\bf Case 2,} $i\neq0$ and $j=0$

Suppose that $i\neq0$. First, we compute $[\ev(x^+_{i,1}),\ev(x^+_{0,0})]$. By the definition of $\ev$, we obtain
\begin{align}
&\phantom{{}={}}[\ev(x^+_{i,1}),\ev(x^+_{0,0})]\nonumber\\
&=[(\alpha - (i-2\delta(i\geq m+1)(i-m))\ve_1) x_{i}^{+}, E_{m+n,1}(1)]\nonumber\\
&\quad + [\hbar\sum_{s \geq 0}\sum_{k=1}^i {(-1)}^{p(k)}E_{i,k}(-s) E_{k,i+1}(s), E_{m+n,1}(1)]\nonumber\\
&\quad + [\hbar\sum_{s \geq 0}\sum_{k=i+1}^{m+n} {(-1)}^{p(k)}E_{i,k}(-s-1) E_{k,i+1}(s+1), E_{m+n,1}(1)].\label{super100}
\end{align}
By direct computation, the second term of \eqref{super100} is equal to
\begin{align}
&\hbar\sum_{s \geq 0}\sum_{k=1}^{m+n-1}\delta_{m+n,i+1}{(-1)}^{p(k)}E_{m+n-1,k}(-s) E_{k,1}(s+1)\nonumber\\
&\quad-\hbar\sum_{s \geq 0}{(-1)}^{p(E_{m+n,1})p(E_{1,i})+p(1)}E_{i,1}(-s) E_{m+n,i}(s+1)\nonumber\\
&\quad-\hbar \sum_{s \geq 0}\delta_{1,i}E_{m+n,1}(1-s) E_{1,2}(s)\label{611}
\end{align}
and the third term of \eqref{super100} is equal to
\begin{align}
&\hbar\sum_{s \geq 0}\sum_{k=i+1}^{m+n}{(-1)}^{p(m+n)}E_{m+n-1,m+n}(-s-1) E_{m+n,1}(s+2)\nonumber\\
&\quad+\hbar\sum_{s \geq 0}{(-1)}^{p(E_{m+n,1})p(E_{m+n,i})+p(m+n)}E_{i,1}(-s) E_{m+n,i}(s+1)\nonumber\\
&\quad-\hbar\sum_{s \geq 0}\sum_{k=2}^{m+n}\delta_{1,i}{(-1)}^{p(k)}E_{m+n,k}(-s) E_{k,1}(s+1)+\delta_{i,1}cE_{m+n,2}(1).\label{612}
\end{align}
Next, we rewrite the sum of the second term of \eqref{611} and the second term of \eqref{612} as follows;
\begin{align*}
&\phantom{{}={}}\text{the second term of \eqref{611}}+\text{the second term of \eqref{612}}=0.
\end{align*}
Therefore, $[\ev(x^+_{i,1}),\ev(x^+_{0,0})]$ is equal to
\begin{align}
&[(\alpha - i \varepsilon_1) x_{i}^{+}, E_{m+n,1}(1)]+\hbar\sum_{s \geq 0}\sum_{k=1}^{m+n-1}\delta_{m+n,i+1}{(-1)}^{p(k)}E_{m+n-1,k}(-s) E_{k,1}(s+1)\nonumber\\
&\quad-\hbar\sum_{s \geq 0}\delta_{1,i}E_{m+n,1}(1-s) E_{1,2}(s)\nonumber\\
&\quad+\hbar\sum_{s \geq 0}\delta_{m+n,i+1}{(-1)}^{p(m+n)}E_{m+n-1,m+n}(-s-1) E_{m+n,1}(s+2)\nonumber\\
&\quad-\hbar\sum_{s \geq 0}\sum_{k=2}^{m+n}\delta_{1,i}{(-1)}^{p(k)}E_{m+n,k}(-s) E_{k,1}(s+1)+\delta_{i,1}cE_{m+n,2}(1).\label{super700}
\end{align}

Next, let us compute $[\ev(x^+_{i,0}),\ev(x^+_{0,1})]$. By direct computation, we have
\begin{align}
&\phantom{{}={}}[\ev(x^+_{i,0}),\ev(x^+_{0,1})]\nonumber\\
&=[E_{i,i+1},(\alpha - (m-n) \varepsilon_1) x_{0}^{+}]+[E_{i,i+1}, \hbar \sum_{s \geq 0} \sum_{k=1}^{m+n} {(-1)}^{p(k)}E_{m+n,k}(-s) E_{k,1}(s+1)].\label{super800}
\end{align}
By direct computation, the second term of \eqref{super800} is equal to
\begin{align}
&\hbar \sum_{s \geq 0} \sum_{k=1}^{m+n}\delta_{m+n,i+1}{(-1)}^{p(k)}E_{m+n-1,k}(-s) E_{k,1}(s+1)\nonumber\\
&\quad-\hbar \sum_{s \geq 0} {(-1)}^{p(i)+p(E_{i,i+1})p(E_{m+n,i})}E_{m+n,i+1}(-s) E_{i,1}(s+1)\nonumber\\
&\quad+\hbar \sum_{s \geq 0}{(-1)}^{p(i+1)+p(E_{i,i+1})p(E_{m+n,i+1})}E_{m+n,i+1}(-s) E_{i,1}(s+1)\nonumber\\
&\quad-\hbar \sum_{s \geq 0} \sum_{k=1}^{m+n}\delta_{1,i}{(-1)}^{p(k)+p(E_{1,2})p(E_{m+n,1})}E_{m+n,k}(-s) E_{k,i+1}(s+1).\label{711}
\end{align}
The sum of the second term of \eqref{711} and the third term of \eqref{711} is equal to zero.
Thus, $[\ev(x^+_{i,0}),\ev(x^+_{0,1})]$ is equal to
\begin{align}
&[E_{i,i+1},(\alpha -(m-n) \varepsilon_1) x_{0}^{+}]+ \hbar \sum_{s \geq 0} \sum_{k=1}^{m+n}\delta_{m+n,i+1}{(-1)}^{p(k)}E_{m+n-1,k}(-s) E_{k,1}(s+1)\nonumber\\
&\quad-\hbar \sum_{s \geq 0} \sum_{k=1}^{m+n}\delta_{1,i}{(-1)}^{p(k)+p(E_{1,2})p(E_{m+n,1})}E_{m+n,k}(-s) E_{k,2}(s+1).\label{super701}
\end{align}

Therefore, when $i\neq0, 1, m+n-1$, $[\ev(x^+_{i,1}),\ev(x^+_{0,0})]-[\ev(x^+_{i,0}),\ev(x^+_{0,1})]$ is zero.
When $i=1$, $[\ev(x^+_{1,1}),\ev(x^+_{0,0})]-[\ev(x^+_{1,0}),\ev(x^+_{0,1})]$ is equal to
\begin{align}
&[(\alpha - \ve_1) x_{1}^{+}, E_{m+n,1}(1)]-[E_{1,2},(\alpha -(m-n) \varepsilon_1) x_{0}^{+}]\nonumber\\
&\quad-\hbar\sum_{s \geq 0}E_{m+n,1}(1-s) E_{1,2}(s)-\hbar\sum_{s \geq 0}\sum_{k=2}^{m+n}{(-1)}^{p(k)}E_{m+n,k}(-s) E_{k,2}(s+1)\nonumber\\
&\quad+cE_{m+n,2}(1)+\hbar \sum_{s \geq 0} \sum_{k=1}^{m+n}{(-1)}^{p(k)}E_{m+n,k}(-s) E_{k,2}(s+1).\label{2050}
\end{align}
By direct computation, we obtain
\begin{align*}
&\phantom{{}={}}\text{the third term of \eqref{2050}}+\text{the 4-th term of \eqref{2050}}+\text{the 6-th term of \eqref{2050}}\\
&=-\hbar E_{m+n,1}(1) E_{1,2}(0)=-\dfrac{\hbar}{2}\{E_{1,2}(0),E_{m+n,1}(1)\}+\dfrac{\hbar}{2}[E_{1,2}(0),E_{m+n,1}(1)].
\end{align*}
Moreover, by direct computation, we obtain
\begin{align*}
[(\alpha - \ve_1 )x_{1}^{+}, E_{m+n,1}(1)]-[E_{1,2},(\alpha -(m-n) \varepsilon_1 )x_{0}^{+}]=(m-n-1)\ve_1[x_{1}^{+}, E_{m+n,1}(1)].
\end{align*}
Therefore, $[\ev(x^+_{1,1}),\ev(x^+_{0,0})]-[\ev(x^+_{1,0}),\ev(x^+_{0,1})]$ is equal to
\begin{align*}
-\dfrac{\hbar}{2}\{E_{1,2}(0),E_{m+n,1}(1)\}-\dfrac{\ve_1-\ve_2}{2}[x_{1}^{+}, E_{m+n,1}(1)].
\end{align*}
by the assumption $\hbar c=(m-n)\varepsilon_1$.

When $i=m+n-1$, $[\ev(x^+_{m+n-1,1}),\ev(x^+_{0,0})]-[\ev(x^+_{m+n-1,0}),\ev(x^+_{0,1})]$ is equal to
\begin{align*}
&[(\alpha - (m-n+1)\ve_1)x_{m+n-1}^{+}, E_{m+n,1}(1)]-[E_{m+n-1,m+n},(\alpha -(m-n) \varepsilon_1) x_{0}^{+}]\nonumber\\
&+\hbar\sum_{s \geq 0}\sum_{k=1}^{m+n-1}{(-1)}^{p(k)}E_{m+n-1,k}(-s) E_{k,1}(s+1)\nonumber\\
&+\hbar\sum_{s \geq 0}\sum_{k=m+n}^{m+n}{(-1)}^{p(k)}E_{m+n-1,k}(-s-1) E_{k,1}(s+2)\nonumber\\
&-\hbar \sum_{s \geq 0} \sum_{k=1}^{m+n}{(-1)}^{p(k)}E_{m+n-1,k}(-s) E_{k,1}(s+1).\label{2051}
\end{align*}
By direct computation, we obtain
\begin{align*}
&\phantom{{}={}}\text{the third term of \eqref{2050}}+\text{the 4-th term of \eqref{2050}}+\text{the 5-th term of \eqref{2050}}\\
&=\hbar E_{m+n-1,m+n}(0) E_{m+n,1}(1)\\
&=\dfrac{\hbar}{2}\{E_{m+n-1,m+n}(0), E_{m+n,1}(1)\}+\dfrac{\hbar}{2}[E_{m+n-1,m+n}(0), E_{m+n,1}(1)].
\end{align*}
Moreover, by direct computation, we have
\begin{align*}
[(\alpha - (m-n+1)\ve_1)x_{i}^{+}, E_{m+n,1}(1)]-[E_{i,i+1},(\alpha -(m-n) \varepsilon_1) x_{0}^{+}]=-\ve_1[x^+_i,x^+_0]
\end{align*}

Then, $[\ev(x^+_{m+n-1,1}),\ev(x^+_{0,0})]-[\ev(x^+_{m+n-1,0}),\ev(x^+_{0,1})]$ is equal to
\begin{align*}
\dfrac{\hbar}{2}\{E_{m+n-1,m+n}(0), E_{m+n,1}(1)\}-\dfrac{\ve_1-\ve_2}{2}[x^+_{m+n-1},x^+_0].
\end{align*}

This completes the proof of \eqref{rel6}. 

\subsection{The proof of \eqref{rel7}}

Finally, we show $[\ev(h_{i,1}),\ev(h_{j,1})]=0$. Suppose that $i, j\neq0$. It is enough to show the case where $i<j$. We set
\begin{gather*}
A_i=\sum_{s \geq 0}\sum_{k=1}^{i} {(-1)}^{p(k)}E_{i,k}(-s) E_{k,i}(s),\quad B_i=\sum_{s \geq 0}\sum_{k=i+1}^{m+n} {(-1)}^{p(k)}E_{i,k}(-s-1) E_{k,i}(s+1),\\
C_i=\sum_{s \geq 0}\sum_{k=1}^{i} {(-1)}^{p(k)}E_{i+1,k}(-s) E_{k,i+1}(s),\quad D_i=\sum_{s \geq 0}\sum_{k=i+1}^{m+n} {(-1)}^{p(k)}E_{i+1,k}(-s-1) E_{k,i+1}(s+1).
\end{gather*}
Then, by the definition of $\ev(h_{i,1})$, we have
\begin{align*}
&\quad[\ev(h_{i,1}),\ev(h_{j,1})]\\
&={(-1)}^{p(i)+p(j)}\{[A_i, A_j]+[B_i,A_j]+[B_i,B_j]+[A_i,B_j]\}\\
&\quad+{(-1)}^{p(i)+p(j+1)}\{[A_i, C_j]+[B_i,C_j]+[B_i,D_j]+[A_i,D_j]\}\\
&\quad+{(-1)}^{p(i+1)+p(j)}\{[C_i,A_j]+[D_i,B_j]+[D_i,A_j]+[C_i,B_j]\}\\
&\quad+{(-1)}^{p(i+1)+p(j+1)}\{[C_i,C_j]+[D_i,C_j]+[D_i,D_j]+[C_i,D_j]\}.
\end{align*}
By the definition of $A_i$, $B_i$, $C_i$, and $D_i$, we obtain
\begin{gather*}
[A_i,B_j]=[A_i,D_j]=0.
\end{gather*}
Thus, it is enough to show the following lemma.
\begin{Lemma}
The following relations hold;
\begin{gather*}
[A_i, A_j]+[B_i,A_j]+[B_i,B_j]=0,\\
[A_i, C_j]+[B_i,C_j]+[B_i,D_j]+[A_i,D_j]=0,\\
[C_i,A_j]+[D_i,B_j]+[D_i,B_j]+[C_i,B_j]=0,\\
[C_i,C_j]+[D_i,C_j]+[D_i,D_j]=0.
\end{gather*}
\end{Lemma}
\begin{proof}
We only show that $[A_i, A_j]+[B_i,A_j]+[B_i,B_j]=0$ holds. Other relations are obtained in the same way. By direct computation, we can rewrite $[A_i, A_j]$ as follows;
\begin{align}
&\phantom{{}={}}[A_i, A_j]\nonumber\\
&=-\sum_{s,t\geq 0}\sum_{k=1}^{i}{(-1)}^{p(k)+p(i)+p(E_{i,k})p(E_{j,i})}E_{j,k}(-s-t) E_{i,j}(t)E_{k,i}(s)\nonumber\\
&\quad+\sum_{s,t\geq 0}\sum_{k=1}^{i}\sum_{l=1}^{j}\delta_{k,l}{(-1)}^{p(E_{i,k})p(E_{j,k})}E_{j,k}(-t) E_{i,j}(-s+t)E_{k,i}(s)\nonumber\\
&\quad-\sum_{s,t\geq 0}\sum_{k=1}^{i}\sum_{l=1}^{j}\delta_{k,l}{(-1)}^{p(E_{k,i})p(E_{j,k})}E_{i,k}(-s)E_{j,i}(s-t) E_{k,j}(t)\nonumber\\
&\quad+\sum_{s,t\geq 0}\sum_{k=1}^{i}{(-1)}^{p(k)+p(i)+p(E_{k,i})p(E_{i,j})}E_{i,k}(-s)E_{j,i}(-t) E_{k,j}(s+t)\nonumber\\
&\quad+\sum_{s\geq0}(sE_{i,i}(-s)E_{j,j}(s)-sE_{j,j}(-s)E_{i,i}(s)).\label{aij}
\end{align}
Since we find two relations
\begin{align*}
&\quad\text{the second term of \eqref{aij}}\\
&=\sum_{s,t\geq 0}\sum_{k=1}^{i}\sum_{l=1}^{j}\delta_{k,l}{(-1)}^{p(E_{i,k})p(E_{j,k})}E_{j,k}(-s-t) E_{i,j}(t)E_{k,i}(s)\\
&\quad+\sum_{s,t\geq 0}\sum_{k=1}^{i}\sum_{l=1}^{j}\delta_{k,l}{(-1)}^{p(E_{i,k})p(E_{j,k})}E_{j,k}(-s) E_{i,j}(-t-1)E_{k,i}(s+t+1),\\
&\quad\text{the third term of \eqref{aij}}\\
&=\sum_{s,t\geq 0}\sum_{k=1}^{i}\sum_{l=1}^{j}\delta_{k,l}{(-1)}^{p(E_{k,i})p(E_{j,l})}E_{i,k}(-s-t-1)E_{j,i}(s+1) E_{k,j}(t)\\
&\quad+\sum_{s,t\geq 0}\sum_{k=1}^{i}\sum_{l=1}^{j}\delta_{k,l}{(-1)}^{p(E_{k,i})p(E_{k,j})}E_{i,k}(-s)E_{j,i}(-t) E_{k,j}(s+t),
\end{align*}
we have
\begin{align}
&\phantom{{}={}}[A_i, A_j]\nonumber\\
&=-\sum_{s,t\geq 0}\sum_{k=1}^{i}{(-1)}^{p(k)+p(i)+p(E_{i,k})p(E_{j,i})}E_{j,k}(-s-t) E_{i,j}(t)E_{k,i}(s)\nonumber\\
&\quad+\sum_{s,t\geq 0}\sum_{k=1}^{i}\sum_{l=1}^{j}\delta_{k,l}{(-1)}^{p(E_{i,k})p(E_{j,k})}E_{j,k}(-s-t) E_{i,j}(t)E_{k,i}(s)\nonumber\\
&\quad+\sum_{s,t\geq 0}\sum_{k=1}^{i}\sum_{l=1}^{j}\delta_{k,l}{(-1)}^{p(E_{i,k})p(E_{j,k})}E_{j,k}(-s) E_{i,j}(-t-1)E_{k,i}(s+t+1)\nonumber\\
&\quad-\sum_{s,t\geq 0}\sum_{k=1}^{i}\sum_{l=1}^{j}\delta_{k,l}{(-1)}^{p(E_{k,i})p(E_{j,l})}E_{i,k}(-s-t-1)E_{j,i}(s+1) E_{k,j}(t)\nonumber\\
&\quad-\sum_{s,t\geq 0}\sum_{k=1}^{i}\sum_{l=1}^{j}\delta_{k,l}{(-1)}^{p(E_{k,i})p(E_{k,j})}E_{i,k}(-s)E_{j,i}(-t) E_{k,j}(s+t)\nonumber\\
&\quad+\sum_{s,t\geq 0}\sum_{k=1}^{i}{(-1)}^{p(k)+p(i)+p(E_{k,i})p(E_{i,j})}E_{i,k}(-s)E_{j,i}(-t) E_{k,j}(s+t)\nonumber\\
&\quad+\sum_{s\geq0}(sE_{i,i}(-s)E_{j,j}(s)-sE_{j,j}(-s)E_{i,i}(s)).\label{aij1}
\end{align}
We simplify the right hand side of \eqref{aij1}. By direct computation, we obtain
\begin{align}
&\quad\text{the first term of \eqref{aij1}}+\text{the second term of \eqref{aij1}}\nonumber\\
&=-\sum_{s,t\geq 0}\sum_{k=1}^{i}{(-1)}^{p(k)+p(i)+p(E_{i,k})p(E_{j,i})}E_{j,k}(-s-t) E_{i,j}(t)E_{k,i}(s)\nonumber\\
&\quad+\sum_{s,t\geq 0}\sum_{k=1}^{i}{(-1)}^{p(E_{i,k})p(E_{j,k})}E_{j,k}(-s-t) E_{i,j}(t)E_{k,i}(s)\nonumber\\
&=0\label{aij2}
\end{align}
since $p(k)+p(i)+p(E_{i,k})p(E_{j,i})=p(E_{i,k})p(E_{j,k})$.
Similarly, we have
\begin{align}
\text{the 4-th term of \eqref{aij1}}+\text{the 6-th term of \eqref{aij1}}=0.\label{aij3}
\end{align}
By \eqref{aij2} and \eqref{aij3}, we find the equality
\begin{align*}
&\phantom{{}={}}[A_i, A_j]\\
&=\sum_{s,t\geq 0}\sum_{k=1}^{i}\sum_{l=1}^{j}\delta_{k,l}{(-1)}^{p(k)+p(l)+p(E_{i,k})p(E_{j,l})}E_{j,l}(-s) E_{i,j}(-t-1)E_{k,i}(s+t+1)\\
&\quad-\sum_{s,t\geq 0}\sum_{k=1}^{i}\sum_{l=1}^{j}\delta_{k,l}{(-1)}^{p(k)+p(l)+p(E_{k,i})p(E_{j,l})}E_{i,k}(-s-t-1)E_{j,i}(s+1) E_{l,j}(t)\nonumber\\
&\quad+\sum_{s\geq0}(sE_{i,i}(-s)E_{j,j}(s)-sE_{j,j}(-s)E_{i,i}(s)).
\end{align*}
Computing the parity, we obtain
\begin{align}
&\phantom{{}={}}[A_i, A_j]\nonumber\\
&=\sum_{s,t\geq 0}\sum_{k=1}^{i}{(-1)}^{p(E_{i,k})p(E_{j,k})}E_{j,k}(-s) E_{i,j}(-t-1)E_{k,i}(s+t+1)\nonumber\\
&\quad-\sum_{s,t\geq 0}\sum_{k=1}^{i}{(-1)}^{p(E_{k,i})p(E_{j,k})}E_{i,k}(-s-t-1)E_{j,i}(s+1) E_{k,j}(t)\nonumber\\
&\quad+\sum_{s\geq0}(sE_{i,i}(-s)E_{j,j}(s)-sE_{j,j}(-s)E_{i,i}(s)).\label{aij4}
\end{align}
Similarly, by direct computation, we have
\begin{align}
&\phantom{{}={}}[B_i,B_j]\nonumber\\
&=\sum_{s,t \geq 0}\sum_{l=j+1}^{m+n}{(-1)}^{p(j)+p(l)+p(E_{j,l})p(E_{j,i})}E_{i,l}(-s-t-2)E_{j,i}(s+1)E_{l,j}(t+1)\nonumber\\
&\quad-\sum_{s,t \geq 0}\sum_{k=i+1}^{m+n}\sum_{l=j+1}^{m+n}\delta_{k,l}{(-1)}^{p(E_{j,k})p(E_{k,i})}E_{i,k}(-s-t-2)E_{j,i}(s+1)E_{k,j}(t+1)\nonumber\\
&\quad-\sum_{s,t \geq 0}\sum_{k=i+1}^{m+n}\sum_{l=j+1}^{m+n}\delta_{k,l}{(-1)}^{p(E_{j,k})p(E_{k,i})}E_{i,k}(-s-1)E_{j,i}(-t)E_{k,j}(s+t+1)\nonumber\\
&\quad+\sum_{s,t \geq 0}\sum_{k=i+1}^{m+n}\sum_{l=j+1}^{m+n}\delta_{k,l}{(-1)}^{p(E_{k,i})p(E_{k,j})}E_{j,l}(-s-t-1)E_{i,j}(t)E_{k,i}(s+1)\nonumber\\
&\quad+\sum_{s,t \geq 0}\sum_{k=i+1}^{m+n}\sum_{l=j+1}^{m+n}\delta_{k,l}{(-1)}^{p(E_{k,i})p(E_{k,j})}E_{j,k}(-t-1)E_{i,j}(-s-1)E_{k,i}(s+t+2)\nonumber\\
&\quad-\sum_{s,t \geq 0}\sum_{l=j+1}^{m+n}{(-1)}^{p(j)+p(l)+p(E_{j,i})p(E_{l,j})}E_{j,l}(-t-1)E_{i,j}(-s-1)E_{l.i}(s+t+2).\label{bij1}
\end{align}
We simplify the right hand side of \eqref{bij1}. By direct computation, we obtain
\begin{align}
\text{the first term of \eqref{bij1}}+\text{the second term of \eqref{bij1}}=0\label{bij2}
\end{align}
and
\begin{align}
\text{the 5-th term of \eqref{bij1}}+\text{the 6-th term of \eqref{bij1}}=0\label{bij3}
\end{align}
By \eqref{bij2} and \eqref{bij3}, we find the equality
\begin{align}
&\phantom{{}={}}[B_i,B_j]\nonumber\\
&=-\sum_{s,t \geq 0}\sum_{k=i+1}^{m+n}\sum_{l=j+1}^{m+n}\delta_{k,l}{(-1)}^{p(E_{j,k})p(E_{k,i})}E_{i,k}(-s-1)E_{j,i}(-t)E_{k,j}(s+t+1)\nonumber\\
&\quad+\sum_{s,t \geq 0}\sum_{k=i+1}^{m+n}\sum_{l=j+1}^{m+n}\delta_{k,l}{(-1)}^{p(E_{k,i})p(E_{k,j})}E_{j,k}(-s-t-1)E_{i,j}(t)E_{k,i}(s+1)\nonumber\\
&=-\sum_{s,t \geq 0}\sum_{l=j+1}^{m+n}{(-1)}^{p(E_{j,l})p(E_{l,i})}E_{i,l}(-s-1)E_{j,i}(-t)E_{l,j}(s+t+1)\nonumber\\
&\quad+\sum_{s,t \geq 0}\sum_{l=j+1}^{m+n}{(-1)}^{p(E_{l,i})p(E_{l,j})}E_{j,l}(-s-t-1)E_{i,j}(t)E_{l,i}(s+1).\label{bij5}
\end{align}
By direct computation, we also obtain
\begin{align}
&\phantom{{}={}}[B_i,A_j]\nonumber\\
&=\sum_{s,t \geq 0}\sum_{l=1}^{j}{(-1)}^{p(j)+p(l)}E_{i,l}(-s-t-1) E_{l,j}(t)E_{j,i}(s+1)\nonumber\\
&\quad-\sum_{s,t \geq 0}\sum_{k=i+1}^{m+n}{(-1)}^{p(k)+p(i)+p(E_{i,k})p(E_{j,i})}E_{j,k}(-s-t-1) E_{i,j}(t)E_{k,i}(s+1)\nonumber\\
&\quad+\sum_{s,t \geq 0}\sum_{k=i+1}^{m+n}\sum_{l=1}^{j}\delta_{k,l}{(-1)}^{p(E_{i,k})p(E_{j,l})}E_{j,l}(-s-t-1) E_{i,j}(t)E_{k,i}(s+1)\nonumber\\
&\quad+\sum_{s,t \geq 0}\sum_{k=i+1}^{m+n}\sum_{l=1}^{j}\delta_{k,l}{(-1)}^{p(E_{i,k})p(E_{j,l})}E_{j,l}(-s) E_{i,j}(-t-1)E_{k,i}(s+t+1)\nonumber\\
&\quad-\sum_{s,t \geq 0}\sum_{k=i+1}^{m+n}\sum_{l=1}^{j}\delta_{k,l}{(-1)}^{p(E_{k,i})p(E_{j,l})}E_{i,k}(-s-1)E_{j,i}(-t) E_{l,j}(s+t+1)\nonumber\\
&\quad-\sum_{s,t \geq 0}\sum_{k=i+1}^{m+n}\sum_{l=1}^{j}\delta_{k,l}{(-1)}^{p(E_{k,i})p(E_{j,l})}E_{i,k}(-s-t-1)E_{j,i}(s+1) E_{l,j}(t)\nonumber\\
&\quad+\sum_{s,t \geq 0}\sum_{k=i+1}^{m+n}{(-1)}^{p(k)+p(i)+p(E_{k,i})p(E_{i,j})}E_{i,k}(-s-1)E_{j,i}(-t) E_{k,j}(s+t+1)\nonumber\\
&\quad-\sum_{s,t \geq 0}\sum_{l=1}^{j}{(-1)}^{p(j)+p(l)}E_{i,j}(-s-1)E_{j,l}(-t) E_{l,i}(s+t+1).\label{aibj}
\end{align}
Let us simplify the right hand side of \eqref{aibj}. We prepare the following four relations by direct computation;
\begin{align}
&\quad\text{the second term of \eqref{aibj}}+\text{the third term of \eqref{aibj}}\nonumber\\
&=-\sum_{s,t \geq 0}\sum_{k=i+1}^{m+n}{(-1)}^{p(k)+p(i)+p(E_{i,k})p(E_{j,i})}E_{j,k}(-s-t-1) E_{i,j}(t)E_{k,i}(s+1)\nonumber\\
&\quad+\sum_{s,t \geq 0}\sum_{k=i+1}^{m+n}\sum_{l=1}^{j}\delta_{k,l}{(-1)}^{p(E_{i,k})p(E_{j,k})}E_{j,k}(-s-t-1) E_{i,j}(t)E_{k,i}(s+1)\nonumber\\
&=-\sum_{s,t \geq 0}\sum_{k=j+1}^{m+n}{(-1)}^{p(E_{i,k})p(E_{j,i})}E_{j,k}(-s-t-1) E_{i,j}(t)E_{k,i}(s+1),\label{aibj2}\\
&\quad\text{the first term of \eqref{aibj}}+\text{the 6-th term of \eqref{aibj}}\nonumber\\
&=\sum_{s,t \geq 0}\sum_{l=1}^{j}{(-1)}^{p(j)+p(l)}E_{i,l}(-s-t-1) E_{l,j}(t)E_{j,i}(s+1)\nonumber\\
&\quad-\sum_{s,t \geq 0}\sum_{l=1}^j\sum_{k=i+1}^{m+n}\delta_{k,l}{(-1)}^{p(E_{k,i})p(E_{j,k})}E_{i,k}(-s-t-1)E_{j,i}(s+1) E_{k,j}(t)\nonumber\\
&=\sum_{s,t \geq 0}\sum_{k=1}^{i}{(-1)}^{p(E_{k,i})p(E_{j,k})}E_{i,k}(-s-t-1)E_{j,i}(s+1) E_{k,j}(t)\nonumber\\
&\quad+\sum_{s,t \geq 0}\sum_{l=1}^{j}{(-1)}^{p(j)+p(l)}E_{i,l}(-s-t-1)[E_{l,j}(t),E_{j,i}(s+1)]\nonumber\\
&=\sum_{s,t \geq 0}\sum_{k=1}^{i}{(-1)}^{p(E_{k,i})p(E_{j,k})}E_{i,k}(-s-t-1)E_{j,i}(s+1) E_{k,j}(t)\nonumber\\
&\quad+\sum_{s,t \geq 0}\sum_{l=1}^{j}{(-1)}^{p(j)+p(l)}E_{i,l}(-s-t-1)E_{l,i}(s+t+1)\nonumber\\
&\quad-\sum_{s,t \geq 0}E_{i,i}(-s-t-1)E_{j,j}(s+t+1),\label{aibj2.1}\\
&\quad\text{the 4-th term of \eqref{aibj}}+\text{the 8-th term of \eqref{aibj}}\nonumber\\
&=\sum_{s,t \geq 0}\sum_{k=i+1}^{m+n}\sum_{l=1}^j\delta_{k,l}{(-1)}^{p(E_{i,k})p(E_{j,k})}E_{j,k}(-s) E_{i,j}(-t-1)E_{k,i}(s+t+1)\nonumber\\
&\quad-\sum_{s,t \geq 0}\sum_{l=1}^{j}{(-1)}^{p(k)+p(l)}E_{i,j}(-s-1)E_{j,l}(-t) E_{l,i}(s+t+1)\nonumber\\
&=-\sum_{s,t \geq 0}\sum_{l=1}^{i}{(-1)}^{p(E_{i,l})p(E_{j,l})}E_{j,l}(-s) E_{i,j}(-t-1)E_{l,i}(s+t+1)\nonumber\\
&\quad-\sum_{s,t \geq 0}\sum_{l=1}^{j}{(-1)}^{p(j)+p(l)}[E_{i,j}(-s-1),E_{j,l}(-t)] E_{l,i}(s+t+1)\nonumber\\
&=-\sum_{s,t \geq 0}\sum_{l=1}^{i}{(-1)}^{p(E_{i,l})p(E_{j,l})}E_{j,l}(-s) E_{i,j}(-t-1)E_{l,i}(s+t+1)\nonumber\\
&\quad-\sum_{s,t \geq 0}\sum_{l=1}^{j}{(-1)}^{p(j)+p(l)}E_{i,l}(-s-t-1)E_{l,i}(s+t+1)\nonumber\\
&\quad+\sum_{s,t \geq 0}E_{j,j}(-s-t-1)E_{i,i}(s+t+1),\label{aibj3}\\
&\quad\text{the 5-th term of \eqref{aibj}}+\text{the 7-th term of \eqref{aibj}}\nonumber\\
&=-\sum_{s,t \geq 0}\sum_{k=i+1}^{m+n}\sum_{l=1}^{j}\delta_{k,l}{(-1)}^{p(E_{k,i})p(E_{j,k})}E_{i,k}(-s-1)E_{j,i}(-t) E_{k,j}(s+t+1)\nonumber\\
&\quad+\sum_{s,t \geq 0}\sum_{k=i+1}^{m+n}{(-1)}^{p(k)+p(i)+p(E_{k,i})p(E_{i,j})}E_{i,k}(-s-1)E_{j,i}(-t) E_{k,j}(s+t+1)\nonumber\\
&=\sum_{s,t \geq 0}\sum_{k=j+1}^{m+n}{(-1)}^{p(E_{k,i})p(E_{j,k})}E_{i,k}(-s-1)E_{j,i}(-t) E_{k,j}(s+t+1).\label{aibj4}
\end{align}
Thus, by \eqref{aibj2}-\eqref{aibj4}, we have
\begin{align}
&\phantom{{}={}}[B_i,A_j]\nonumber\\
&=-\sum_{s,t \geq 0}\sum_{k=j+1}^{m+n}{(-1)}^{p(E_{i,k})p(E_{j,i})}E_{j,k}(-s-t-1) E_{i,j}(t)E_{k,i}(s+1)\nonumber\\
&\quad+\sum_{s,t \geq 0}\sum_{k=1}^{i}{(-1)}^{p(E_{k,i})p(E_{j,k})}E_{i,k}(-s-t-1)E_{j,i}(s+1) E_{k,j}(t)\nonumber\\
&\quad-\sum_{s,t \geq 0}\sum_{l=1}^{i}{(-1)}^{p(E_{i,l})p(E_{j,l})}E_{j,l}(-s) E_{i,j}(-t-1)E_{l,i}(s+t+1)\nonumber\\
&\quad+\sum_{s,t \geq 0}\sum_{k=j+1}^{m+n}{(-1)}^{p(E_{k,i})p(E_{j,k})}E_{i,k}(-s-1)E_{j,i}(-t) E_{k,j}(s+t+1)\nonumber\\
&\quad-\sum_{s\geq0}(sE_{i,i}(-s)E_{j,j}(s)-sE_{j,j}(-s)E_{i,i}(s)).\label{aibj5}
\end{align}
Adding \eqref{aij4}, \eqref{bij5}, and \eqref{aibj5}, we obtain $[A_i, A_j]+[B_i,A_j]+[B_i,B_j]=0$.
\end{proof} 
This completes the proof of Lemma~\ref{Lemma1}.
\bibliographystyle{plain}
\bibliography{affmod2}
\end{document}